\documentclass[11pt,a4paper]{amsart}
\usepackage[all]{xy}
\usepackage{amssymb,mathrsfs,amscd,graphicx,array, multirow,longtable,
enumerate, amsfonts, euscript,hyperref, geometry}
\long\def\forget#1{}

\usepackage{color}
\newcounter{commentcounter}

\normalsize

\theoremstyle{plain}
\newtheorem{thm}{Theorem}[section]
\newtheorem{lemma}[thm]{Lemma}
\newtheorem{prop}[thm]{Proposition}
\newtheorem{cor}[thm]{Corollary}
\newtheorem{conj}[thm]{Conjecture}

\theoremstyle{definition}
\newtheorem{defn}[thm]{Definition}
\newtheorem{eg}[thm]{Example}

\theoremstyle{remark}
\newtheorem{remark}[thm]{Remark}

\newcommand{\nc}{\newcommand}

\newcommand{\scrM}{\mathscr{M}}

\newcommand{\scrS}{\mathscr{S}}

\def\makeop#1{\expandafter\def\csname#1\endcsname
  {\mathop{\rm #1}\nolimits}\ignorespaces}
\makeop{Hom}   \makeop{End}   \makeop{Aut}   \makeop{Isom}  \makeop{Pic} 
\makeop{Gal}   \makeop{ord}   \makeop{Char}  \makeop{Div}   \makeop{Lie} 
\makeop{PGL}   \makeop{Corr}  \makeop{PSL}   \makeop{sgn}   \makeop{Spf}
\makeop{Spec}  \makeop{Tr}    \makeop{Nr}    \makeop{Fr}    \makeop{disc}
\makeop{Proj}  \makeop{supp}  \makeop{ker}   \makeop{im}    \makeop{dom}
\makeop{coker} \makeop{Stab}  \makeop{SO}    \makeop{SL}    \makeop{SL}
\makeop{Cl}    \makeop{cond}  \makeop{Br}    \makeop{inv}   \makeop{rank}
\makeop{id}    \makeop{Fil}   \makeop{Frac}  \makeop{GL}    \makeop{SU}
\makeop{Nrd}   \makeop{Sp}    \makeop{Tr}    \makeop{Trd}   \makeop{diag}
\makeop{Res}   \makeop{ind}   \makeop{depth} \makeop{Tr}    \makeop{st}
\makeop{Ad}    \makeop{Int}   \makeop{tr}    \makeop{Sym}   \makeop{can}
\makeop{length}\makeop{SO}    \makeop{torsion} \makeop{GSp} \makeop{Ker}
\makeop{Adm}   \makeop{Mat}   \makeop{Ig}
\def\makebb#1{\expandafter\def
  \csname bb#1\endcsname{{\mathbb{#1}}}\ignorespaces}
\def\makebf#1{\expandafter\def\csname bf#1\endcsname{{\bf
      #1}}\ignorespaces} 
\def\makegr#1{\expandafter\def
  \csname gr#1\endcsname{{\mathfrak{#1}}}\ignorespaces}
\def\makeescr#1{\expandafter\def
  \csname escr#1\endcsname{{\EuScript{#1}}}\ignorespaces}
\def\makecal#1{\expandafter\def\csname cal#1\endcsname{{\mathcal
      #1}}\ignorespaces} 

\def\doLetters#1{#1A #1B #1C #1D #1E #1F #1G #1H #1I #1J #1K #1L #1M
                 #1N #1O #1P #1Q #1R #1S #1T #1U #1V #1W #1X #1Y #1Z}
\def\doletters#1{#1a #1b #1c #1d #1e #1f #1g #1h #1i #1j #1k #1l #1m
                 #1n #1o #1p #1q #1r #1s #1t #1u #1v #1w #1x #1y #1z}
\doLetters\makebb   \doLetters\makecal  \doLetters\makebf

\doLetters\makeescr 

\doletters\makebf   \doLetters\makegr   \doletters\makegr

     \def\qed{\qedmark\medbreak}%
\def\qedmark{{\enspace\vrule height 6pt width 5pt depth 1.5pt}}%
    \def\setminus{\smallsetminus}

\def\sep{{\rm sep}}
\makeop{Bl}

\def\Spec{{\rm Spec}\,}

\def\Fpbar{\overline{\bbF}_p}
\def\Fp{{\bbF}_p}
\def\Fq{{\bbF}_q}

\def\normal{\vartriangleleft}

\newcommand{\Z}{\mathbb Z}
\newcommand{\Q}{\mathbb Q}

\newcommand{\C}{\mathbb C}
\newcommand{\A}{\mathbb A}    
\renewcommand{\O}{\mathcal O} 
\newcommand{\F}{\mathbb F}

\newcommand{\npr}{\noindent }

\newcommand{\isoto}{\stackrel{\sim}{\longrightarrow}}
\nc{\embed}{\hookrightarrow}

\newcommand{\fdot}{{\,{\scriptscriptstyle\bullet}\,}}

\newcommand{\ch}{characteristic }
\newcommand{\ac}{algebraically closed }
\newcommand{\dieu}{Dieudonn\'{e} }

\nc{\ol}{\overline}
\nc{\wt}{\widetilde}
\nc{\opp}{\mathrm{opp}}

\def\onto{\twoheadrightarrow}
\newcommand{\longonto}{\hbox{$\kern3pt\longrightarrow\kern-15pt\to\kern5pt$}}

\def\wh{\widehat}

\makeop{Ram}
\makeop{Rep}

\def\Achar{\text{A-char\,}}
\def\vol{{\rm vol}}

\newcommand{\indlimw}[1]{\lim_{\underset{#1}{\longrightarrow}}}

\newcommand{\projlimw}[1]{\lim_{\underset{#1}{\longleftarrow}}}

\usepackage{ifthen}
\newcommand{\dirlim}[1][]{\ifthenelse{\equal{#1}{}}
{\displaystyle \lim_{\longrightarrow}}
{\displaystyle \lim_{\underset{#1}{\longrightarrow}}}
}

\newcommand{\olE}{{\,\overline{\!E}}}
\newcommand{\olL}{{\,\overline{\!L}}}
\newcommand{\olR}{{\,\overline{\!R}}}
\newcommand{\olS}{{\,\overline{\!S}}}

\DeclareMathOperator{\cInd}{ind}
\newcommand{\BOne} {{\mathchoice{\hbox{\rm1\kern-2.7pt l\kern.9pt}}
                              {\hbox{\rm1\kern-2.7pt l\kern.9pt}}
                              {\hbox{\scriptsize\rm1\kern-2.3pt l\kern.4pt}}
                              {\hbox{\scriptsize\rm1\kern-2.4pt l\kern.5pt}}}}
\newcommand{\UOne}{\underline{\BOne}}

\renewcommand{\phi}{\varphi}

\begin{document}
\renewcommand{\thefootnote}{\fnsymbol{footnote}}
\setcounter{footnote}{-1}
\numberwithin{equation}{section}


\title[Satake compactifications and Drinfeld modular forms]
  {Arithmetic Satake compactifications and algebraic Drinfeld
  modular forms}


\author{Urs Hartl}
\address{(Hartl) Universit\"at M\"unster, Mathematisches Institut,
  Einsteinstr. 62, D--48149 Germany}
\email{https:/\!/www.uni-muenster.de/Arithm/hartl/}

\author{Chia-Fu Yu}
\address{
(Yu) Institute of Mathematics, Academia Sinica and NCTS\\
Astronomy Mathematics Building \\
No.~1, Roosevelt Rd. Sec.~4 \\ 
Taipei, Taiwan, 10617} 
\email{chiafu@math.sinica.edu.tw}



\date{\today}

\begin{abstract}
In this article we construct the arithmetic Satake compactification of
the Drinfeld moduli schemes of arbitrary rank 
over the ring of integers of any global function field 
away from the level structure, 
and show that the universal family extends uniquely to a
generalized Drinfeld module over the compactification.
Using these and functorial
properties, we define algebraic Drinfeld
modular forms over more general bases and the
action of the (prime-to-residue characteristic and level) Hecke algebra. 
The construction also furnishes many algebraic Drinfeld
modular forms obtained from the coefficients of the universal family 
which are also Hecke eigenforms. 
Among them we obtain generalized Hasse invariants 
which are already defined on the arithmetic Satake compactification
and not only its special fiber. We use
these generalized Hasse invariants to study the geometry 
of the special fiber.
We conjecture that our Satake compactification is Cohen-Macaulay.
If this is the case, we establish the Jacquet-Langlands
correspondence (mod $v$) between Hecke eigensystems of rank $r$ 
Drinfeld modular forms and those of algebraic modular 
forms (in the sense of Gross) 
attached to a compact inner form of $\GL_r$.
\end{abstract} 

\maketitle


\section{Introduction}
\label{sec:1}
Drinfeld modular curves and Drinfeld modular forms of rank $2$ 
are the function field analogues of elliptic modular curves and 
modular forms and have been intensively studied. Drinfeld modular varieties 
of higher rank $r$ are the function field $\GL_r$-analogue of
Shimura varieties. They have more structure and are even more
interesting.     
Analytic and algebraic Drinfeld modular forms of higher rank with
values in $\C_\infty$ (see below) were recently introduced and studied by
Basson, Breuer and Pink~\cite{BBP} using the Satake
compactification of the Drinfeld modular varieties constructed by Pink 
and Schieder \cite{pink, pink-schieder}. In a series of works~
\cite{gekeler:I,gekeler:III,gekeler:V,gekeler:II,gekeler:IV,gekeler:VI,
gekeler:VII}, Gekeler investigated the analytic aspect of Drinfeld
modular forms of higher rank by a different approach.

Our goal in this article is to define and construct Drinfeld modular
forms of higher rank over more general bases, for example, the ring of integers and
its reduction modulo a power of a prime ideal. This provides a
framework for studying the arithmetic aspect of Drinfeld modular forms. 
For example, 
one can look for the congruences between two Drinfeld modular 
forms of different weights or ranks (the latter for example through morphisms between  Drinfeld modular varieties as in Lemma~\ref{S.7} or through restriction to the boundary of the Satake compactification),
or study the congruences between Drinfeld Hecke eigenforms and
related Galois representations.
Investigating such congruences has proved to be very useful 
in the construction of Galois representations of Hecke eigenforms 
as in the famous article of Deligne and Serre \cite{deligne-serre}
for modular forms of weight $1$, and those of Wiles \cite{wiles:invent88} 
and of Taylor~\cite{taylor:hmf89, taylor:smf91,taylor:hmfII95} for Hilbert
modular forms and Siegel modular forms of degree 2 with lower weight. 
One can also explore
the analogous theory for $p$-adic modular forms following Katz
\cite{katz:p-adic350}. 
These $\wp$-adic Drinfeld modular forms were studied recently by Hattori \cite{hattori:dual} for rank $2$ and by \cite{nicole-rosso} and \cite{Greve+Hartl} for arbitrary rank.

Before we explain our results, let us explain the
(old) strategy of defining elliptic modular forms over rings 
of integers using algebraic geometry. 
Suppose $\ol M_n$ is the projective elliptic modular curve over $\Q(\zeta_n)$ 
of level-$n$ structure with $n\ge 3$, where $\zeta_n$ is a primitive
$n$-th root of unity. One first realizes 
the elliptic modular forms of level $n$ and weight $k$, which
are a priori defined
analytically, as the elements of $H^0(\ol M_n\otimes\C,\omega^{\otimes
k}\otimes\C)$ for a suitable invertible sheaf $\omega$ on $\ol M_n$, which is the dual of the Lie algebra of the universal elliptic curve. 
Then one constructs an integral model $\ol \bfM_n$ of $\ol M_n$ over 
$\Z[1/n,\zeta_n]$ and extends $\omega$ canonically over $\ol \bfM_n$. 
Here $\ol \bfM_n$ is the (minimal) Satake compactification of the moduli space $\bfM_n$ of elliptic curves with level-$n$ structure over $\Z[1/n,\zeta_n]$ obtained by adding generalized elliptic curves in \cite[p.~72f]{Deligne-Rapoport}. 
Algebraic modular forms of level $n$ and weight $k$ over 
a $\Z[1/n,\zeta_n]$-algebra $L$ 
then are defined as elements of $H^0(\ol \bfM_n\otimes L, \omega^{\otimes
  k}\otimes L)$. 

This strategy has been worked out for 
Siegel and Hilbert 
moduli schemes by Chai, Faltings and Rapoport (see \cite{chai:amchb,
  faltings-chai, rapoport:thesis}). 
Historically, the Satake
compactifications of complex Siegel modular varieties 
were constructed by Satake first 
analytically. 
Then Ash, Mumford, Rapoport and Tai \cite{AMRT} 
constructed complex smooth toroidal compactifications of locally
symmetric varieties. 
However, the order in the construction of the arithmetic version is
reverse: the arithmetic toroidal compactifications were
constructed first and were used to construct the arithmetic
Satake (= minimal) compactification. 
Chai and Faltings showed that the ample invertible sheaf $\omega$ admits
a canonical extension 
over a smooth toroidal compactification and used it to
define Siegel modular forms over a ring of integers. The arithmetic Satake
compactification then is constructed to be the Proj of the graded ring of
arithmetic Siegel modular forms. 
     
To explain the results of our article, we describe some background of compactifications of Drinfeld modular
varieties.  
Let $F$ be a global function field with finite constant 
field $\Fq$ with $q$ elements. 
Let $\infty$ be a fixed place of $F$, and $A$ the ring of
$\infty$-integers of $F$. Denote by $\C_\infty:=\wh {\ol{F}}_\infty$ 
the completion of an algebraic closure of the completion $F_\infty$ of
$F$ at $\infty$. 
The compactification of Drinfeld moduli schemes of rank $2$ over $A$
was constructed by Drinfeld; see \cite[Proposition 9.3]{drinfeld:1}.   
Gekeler \cite{gekeler:satake} gave an outline of the Satake 
compactification for higher rank over $\C_\infty$.  
The Satake compactification
over $F$ for arbitrary rank was constructed by Kapranov~\cite{kapranov}
for $F=\Fq(t)$ and by Pink~\cite{pink} for arbitrary $F$. 
Pink's method is rather different from previous ones. 
He showed that the old strategy works fine with Drinfeld modular
varieties, namely, the universal family extends to
a generalized Drinfeld $A$-module over the Satake compactification. This 
paves a way to define algebraic Drinfeld modular forms of arbitrary
rank over $F$. 
Arithmetic compactifications for rank $2$ were revisited by
Lehmkuhl~\cite{lehmkuhl} in more details and also by Hattori \cite{hattori:comp}. H\"aberli~\cite{haberli} gives an analytic construction of the Satake compactification of Drinfeld modular varieties of arbitrary rank over $\C_\infty$. He also shows the agreement of the Satake compactification by the analytic construction and Pink's Satake compactification by the algebraic method. In particular, the universal Drinfeld $A$-module extends to a generalized Drinfeld $A$-module over the analytic Satake compactification constructed by Kapranov for $A=\Fq[t]$ and by H\"aberli for an arbitrary global function field $F$. 
This answers a question of Pink \cite[Remark 4.9]{pink}.

In this article we
construct the arithmetic Satake compactification of Drinfeld moduli
schemes over the localization $A_{(v)}$ of a prime $v\in \Spec A$. 
Let $G=\GL_r$ with $r\ge 1$. Denote by $A_v$ the completion of $A$ at $v$, 
by $\wh A$ the profinite completion of $A$, 
and by $\A^\infty:=\wh{A}\otimes_A F$ the finite adele ring of $F$.

\begin{thm}\label{thm:A}
For every fine open compact subgroup $K=K_v K^v\subset G(\A^\infty)$, where $K_v=G(A_v)$ and $K^v\subset G(\A^{v\infty})$, the Drinfeld moduli scheme 
$\bfM^r_K$ over $A_{(v)}$ of rank $r$ and level $K$ possesses an arithmetic Satake compactification 
  $\ol{\bfM}^r_K$ projective flat over $A_{(v)}$. The Satake compactification and
  its universal family are unique up to unique isomorphism. 
  The dual $\omega_K:=\Lie(\olE_K)^\vee$ of 
  the relative Lie algebra of the universal generalized Drinfeld $A$-module $\olE_K$
  over $\ol{\bfM}^r_K$ is ample. Moreover, the Satake compactification is compatible with the transition maps of changing $K$ and the prime-to-$v$ Hecke correspondences, and the universal family $\olE_K$ and $\omega_K$ satisfy the functorial property with respect to the transition maps and prime-to-$v$ Hecke correspondences.
\end{thm}
We prove this in Theorem~\ref{S.2} even in the classical form, and in the discussions after Definition~\ref{def:MF.10} in the adelic form.
We follow the
approach of Pink \cite{pink}.
Thus, the generic fiber of our compactification constructed here gives
Pink's Satake compactification. We expect the following to hold.

\begin{conj}\label{ConjCM}
The Satake compactification $\ol{\bfM}^r_K$ is Cohen-Macaulay.
\end{conj}

In the special case where $A=\Fq[t]$ and the level is $K=K(t^n)$ for $n\in\Z_{>1}$ the conjecture was proved by Pink and Schieder~\cite[Section 10]{pink-schieder} for $n=1$ and by Pink~\cite{PinkDetMorph} for arbitrary $n$. In Remark~\ref{RemCM} we give more evidence for Conjecture~\ref{ConjCM}.

Using Theorem~\ref{thm:A}, we define for any positive integer $k$ and
any $A_{(v)}$-algebra $L$ 
\[ M_k(r, K, L):=H^0(\ol{\bfM}^r_K\otimes_{A_{(v)}} L, \omega_K^{\otimes k}\otimes L) \]
the $L$-module of algebraic Drinfeld modular forms of rank $r$, weight $k$ and level $K$.
Thanks to work of Basson, Breuer and
Pink \cite{BBP} on the comparison theorem of 
analytically and algebraically defined Drinfeld modular forms, 
we extend the notion of Drinfeld modular forms over the ring of 
integers away from the level.

In contrast to the case of Siegel moduli schemes, our results may be 
surprising due to the following reasons:
\begin{itemize}
\item[(i)] The construction of arithmetic Satake compactifications for
  Drinfeld modular varieties does not rely on that of smooth 
  arithmetic toroidal
  compactifications. Indeed, this is a big advantage (due to Pink's idea) 
  because the smooth arithmetic compactifications for 
  Drinfeld modular varieties have not yet been completely
  constructed. Some special cases of smooth arithmetic compactifications of
  arbitrary rank where $A=\Fq[t]$ and the level is $K=K(t)$  
  were constructed by Pink and Schieder
  \cite[Section 10]{pink-schieder}, and more general level $K=K(\grn)$ were constructed very recently by Fukaya, Kato and Sharifi \cite{fukaya-kato-sharifi}.
\item [(ii)] The universal family extends to a 
  generalized Drinfeld $A$-module over the arithmetic Satake
  compactification. It is not known that the universal family over
  a Siegel moduli scheme extends to its arithmetic Satake
  compactification. 
\end{itemize}

An advantage in the function field case is that we
can construct coefficient modular forms and generalized Hasse invariants on the special fiber over a prime $v\in\Spec A$ of the arithmetic Satake
compactification. And we can lift these generalized Hasse invariants to the
integral model over $A_{(v)}$ (which is still in positive characteristic). The analogous problem is still unsolved for Shimura varieties although significant progress was made by Boxer~\cite{boxer} and Goldring and Koskivirta~\cite{goldring-koskivirta:19}.
This allows us to
prove our main applications (see Lemma~\ref{HI.3} and Theorem~\ref{HI.5}): 

Let $\ol{\bfM}^r_K$ be the Satake compactification as in
Theorem~\ref{thm:A} and let
$\ol{\scrM}^r_K:=\ol{\bfM}^r_K\otimes_{A_{(v)}} \F_v$ be its special fiber over $v$. 
Let $\grp \subset A$ be the prime ideal corresponding to the place $v$, and let $a\in \grp$ with $v(a)=1$. Let
$H^a_0,\ldots,H^a_{r-1}$ be the generalized Hasse invariants (see Definition~\ref{HI.2}). They depend on $a$. We study the (locally) closed subschemes defined by the Hasse invariants inside the special fiber over $v$ of both Drinfeld moduli schemes and their Satake compactifications. 
These vanishing loci provide the stratification according to the height of the universal generalized Drinfeld $A$-module on the Satake compactification, or equivalently according to the $v$-rank of the Drinfeld $A$-module on the Drinfeld moduli scheme; see Section~\ref{sec:SS.1} and Lemma~\ref{HI.3}(2). The closed stratum is the supersingular locus. We prove the following results on the geometry of these (locally) closed subschemes.

\begin{thm}\label{thm:B1}
For $1\le h\le r$, let $(\ol \scrM^r_{K})^{\ge
  h}$ and $(\scrM^r_{K})^{\ge h}$ be the vanishing loci of
$H^a_0,\ldots,H^a_{h-1}$ in $\ol\scrM^r_K$ and in  
$\scrM^r_{K}:={\bfM}^r_K\otimes_{A_{(v)}} \F_v$, respectively. Set 
\[ \text{$(\ol
\scrM^r_{K})^{(h)}:=(\ol \scrM^r_{K})^{\ge h}\setminus(\ol \scrM^r_{K})^{\ge 
  h+1}$\quad and\quad $(\scrM^r_{K})^{(h)}:=(\scrM^r_{K})^{\ge h}\setminus(\scrM^r_{K})^{\ge
  h+1}$.} \]   
\begin{itemize}
\item[\rm (1)] The subschemes $(\ol \scrM^r_{K})^{\ge h}$ and $(\scrM^r_{K})^{\ge h}$ are independent of $a$ (satisfying $v(a)=1$).
\item[\rm (2)] The subschemes 
  $(\ol \scrM^r_{K})^{\ge h}$ and $(\ol \scrM^r_{K})^{(h)}$ are  
  of pure dimension $r-h$. Moreover, $(\scrM^r_{K})^{(h)}$ is
   Zariski dense in any one of the schemes $(\ol \scrM^r_{K})^{\ge h}$,  $(\scrM^r_{K})^{\ge h}$ and $(\ol \scrM^r_{K})^{(h)}$.
\item[\rm (3)] The subschemes $(\ol \scrM^r_{K})^{(h)}$ and
$(\scrM^r_{K})^{(h)}$ are affine.
\item[\rm (4)] For any $h< r$, every (geometric) irreducible component of $(\ol \scrM^r_K)^{\ge h}$ contains a (geometric) irreducible component of $(\ol \scrM^r_K)^{\ge h+1}$. Likewise, every (geometric) irreducible component of $(\scrM^r_K)^{\ge h}$ contains a (geometric) irreducible component of $(\scrM^r_K)^{\ge h+1}$.
\item[\rm (5)] Let $F_{\det K}$ be the class field of $F$
  corresponding to the open subgroup 
  $F^\times \det K \subset (\A^\infty)^\times$ by class field theory. The moduli space $\ol\scrM^r_{K}=(\ol\scrM^r_{K})^{\ge 1}$ has
\[ \bigl|(\A^\infty)^\times/(F^\times\cdot \det K)\bigr|\] geometric
connected components, and   
\[ \bigl|(\A^\infty)^\times/(F^\times\cdot \det
K)\bigr|\,\big/\,f_v \] connected components, 
where $f_v$ is the order of the Frobenius element 
$(\grp, F_{\det K}/F)$ in $\Gal(F_{\det K}/F)$. 
If $K=K(\grn)$ for a nonzero proper ideal $\grn\subset A$, then we have
\[ \bigl|(\A^\infty)^\times/(F^\times\cdot \det
K)\bigr|\,\big/\,f_v \;=\; h(A) \cdot|(A/\grn)^\times| \,\big/\,\bigl((q-1)f_1f_2\bigr), \] where $h(A)$ is the class number of $A$, $f_1$ is the smallest
positive integer such that $\grp^{f_1}=(b)$ is a principal ideal, and
$f_2$ is the order of the image of $b$ in
$(A/\grn)^\times/{\Fq}^\times$. 
\end{itemize}
\end{thm}

By Lemma~\ref{HI.3}(2), a point $x$ of $\ol \scrM^r_{K}$ lies in  $(\ol\scrM^r_{K})^{(h)}$ if and only if the Drinfeld $A$-module $\ol \varphi_{K,x}$ over the point $x$ has height $h$. Recall our Conjecture~\ref{ConjCM}. If the conjecture holds true, we can determine the (geometric) connected components of the subschemes $(\ol\scrM^r_{K})^{\ge h}$ for all $h\ne r-1$ and not just for $h=1$ as in Theorem~\ref{thm:B1}(5).

\begin{thm}\label{thm:B2}
Let the notation be as in Theorem~\ref{thm:B1}. Assume that the Satake compactification $\ol \bfM^r_K$ is Cohen-Macaulay. Let $h$ be an integer with $1\le h\le r$.
\begin{itemize}
\item[\rm (1)] The sequence of Hasse invariants $(H^a_0,\ldots,H^a_{r-1})$ is a regular sequence on $\ol\bfM^r_K$.
\item[\rm (2)] The closed subscheme $X_h:=V(H^a_1,\ldots,H^a_h)$ of $\ol\bfM^r_K$ is flat over $A_{(v)}$.
\item[\rm (3)] Each closed subscheme 
 $(\ol \scrM^r_{K})^{\ge h}$ is also Cohen-Macaulay.
\item[\rm (4)] If $h< r-1$, then the natural maps 
\[ \pi_0((\ol \scrM^r_K)^{\ge h+1})\to\pi_0((\ol \scrM^r_K)^{\ge h})\] and \[ \pi_0((\ol \scrM^r_K)^{\ge h+1}\otimes_{\F_v} \ol \F_v)\to\pi_0((\ol \scrM^r_K)^{\ge h}\otimes_{\F_v} \ol \F_v) \] 
of (geometric) connected components are bijective.
\item[\rm (5)] For any $h<r$ the moduli space $(\ol\scrM^r_{K})^{\ge h}$ has
\[ \bigl|(\A^\infty)^\times/(F^\times\cdot \det K)\bigr|\] geometric
connected components, and   
\[ \bigl|(\A^\infty)^\times/(F^\times\cdot \det
K)\bigr|\,\big/\,f_v \] connected components.
\end{itemize}
\end{thm}  

We also study the Hecke action on Drinfeld modular forms. 
Let $\calA(G',\ol\F_v)$ be the space of all locally constant functions \mbox{$f:G'(F)\backslash G'(\A)/G'(F_\infty)\to \ol\F_v$}, where $\A$ denotes the adeles of $F$ and $G'=D^\times$ is the group scheme of units in the central division algebra $D$ over $F$ ramified precisely at $\infty$ and $v$, with invariants $\inv_\infty(D)=-1/r$ and $\inv_v(D)=1/r$. Put $U(v):=
\ker (G'(A_v)=O_{D_v}^\times \to \F_{v^r}^\times)$, cf.~\eqref{eq:SS.8}. 
\begin{thm}(Theorem~\ref{HI.6})\label{thm:C}
Let $\grn\subset A$ be a 
prime to $v$ non-zero ideal and $K=K_v K^v=K(\grn):=\ker\bigl(G(\wh{A})\to G(A/\grn)\bigr)$. 
Consider the 
  sets of prime-to-$v\grn$ Hecke eigensystems $\calH^{\infty v\grn}_{\ol\F_v}\to\ol\F_v$ arising from
  \begin{enumerate}
  \item[{\rm (1)}] algebraic Drinfeld
  modular forms in $M_k(r, K_v, \ol \F_v)^{K^v}$ for all $k\ge 0$, where $\displaystyle M_k(r, K_v, \ol \F_v):=\indlimw{\wt K^v} M_k(r, K_v\wt K^v, \ol\F_v)$ and $M_k(r,K_v\wt K^v,\ol\F_v):=H^0(\ol \bfM^r_{K_v\wt K^v}\otimes_{A_{(v)}} \ol\F_v,\;\omega_{K_v\wt K^v}^{\otimes k}\otimes {\ol\F_v})$, and 
  \item[{\rm (2)}] elements of $\calA(G',\ol \F_v)^{U(v)K^v}$, respectively,
  \end{enumerate}
  where $\calH^{\infty v\grn}_{\ol \F_v}=\calH_{\ol \F_v}(G(\A^{\infty
  v\grn}),K^{v\grn})\simeq\calH_{\ol \F_v}(G'(\A^{\infty
  v\grn}),K^{v\grn})$ is 
  the prime-to-$v\grn$ spherical Hecke algebra over $\ol \F_v$.
If Conjecture~\ref{ConjCM} holds for $K=K_v\wt K^v$  for a cofinal system of compact open subgroups $\wt K^v\subset G(\A^{\infty v})$, then both sets of Hecke eigensystems are equal.
  In particular, there are then only finitely many Hecke eigensystems of
  algebraic Drinfeld modular forms over $\ol \F_v$ of a fixed level
  and all weights. 
\end{thm}

Theorem~\ref{thm:C} amounts to the Jacquet-Langlands
correspondence mod $v$ for Hecke eigensystems. As a corollary we also derive an explicit upper bound and the
asymptotic behavior of the size of these Hecke eigensystems. One technical difficulty is that we do not know whether $M_k(r, K_v, \ol \F_v)^{K^v}=M_k(r, K_vK^v, \ol \F_v)$. This would be true if the special fiber $\ol\scrM^r_{K_vK^v}$ at $v$ of the Satake compactification was normal; see Theorem~\ref{ThmSmoothHeckeAction}. However, we only know that $\ol\scrM^r_{K_vK^v}$ is reduced; see Proposition~\ref{reduced}. We prove in Theorem~\ref{ThmSmoothHeckeAction} that, nevertheless, $M_k(r, K_v, \ol \F_v)$ is a smooth admissible $G(\A^{v\infty})$-module.

This article is organized as follows.
After reviewing (generalized) Drinfeld modules and the construction of
Drinfeld modular varieties in Section~\ref{sec:2}, we construct their arithmetic
Satake compactification in Section~\ref{sec:S} following Pink's
approach~\cite{pink} over the function field. Algebraic Drinfeld modular forms are defined and
their preliminary properties are studied in Section~\ref{sec:MF}. The next
Section~\ref{sec:SS} deals with the supersingular locus and its relation with
algebraic modular forms in the sense of Gross. 
In the final Section~\ref{sec:HI} we introduce the generalized Hasse invariants and prove the two Theorems~\ref{HI.5} and \ref{HI.6} explained above.

\section{Drinfeld modules and moduli spaces}
\label{sec:2}

Let $q$ be a power of a prime number $p$, and $\calC$ 
a geometrically connected smooth projective algebraic curve
over a finite finite $\Fq$ of $q$ elements.  
Let $F$ be the function field of $\calC$, which is a 
global function field of \ch $p>0$ 
with field of constants $\Fq$. 
Fix a closed point $\infty$ of $\calC$, referred as the place of $F$ 
at infinity. 
Let $A:=\Gamma(\calC \setminus\{\infty\},\calO_\calC)$ 
be the ring of functions in $F$ regular away from
$\infty$. $A$ is a Dedekind domain with finite unit group $A^\times = 
\F_q^\times$. 
For any place $v$ of $F$, denote by $F_v$ the completion of
$F$ at $v$, $\O_v$ the valuation ring, $\F_v$ the residue field and $|\
|_v$ the normalized valuation of $F$ at $v$. If $v$ is a finite place,
we also write $A_v$ for $\O_v$.
For any nonzero element $a\in A$, define $\deg (a):=\dim_{\Fq} A/(a)$.

Denote by $\wh A$  the pro-finite completion of $A$, and 
$\bbA$ (resp.~$\bbA^\infty$) the (resp.~finite) adele ring of $F$.
Let $\ol F \subset  \ol F_\infty$ be fixed algebraic
closures of $F\subset F_\infty$, respectively. 
Let $\C_\infty$ be the completion of $\ol
F_\infty$ with respect to the unique extension of $|\ |_\infty$ .

Let $\tau$ denote the endomorphism $x\mapsto x^q$ of the additive
group $\bbG_{a,\Fp}$. For any field $L\supset \Fp$, denote by
$\End_{\Fq}(\bbG_{a,L})$ the ring of $\Fq$-linear endomorphisms 
of $\bbG_{a,L}=\bbG_{a}\otimes_{\Fp} L$ over $L$. 
It is known that
\[ \End_{\bbF_q}(\bbG_{a,L})=L\{\tau\}:=\left \{\sum_{i=0}^n \varphi_i \tau^i,
    \text{for some $n\in \bbN$},\ \varphi_i\in L\right \}, \quad \tau
\varphi_i=\varphi_i^q \tau. \] 
Denote by $\partial:L\{\tau\}\to L,\ \sum_i \varphi_i
\tau^i\mapsto \varphi_0$, the derivative map.

An \emph{$A$-field} is a field $L$ together with a ring homomorphism $\gamma:
A\to L$. We say that $(L,\gamma)$ is of \emph{generic $A$-characteristic} if 
$\Achar (L,\gamma):=\ker \gamma$ is zero, 
otherwise that $(L,\gamma)$ is of
\emph{$A$-characteristic} $v$, where $v$ is the place corresponding to the
non-zero prime ideal $\Achar (L,\gamma)$. 
If there is no confusion, we write $L$ for $(L,\gamma)$.  
More generally, for an $A$-scheme $S$ we let $\gamma\colon A\to\Gamma(S,\O_S)$ be the ring homomorphism induced from the structure morphism $S\to\Spec A$. We write $(S,\gamma)$ for such an $A$-scheme.

Recall that 
a \emph{Drinfeld $A$-module} over an $A$-field $L$ is a ring homomorphism 
\begin{equation}
  \label{eq:2.1}
  \varphi: A\to L\{\tau\}, \quad a\mapsto \varphi_a=\sum_i
\varphi_{a,i} \, \tau^i,
\end{equation}
such that $\partial \circ \varphi=\gamma$ and $\varphi$ does not factor
through the inclusion $L\subset L\{\tau\}$.  
There is a unique positive integer $r$ such that for any non-zero element 
$a\in A$, one has $\varphi_{a,i}=0$ for all $i>r
\deg(a)$ and $\varphi_{a,r\deg(a)}\neq 0$
\cite[Proposition~2.1]{drinfeld:1}. The integer $r$ is called the {\it rank} of
$\varphi$.

To introduce families of Drinfeld $A$-modules we mainly follow 
\cite{pink}, besides the standard references \cite{drinfeld:1} and
\cite{laumon:1}.   
By definition the {\it trivial line bundle} over a scheme $S$ 
is the
additive group scheme $\bbG_{a,S}$ over $S$ together with the multiplication
$\bbG_{m,S} \times_S \bbG_{a,S}\to \bbG_{a,S}, (x,y)\mapsto xy$.
An arbitrary {\it line bundle over $S$} 
is a commutative group scheme $E$ over 
$S$ together with a scalar multiplication $\bbG_{m,S}\times_S E\to E$ which as a pair, is Zariski locally on $S$ isomorphic to the trivial line
bundle. 
A \emph{homomorphism} of line bundles is a morphism of group schemes 
that is compatible with the $\bbG_m$-actions. 
By working on a local trivialization the following facts are easy to prove.
The sections of $E$ over an open subset $U\subset S$ form a module over 
$\O_S(U)$ such that the scalar multiplication with elements in 
$\O_S(U)^\times=\bbG_{m,S}(U)$ coincides with the $\bbG_{m,S}$-action.
On sections over an open 
$U\subset S$ a homomorphism of line bundles is automatically 
$\O_S(U)$-linear. Indeed, it is additive by definition and compatible 
with scalar multiplication by elements in $\O_S(U)^\times$. The 
compatibility with all $a\in\O_S(U)$ follows because at every 
point $s$ of $U$ either $a$ or $1+a$ is invertible in a neighborhood of $s$.
We let $\Hom_{\O_S}(E_1,E_2)$
denote the group of homomorphisms between the line bundles $E_1$ and $E_2$. We also let $\Hom(E_1,E_2)$ denote the group of homomorphisms of the commutative
group schemes underlying $E_1$ and $E_2$ forgetting the $\bbG_m$-actions. As usual we write $\End_{\O_S}(E):=\Hom_{\O_S}(E,E)$ and $\End(E):=\Hom(E,E)$ for the endomorphism rings. Note that the homomorphism 
\[
\tau\colon E\to E^{\otimes q}=\sigma^*E\,,\qquad x\mapsto x^q
\]
is only additive and not compatible with the $\bbG_m$-actions on $E$ and $E^{\otimes q}$. So it lies in $\Hom(E,E^{\otimes q})$ but not in $\Hom_{\O_S}(E,E^{\otimes q})$.\\

According to the original definition \cite[Section~5, p.~575]
{drinfeld:1}, a \emph{Drinfeld $A$-module of rank $r$} 
($r$ being a positive integer) 
over an $A$-scheme $(S,\gamma)$ is a pair $(E,\varphi)$,
where $E$ is a line bundle over $S$ and $\varphi:A\to \End(E)$ is a ring
homomorphism such that 
$\partial \circ \varphi=\gamma$, and for any $L$-valued
point $s:\Spec L\to S$ over $\Spec A$, 
where $L$ is an $A$-field, the pull-back 
$\varphi_s$ is a Drinfeld $A$-module of rank $r$ over $L$.
A \emph{homomorphism} of Drinfeld $A$-modules between 
$(E,\varphi)$ and $(E',\varphi')$ is
a homomorphism 
\begin{equation}\label{eq:2.15}
 u:E\to E'   
\end{equation}
of commutative group schemes over $S$ such that
$\varphi'_a \circ u=u \circ \varphi_a$ for all $a\in A$. Compatibility with the structure of line bundles, i.e. with the $\bbG_m$-actions, is not required. 
Let $\sigma:S\to S$ be 
the $q$-th power Frobenius map.
The endomorphism $\varphi_a$ for $a\in A$ can be expressed uniquely as
a locally finite sum \cite[Section~5]{drinfeld:1}   
\begin{equation}
  \label{eq:2.2}
  \varphi_a=\sum_{i\ge 0} \varphi_{a,i}\tau^i, \quad \varphi_{a,i}\in 
\Gamma(S,E^{\otimes(1-q^i)}), \quad \tau^i: E\to
E^{\otimes q^i}=\sigma^{i*}E,\ x\mapsto x^{q^i}.
\end{equation}
By this we mean that the expression in  (\ref{eq:2.2}) 
is a finite sum on any quasi-compact open subset.

A Drinfeld $A$-module $(E,\varphi)$ of rank $r$ 
over $S$ is called {\it standard} if for any
$a\in A$ and any $i>r \deg(a)$, 
the term $\varphi_{a,i}$ in (\ref{eq:2.2}) is zero.  
It is shown \cite[Proposition~5.2]{drinfeld:1} that 
every Drinfeld $A$-module is isomorphic to a 
standard Drinfeld $A$-module, and that every isomorphism of
standard Drinfeld $A$-modules is $\O_S$-linear,
i.e.~in $\Hom_{\calO_S}(E_1,E_2)$ as opposed to in $\Hom(E_1,E_2)$, see also \cite[Lemma~3.8]{hartl:isog}.   
 
In \cite{pink}, R.~Pink worked on the notion of generalized
Drinfeld modules. These modules 
play a similar role as what generalized elliptic
curves do for compactifying elliptic modular curves. 

\begin{defn}[{\cite[Section~3]{pink}}] \label{2.1}
(1) A {\it generalized Drinfeld $A$-module} over an $A$-scheme $S$ is a
  pair $(E,\varphi)$ consisting of a line bundle $E$ over $S$ and a ring
  homomorphism $A\to \End(E)$ satisfying the following properties
  \begin{enumerate}
  \item[(a)] The composition $\partial \circ \varphi$ is the structure
    morphism $\gamma:A\to \Gamma(S,\calO_S)$.
  \item[(b)] Over any point $s\in S$, the fiber $\varphi_s$ at $s$ is a
    Drinfeld $A$-module of rank $r_s\ge 1$. 
  \end{enumerate}

\noindent
(2) A generalized Drinfeld $A$-module $(E,\varphi)$ is said to 
    be {\it of rank $\le
    r$}, where $r$ is a positive integer, if 

  \begin{enumerate}
  \item[(c)] for any $a\in A$, the endomorphism
    $\varphi_a$ has the form
    $\sum_{i=0}^{r\deg(a)} \varphi_{a,i}
    \tau^i$ with sections $\varphi_{a,i}\in \Gamma(S, E^{\otimes(1-q^i)})$.
  \end{enumerate}

\noindent
(3) An isomorphism of generalized Drinfeld $A$-modules is an
    isomorphism of line bundles that commutes with the actions of $A$.
\end{defn}

As in \cite[after Definitin~3.1]{pink} we note that the property that $(E,\varphi)$ is of rank $\le r$ is 
preserved under isomorphisms of line bundles, but in general not 
under (non-linear) isomorphisms of the underlying group schemes.
The definition of generalized Drinfeld 
$A$-modules
$(E,\varphi)$ of rank
$\le r$ is slightly stronger than the notion of that with the property
$r_s\le r$ everywhere. 
Following from the definitions, 
these two notions are equivalent if the scheme $S$ is reduced.
Indeed, if $r_s\le r$ everywhere, then $\varphi_{a,i}$ is zero in the residue fields of all points of $S$, hence is locally nilpotent on $S$ for all $i>r\cdot\deg(a)$. 
But for general $S$ and a generalized Drinfeld $A$-module 
$(E,\varphi)$ of rank $\le r$ over $S$, 
the smallest integer $r_1$ such that $(E,\varphi)$ is of rank $\le r_1$
can be bigger than the maximal point-wise rank $r_2:= \max\{r_s|s\in S\}$. 
According to \cite{pink}, the
non-zero nilpotent components $\varphi_{a,i}$ for $i>r_2 \deg(a)$ 
should be regarded
as deformations ``towards higher rank''.  

\begin{defn}[{\cite[Section~3]{pink}}]\label{2.2}
  A generalized Drinfeld $A$-module of rank $\le r$ with $r_s=r$
  everywhere is called a \emph{Drinfeld $A$-module of rank $r$} over an
  $A$-scheme $S$. 
\end{defn}

This definition corresponds to that of a standard Drinfeld $A$-module
of rank $r$ in \cite{drinfeld:1}. By \cite[Proposition~5.2]{drinfeld:1} or
\cite[Lemma 1.1.2]{laumon:1} (with full details added by \cite[Proposition~3.4]{pink}) there is a natural bijection between 
\[ \{\text{Drinfeld $A$-modules of rank $r$ over $S$ in
  Definition~\ref{2.2}}\}/\simeq \quad \text{and} \] 
\[ 
\{\text{Drinfeld $A$-modules of rank $r$ over
$S$ in the original definition}\}/\simeq\,.
\]
Here the $\simeq$ in the second line means up to isomorphisms of 
group schemes with $A$-action in the 
sense of \eqref{eq:2.15} and not isomorphisms in the sense of 
Definition~\ref{2.1}. 
We shall adopt Definition~\ref{2.2} for Drinfeld
$A$-modules of constant rank in this article. In particular, every 
isomorphism between Drinfeld modules is $\O_S$-linear. 

Let $r$ be a positive integer, and $\grn\subset A$ a nonzero proper
ideal.  
Denote by $A[\grn^{-1}]\subset F$ the $A$-subalgebra generated by
elements of the fractional ideal $\grn^{-1}\subset F$. For each finite place
$v$, one easily calculates that $A[\grn^{-1}]\otimes_A A_v=A_v$ if
$v\notin V(\grn)$, and $A[\grn^{-1}]\otimes_A A_v=F_v$
otherwise. Therefore, $A[\grn^{-1}]$ agrees with the ring
$A[1/\grn]:=\Gamma(\Spec A-V(\grn), \calO_C)$.  
A \emph{(full) level-$\grn$ structure} on a 
Drinfeld $A$-module of rank $r$ over an 
$A[\grn^{-1}]$-scheme $S$ is an isomorphism of finite flat schemes of 
$A$-modules 
\begin{equation}
  \label{eq:level}
  \lambda: (\grn^{-1}/A)^r_S \isoto \varphi[\grn], \quad
\varphi[\grn]:=\bigcap_{a\in \grn} \ker (\varphi_a)\subset E,
\end{equation}
where $\bigcap$ is the scheme theoretic intersection. When $S$ is connected, $\lambda$ is simply given by an isomorphism
$(\grn^{-1}/A)^r \isoto \varphi[\grn](S)$ of finite $A$-modules. Let 
\[ K(\grn):=\ker (\GL_r(\wh A)\to \GL_r(A/\grn))\subset 
\GL_r(\wh A) \]
denote the principal open compact subgroup of level $\grn$.

Let $\bfM^r(\grn)$ denote the moduli scheme over $A[\grn^{-1}]$ of
(isomorphism classes of) Drinfeld $A$-modules of rank $r$ with 
level-$\grn$ structure. Let $M^r(\grn):=\bfM^r(\grn)
\otimes_{A[\grn^{-1}]} F$ denote the generic fiber of $\bfM^r(\grn)$.

\begin{thm}\label{ThmDrinfeld}[{\cite[Section 5]{drinfeld:1}}]
If $\grn\subset A$ a nonzero proper ideal, then the moduli 
scheme $\bfM^r(\grn)$ is an affine smooth scheme of
finite type over $\Spec A[\grn^{-1}]$ of relative dimension $r-1$.
\end{thm}

The finite group
$\GL_r(A/\grn)$ acts on level-$\grn$ structures, and hence gives a
right action on $\bfM^r(\grn)$ by $(E,\varphi, \lambda)\mapsto
(E,\varphi,\lambda g)$, for $g\in \GL_r(A/\grn)$. 
By inflation, $\GL_r(\wh A)$ acts on $\bfM^r(\grn)$ on the right.

For any element $g\in K(1):=\GL_r(\wh A)$, denote by 
\[ J_{\!g\,}: \bfM^r(\grn)\to \bfM^r(\grn), \quad (E,\varphi, \lambda)\mapsto
(E,\varphi,\lambda g), \]
the isomorphism translating the level structures. Observe that 
if $(E,\varphi,
\lambda)$ is the fiber of the universal family
$(\wt E,\wt \varphi, \wt \lambda)\to \bfM^r(\grn)$ at a point $x$,
then the fiber of $(\wt E,\wt \varphi, \wt \lambda)$ at its image
$J_{\!g\,}(x)$ is $(E,\varphi, \lambda g)$. Consider the family $(\wt E,\wt
\varphi, \wt \lambda g)$ over $\bfM^r(\grn)$. By the universal
property of fine moduli schemes, there is a unique morphism $\alpha_g:
\bfM^r(\grn)\to \bfM^r(\grn)$ such that $\beta_g:(\wt E,\wt
\varphi, \wt \lambda g)\simeq \alpha_g^*(\wt E,\wt
\varphi, \wt \lambda)$ over $\bfM^r(\grn)$. The above observation
and modular interpretation say that $\alpha_g=J_{\!g\,}$. 
With level structures ignored,
the composition of $\beta_g$ with the base change isomorphism
  $I_{\!g\,}:(\wt E,\wt \varphi)\isoto J_{\!g\,}^*(\wt E,\wt \varphi)$ gives a
  commutative diagram: 
\begin{equation}
    \label{eq:action}
\begin{CD}
  (\wt E,\wt \varphi) @>I_{\!g\,}>> (\wt E,\wt \varphi)\\ 
   @VVV @VVV \\
   \bfM^r(\grn) @>J_{\!g\,}>> \bfM^r(\grn).
\end{CD}    
\end{equation}
One easily checks that for $g_1,g_2\in K(1)$, the relations $J_{\!g_1
  g_2\,}=J_{\!g_2\,} \circ J_{\!g_1\,}$ and $I_{\!g_1
  g_2\,}=I_{\!g_2\,} \circ I_{\!g_1\,}$ hold on the moduli scheme
  $\bfM^r(\grn)$ and the universal family $(\wt E,\wt \varphi)$,
  respectively.
Therefore, there is a right action of $K(1)$ on the scheme (geometric
  line bundle) $\wt E$ which is equivariant for the action $J_{\!g\,}$ on $\bfM^r(\grn)$. 
 
Let $\tilde\grn\subset \grn$ be two nonzero proper ideals of $A$. Then $\grn^{-1}/A\subset \tilde\grn^{-1}/A$ and $\grn^{-1}/A$ is the $\grn$-torsion in $\tilde\grn^{-1}/A$.
Likewise $\varphi[\grn]$ equals the $\grn$-torsion in $\varphi[\tilde\grn]$. If $\wt\lambda\colon(\tilde\grn^{-1}/A)^r_S\isoto\varphi[\tilde\grn]$ is a level-$\tilde\grn$ structure on $\varphi$, then its restriction to $\grn$-torsion defines an isomorphism $\lambda\colon(\grn^{-1}/A)^r_S\isoto\varphi[\grn]$. 
Since $\GL_r(A/\tilde\grn)\to \GL_r(A/\grn)$ is
surjective with kernel isomorphic to $K(\grn)/K(\tilde\grn)$, the map  
\begin{equation}
  \label{eq:2.4}
  \Isom_S\bigl((\tilde\grn^{-1}/A)^r_S, \varphi[\tilde\grn]\bigr) \to \Isom_S\bigl((\grn^{-1}/A)^r_S,
  \varphi[\grn]\bigr), \quad \wt\lambda\mapsto \lambda:=\wt\lambda|_{(\grn^{-1}/A)^r_S}
\end{equation}
is surjective with fibers being $\bigl(K(\grn)/K(\tilde\grn)\bigr)_S$-orbits, if the source is
non-empty. Thus, the natural map 
\begin{equation}
  \label{eq:2.5}
  \pi_{\grn,\tilde\grn}:\bfM^r(\tilde\grn)\to \bfM^r(\grn), \quad (E,\varphi,
  \wt\lambda)\mapsto (E,\varphi, \lambda)
\end{equation}
induces an isomorphism 
\begin{equation}
  \label{eq:2.6}
  \bfM^r(\tilde\grn)/(K(\grn)/K(\tilde\grn)) \to
  \bfM^r(\grn)[\tilde\grn^{-1}]:=\bfM^r(\grn)\otimes_{A[\grn^{-1}]} 
  A[\tilde\grn^{-1}]. 
\end{equation}
Note that $\bfM^r(\tilde\grn)$ is the finite \'etale Galois 
cover of $\bfM^r(\grn)[\tilde\grn^{-1}]$
parameterizing the level-$\tilde\grn$ structures on the universal family
$( E,  \varphi, \lambda)$ on $\bfM^r(\grn)$
extending $\lambda$. Using this description, 
the pull-back of $( E, \varphi)$ on $\bfM^r(\grn)$
by $\pi_{\grn,\tilde\grn}$ is the universal family on $\bfM^r(\tilde\grn)$. 

Let $K\subset \GL_r(\wh A)=:K(1)$ be an open compact subgroup. 
Choose any proper nonzero ideal $\grn$ with $K(\grn)\subset K$. Define
the moduli scheme of level-$K$ structure by
\begin{equation}
  \label{eq:2.7}
  \bfM^r_{K}[\grn^{-1}]:=\bfM^r(\grn)/(K/K(\grn)), 
\end{equation}
(The name ``moduli scheme'' will be justified in Proposition~\ref{2.5} below.) It is an affine scheme, and is smooth over $A[\grn^{-1}]$ if the action of $K/K(\grn)$ on $\bfM^r(\grn)$ is free by \cite[Exp.~V,~Proposition~1.8]{SGA1}, \cite[Exp.~V, Theorem~7.1]{SGA3} and \cite[IV$_4$, Proposition~17.7.7]{EGA}.
By \eqref{eq:2.6} this does not depend on the choice of $\grn$.
Let $\grn_K$ denote the largest ideal satisfying the
property $K(\grn_K)\subset K$.
One can extend $\bfM^r_K[\grn^{-1}]$ uniquely to a moduli scheme
which is faithfully flat over $A[\grn_K^{-1}]$.
When $K\neq K(1)$, one can simply take $\grn:=\grn_K$. 
When $K=K(1)$, the scheme $\bfM^r_{K(1)}$ is obtained by gluing 
of affine schemes $\bfM^r_{K(1)}[\grn_1^{-1}]$ and 
$\bfM^r_{K(1)}[\grn_2^{-1}]$ along the open subset 
$\bfM^r_{K(1)}[(\grn_1 \grn_2)^{-1}]$, where $\grn_1$ and $\grn_2$ are
any two coprime proper ideals.

\begin{defn}
  Let $K\subset \GL_r(\wh A)=:K(1)$ be an open compact subgroup,
  and  $\grn_K$ denote the largest ideal satisfying the property
  $K(\grn_K)\subset K$. Let $\bfM^r_K$ denote the unique moduli scheme
  over 
  $A[\grn_K^{-1}]$ such that $\bfM^r_K\otimes_{A[\grn_K^{-1}]}
  A[\grn^{-1}]=\bfM^r_{K}[\grn^{-1}]$ for any ideal $\grn \subset
  \grn_K$. With this notation,  
  one has $\bfM^r(\grn)=\bfM^r_{K(\grn)}$ for any proper nonzero ideal
  $\grn\subset A$. Also let $M^r_K:=\bfM^r_K
\otimes_{A} F$ denote the generic fiber of $\bfM^r_K$.
\end{defn}

\begin{prop}\label{Prop2.4} 
{\rm (1)} The moduli scheme $\bfM^r_K$ is affine over $A[\grn_K^{-1}]$. Moreover, if $K\neq K(1)$ and the action of $K/K(\grn_K)$ on $\bfM^r(\grn_K)$ is free, then $\bfM^r_K$ is smooth over $A[\grn_K^{-1}]$.

{\rm (2)} There is a bijection between the set $\pi_0(M^r_K\otimes_F \ol
      F)$ of geometrically connected components and the ray class
      group $(\A^\infty)^\times/F^\times \det(K)$, where
      $\det(K)\subset \wh A^\times$ is the image under the determinant
      map. 
      If $K=K(\grn)$ then $\det(K)=(1+\grn \wh A)^\times$
      and $M^r_K$
      has $|(\A^\infty)^\times/F^\times (1+\grn \wh
      A)^\times|=h(A)\cdot|(A/\grn)^\times|/(q-1)$ 
      geometrically connected components, where $h(A)$ is the class
      number of $A$.
      The set of
      orbits of the $\GL_r(\wh A)$-action on $\pi_0(M^r_K\otimes_F \ol
      F)$ is in bijection with the
      ideal 
      class group $\Cl(A)=(\A^\infty)^\times/F^\times \wh
      A^\times$. In particular, the action of $\GL_r(\wh A)$ on
      $\pi_0(M^r_K\otimes_F \ol F)$ is transitive if and only if $A$
      is a principal ideal domain. 

{\rm (3)} The moduli schemes $\bfM^r_K$ and $M^r_K$ are connected as schemes. 
\end{prop}

Note that $\grn\mapsto|(A/\grn)^\times|$ is Euler's totient function for $A$.
 
\begin{proof}
(1) was already proved above.

(2) Choose a proper ideal $\grn\subset A$ with $K(\grn)\subset K$. Using the modular
interpretation of $M^n(\grn)$ and the analytic theory, there is a natural 
isomorphism of rigid analytic spaces
\[ M^r(\grn)(\C_\infty)\simeq \GL_r(F)\backslash \Omega^r (\C_\infty) 
\times
\GL_r(\bbA^\infty)/K(\grn), \]  
where $\Omega^r$ is the Drinfeld period domain of rank $r$ over
$\C_\infty$, see \cite[Theorem~5.6]{Deligne-Husemoeller}. It induces an isomorphism
\[ M^r_K(\C_\infty)\simeq \GL_r(F)\backslash \Omega^r (\C_\infty) 
\times
\GL_r(\bbA^\infty)/K\,. \]  
Recall that $\Omega^r(\C_\infty)$ 
is the complement of the union of
all $F_\infty$-rational hyperplanes in $\bbP^{r-1}(\C_\infty)$. 
Since $\Omega^r$ is connected as a rigid analytic space, one has a
bijection 
\[ \pi_0(M^r_K \otimes_F \C_\infty) \simeq \GL_r(F)\backslash
\GL_r(\bbA^\infty)/K\,. \]
Through the determinant map, strong approximation for $\SL_r$ yields 
a bijection
\[ \pi_0(M^r_K \otimes_F \ol F)=\pi_0(M^r_K \otimes_F \C_\infty)
\simeq (\A^{\infty})^\times/F^\times \det
(K)\,. \] 
Note that $\det K(\grn)=(1+\grn \wh A)^\times$.
The equality 
$|(\A^\infty)^\times/F^\times (1+\grn \wh
      A)^\times|=h(A)\cdot|(A/\grn)^\times|/(q-1)$ follows immediately from the exact sequences
\begin{equation}\label{Eq1Pic_n} 1 \longrightarrow \frac{F^\times \wh
      A^\times}{F^\times (1+\grn \wh
      A)^\times} \longrightarrow \frac{(\A^\infty)^\times}{F^\times (1+\grn \wh
      A)^\times} \longrightarrow \frac{(\A^\infty)^\times}{F^\times \wh
      A^\times} \longrightarrow 1 \quad {and}
    \end{equation}
    \begin{equation}\label{Eq2Pic_n}
    1 \longrightarrow A^\times=\Fq^\times \longrightarrow \wh A^\times/(1+\grn \wh A)^\times=(A/\grn)^\times \longrightarrow F^\times \wh
      A^\times/F^\times (1+\grn \wh
      A)^\times  \longrightarrow 1 
      \end{equation}
      and from $h(A)=|(\bbA^\infty)^\times/F^\times \wh A^\times|$.
The action of $\GL_r(\wh A)$ on $M^r_K(\C_\infty)$ is 
simply the right translation. Through the determinant map, this action
factors through the right translation of 
$\wh A^\times$ on $(\A^\infty)^\times/F^\times \det(K(\grn))$. 
Then the set of orbits is isomorphic to $\Cl(A)$.   

(3) It suffices to show the connectedness of $M^r(\grn)$ as this surjects
onto $M^r_K$ and $M^r_K\subset \bfM^r_K$ is open and dense. Indeed,
$\bfM^r_K$ is smooth over $A[\grn_K^{-1}]$, and so all its irreducible
components meet $M^r_K$ by \cite[Proposition~III.9.7]{hartshorne}. 
    The connectedness of $M^r(\grn)$ is mentioned in
    \cite[p.~335]{pink}, but this does not seem to be stated explicitly 
    in standard
    references \cite{drinfeld:1,laumon:1}. So we give a proof
    for the reader's convenience. It suffices to show that the
    $\Gal(\ol F/F)$-action on $\pi_0(M^r(\grn)\otimes_F \ol F)$ is transitive. 
    Using the Weil-pairing map $w_\grn: M^r(\grn)\to M^1(\grn)$ 
    constructed by van der Heiden
    \cite[Theorem 4.1]{vdHeiden}, one
    obtains a surjective morphism
    \begin{equation}
      \label{eq:pi0}
      \pi_0(M^r(\grn)\otimes_F \ol F)\to \pi_0(M^1(\grn) \otimes_F \ol
    F)=M^1(\grn)({\ol F})\simeq (\A^\infty)^\times/F^\times (1+\grn
    \wh A)^\times.
    \end{equation}
As shown in (2), $\pi_0(M^r(\grn)\otimes_F \ol F)$ and $\pi_0(M^1(\grn)\otimes_F \ol
    F)$ have the same cardinality, and hence the map \eqref{eq:pi0} 
   is bijective.   
By Drinfeld's description of Drinfeld modules of rank one 
\cite[Section 8]{drinfeld:1}, the group $\Gal(\ol F/F)$ acts 
transitively on the set $M^1(\grn)(\ol F)$ and hence
    on the set $\pi_0(M^r(\grn)\otimes_F \ol F)$. This completes the proof.
    \qed 
\end{proof}

A level-$K$ 
structure on a Drinfeld $A$-module $(E,\varphi)$ 
over an $A[\grn^{-1}]$-scheme $S$ with $K(\grn)\subset K$ 
is a $K$-orbit $\lambda K$ of 
level-$\grn$ structures that is defined over $S$. 
By (\ref{eq:2.4}), this does not depend on
the choice of $\grn$, provided that the support of $\grn$ is
unchanged. Define the contravariant functor 
\[ F^r_K: (\text{$A[\grn_K^{-1}]$-sch}) \to ({\rm Set}) \]
as follows. For $K\neq K(1)$, $F^r_K(S)$ is the set of isomorphism
classes of Drinfeld $A$-modules of rank $r$ with level-$K$ structure over
$S$. For $K=K(1)$, $F^r_{K(1)}(S)$ is the set of isomorphism
classes of Drinfeld $A$-modules of rank $r$ over $S$.

\begin{defn}
  An open compact subgroup $K\subset K(1)$ is said to be {\it fine} if
  there is a prime ideal $\grp\subset A$ such that the image of $K$ 
  in $\GL_r(A/\grp)$ is unipotent.
\end{defn}

\begin{lemma}\label{2.4}
  If $K$ is fine, then the automorphism group of any Drinfeld
  $A$-module of rank $r$ with level-$K$ structure 
  $(E,\varphi, \lambda K)$ over $S$ is trivial.   
\end{lemma}
\begin{proof}
  One can assume that $S=\Spec k$ with an \ac field $k$. As $K$ is
  fine, the image of $K$ in $\GL_r(A/\grp)$ is unipotent for some prime
  $\grp$. Let $g\in \Aut(\varphi)$ be an automorphism fixing
  $\lambda K$. Observe that $g$ fixes $\lambda K'$ for any subgroup 
  $K'\supset K$. 
  Thus, after replacing $K$ by $K(\grp)K$, one can assume that
  $K(\grp)\subset K$. 
  It is known that 
  the automorphic group $\Aut(\varphi)$ is isomorphic to 
  $\F_{q^m}^\times$ for
  some integer $m|r$. This follows from the fact that the endomorphism
  algebra $D=\End^0(\varphi)$ of $\varphi$ 
  is totally ramified at $\infty$, i.e. $D_\infty$ is a central division
  $F_\infty$-algebra. See an argument in \cite[p.322]{gekeler:mass} or 
  \cite[Lemma 2.5]{wei-yu:classno}.
  Thus, we have an $\Fq$-algebra homomorphism
\begin{equation}\label{EqEndInjective} \F_{q^m} \to \End(\varphi) \to \End(\varphi[\grp])\simeq
  \Mat_r(A/\grp). 
  \end{equation}
  Since $\F_{q^m}$ is a field, the map \eqref{EqEndInjective}, and hence also $\F_{q^m}^\times=\Aut(\varphi) 
  \to \GL_r(A/\grp)$ are injective. Now $g \lambda K=\lambda K$ implies
  that $g\in \F_{q^m}^\times \cap \lambda K \lambda^{-1}$.
  Therefore, the image of $g$ in $\GL_r(A/\grp)$ is both semi-simple
  and unipotent and is trivial by the Jordan decomposition. This shows
  that $\Aut(E,\varphi,\lambda K)=1$. \qed   
\end{proof}

\begin{prop}\label{2.5}
  {\rm (1)} For any open compact subgroup $K$, $\bfM^r_K$ is the
      coarse moduli scheme for the 
      functor $F^r_K$. That is, there is a natural transformation
      $\tau: F^r_K\to h_{\bfM^r_K}:=\Hom(\fdot,\bfM^r_K)$, and $\tau$ is
      universal among such natural transformations and induces a bijection of the sets of $k$-points for any algebraically closed field $k\supset \Fp$.
       
  {\rm (2)} If $K$ is fine, then $\bfM^r_K$ represents the functor
      $F^r_K$ and is smooth over $A[\grn_K^{-1}]$ of relative dimension $r-1$. If $K'\normal K$ is normal and $K$ is fine, then the natural
      morphism 
\[ \pi_{K,K'}: \bfM^r_{K'} \to \bfM^r_K [\grn_{K'}^{-1}] \]
is finite \'etale Galois with group $K/K'$. The pull-back of the universal
family $(\wt E,\wt \varphi)$ on $\bfM^r_K$ by $\pi_{K,K'}$ is the universal
family on $\bfM^r_{K'}$.  
\end{prop}
\begin{proof}
  (1) This follows from the construction of $\bfM^r_K$ and the fine
      moduli scheme $\bfM^r_{K(\grn)}=\bfM^r(\grn)$ for a suitable $K(\grn)\subset K$. Indeed, suppose we
      have an object $(E,\varphi,\lambda K)$ over $S$. Then one finds a
      finite Galois  $K/K(\grn)$-cover $S_\grn\to
      S[\grn^{-1}]$ 
      and a family $(E,\varphi,\lambda)$ with level-$\grn$
      structure over $S_\grn$. 
      By the universal property, 
      there is a unique morphism
      $f: S_\grn\to \bfM^r_{K(\grn)}$ such that $(E,\varphi,\lambda)$
      is isomorphic to the pull-back of the universal family. 
       The composition
      $\pi_{K,K(\grn)}\circ f: S_\grn \to \bfM^r_K[\grn^{-1}]$ is 
      $K/K(\grn)$-invariant,
      and hence it induces a unique morphism 
      $f': S[\grn^{-1}]=S_\grn/(K/K(\grn)) \to \bfM^r_K[\grn^{-1}]$. 
      It is straightforward to check this transformation satisfies the
      universal property. 
  
(2) It follows from Lemma~\ref{2.4} that the right action of
      $K/K(\grn)$
      on $\bfM^r(\grn)$ is free. Thus, $\bfM^r_K$ is smooth over $A[\grn_K]$ by Proposition~\ref{Prop2.4} (1), and the universal family $(\wt
      E,\wt \varphi, \wt \lambda)$ on $\bfM^r(\grn)$ descends uniquely
      to a family $(E_K, \varphi_K)$ on $\bfM^r_K$. One can show that the
      $K$-orbit $\wt \lambda K$ is defined over $\bfM^r_K$. Thus, 
      one obtains
      a family $(E_K, \varphi_K,\wt \lambda K)$ in
      $F^r_K(\bfM^r_K)$. Using the same argument as in (1), we show
      that for any object $(E,\varphi, \lambda K)\in F^r_K(S)$, there is
      a unique map $f:S\to \bfM^r_K$ such that $(E,\varphi,\lambda
      K)\simeq f^*(E_K, \varphi_K,\wt \lambda K)$. This shows that
      $\bfM^r_K$ represents $F^r_K$. The remaining assertions
      follow from the same reason as in (\ref{eq:2.6}).\qed
\end{proof}

\begin{defn}[cf.~{\cite[Section~3]{pink}}]\label{2.6}
  A generalized Drinfeld $A$-module $(E,\varphi)$ over $S$ is called
  {\it weakly separating}, if for any Drinfeld $A$-module
  $(E',\varphi')$ over any
  $A$-field $L$, at most finitely many fibers of $(E,\varphi)$ over
  $L$-valued points of $S$ are isomorphic to $(E',\varphi')$.
\end{defn}

Note that our ``test'' objects $(E',\varphi')$ can be in finite
characteristic, in contrast to \cite{pink}.  
 
\begin{prop}\label{2.7}
  Let $(E,\varphi)$ be a weakly separating generalized 
  Drinfeld $A$-module over
  an $A$-scheme $S$ of finite type. Then for any positive integer $r$, 
  there is a unique closed subscheme $S_{\le r}$ of $S$ such that 
  any morphism $f:T\to S$ with the property 
  that  $f^*(E,\varphi)$ is of rank $\le
  r$ over $T$ factors through the inclusion 
  $S_{\le r}\to S$. Moreover, $S_{\le r}$ has
  relative dimension $\le r-1$ over $\Spec A$.      
\end{prop}
\begin{proof}
  The proof is the same as \cite[Proposition~3.10]{pink} and is included
  merely for the reader's convenience. 
  The first statement is local on $S$. Thus, we may assume that
  $E=\bbG_{a,S}$ and  
  $S=\Spec R$ is affine. Suppose $A$ is generated by $a_1, \dots, a_s$
  as an $\Fq$-algebra. Let $S_{\le r}$ be the closed subscheme defined
  by the ideal generated by the elements
  $\varphi_{a_j,i}$ for $j=1,\dots, s$ and $i> r \deg(a_j)$. Then it
  is easy to verify that $S_{\le r}$ satisfies the universal property
  in the proposition.
  The second statement is local for $\Spec A$. Thus, 
  we may further assume that $S$ is an 
  $A[\grn^{-1}]$-scheme for a nonzero proper ideal $\grn\subset A$.  
  For any integer $1\le r'\le r$, let $S_{r'}:=S_{\le r'} \setminus S_{\le r'-1}$
  and $S_{1}:=S_{\le 1}$. It suffices to show that each $S_{r'}$ has
  relative dimension $\le r'-1$. 

  Over $S_{r'}$ the universal generalized Drinfeld module $(E,\varphi)$ is actually a (genuine) Drinfeld module of rank $r'$. Thus the $\grn$-torsion of $(E,\varphi)$ is a finite \'etale group scheme over $S_{r'}$. Adding level-$\grn$ structures
  to $(E,\varphi)$ over $S_{r'}$, that is trivializing this $\grn$-torsion, one obtains a finite \'etale cover $\wt
  S_{r'}$ of $S_{r'}$. This yields a morphism $f:\wt S_{r'}\to
  \bfM^{r'}(\grn)$ by the universal property of fine moduli schemes. 
  As $\wt S_{r'}$ is finite over $S_{r'}$, it suffices to show that
  $\wt S_{r'}$ has relative dimension $\le r'-1$, which is $\le r-1$.  
  By the property that $(E,\varphi)$ is weakly separating, the
  morphism $f$ is quasi-finite. Therefore, $\wt S_{r'}$ has relative 
  dimension $\le r'-1$ over $\Spec A$ and the proposition is 
  proved. \qed 
\end{proof}

Note that in general the locally closed stratum 
$S_r:=S_{\le r} \setminus S_{\le r-1}$ may not be dense 
in $S_{\le r}$ even if $S_r$ is
  non-empty. For example suppose we have a family $S$ with both $S_r$ and
  $S_{\le r-1}$ non-empty. Define a new family $T:=S_{\le r-1} \coprod
  S_{\le r}$ as the topologically disjoint union of $S_{\le r-1}$ and $S_{\le r}$. Then $T_r$ is not dense in $T$.  

We may view generalized Drinfeld modules of higher rank as the function
field analogue of semi-abelian varieties. 
The following is the analogous result for the semistable
reduction theorem for abelian varieties due to Grothendieck, Deligne
and Mumford \cite{DM}.  

\begin{prop}
  Let $R$ be a 
  discrete valuation ring with fraction field $L$, $\gamma:A\to R$ a ring homomorphism, 
  and let $(E,\varphi)$ be a
  Drinfeld $A$-module of rank $r$ over $L$. Then there is a finite
  tamely ramified extension $L'/L$, a generalized Drinfeld
  $A$-module of rank $\le r$ over $R'$, and an isomorphism $\alpha:
  (E,\varphi)\otimes_L L'\simeq (E',\varphi')\otimes_{R'} L'$, where
  $R'$ is the integral closure of $R$ in $L'$.   
\end{prop}
\begin{proof}
  This is proved by Drinfeld in a terse style in his stable reduction theorem \cite[Proposition~7.1]{drinfeld:1} 
  when $R$ is complete. We
  provide more details for the reader's convenience.  
  Let $\pi$ be a uniformizer of $R$ and let $v$ be the valuation on
  $L$ with $v(\pi)=1$.
  We first prove the case where $A=\Fq[t]$. Suppose
  $\varphi_t=a_0+a_1\tau+\dots +a_{r} \tau^r$, $a_i\in L$ and 
  $a_r\neq 0$. Over a field extension $L'/L$, $\varphi_t$ is
  isomorphic to $\varphi^c_{t}=a'_0+\dots+ a'_r \tau^r$ with
  $a'_i=c^{q^i-1} a_i$ ($0\le i \le r$) 
  for some $0\neq c\in L'$. The isomorphism is given by
  $c\circ\varphi_t^c=\varphi_t\circ c$. Let 
\[ \nu:=\min_{1\le i\le r}
  \left \{\frac{v(a_i)}{(q^i-1)}\right \}. \] 
Then there is an integer $0<i_0\le r$ such that 
  $v(a_{i_0})= (q^{i_0}-1)
  \nu$ and $v(a_{i})\ge (q^{i}-1)\nu$ for all $0\le i\le r$. 
  
  Now take 
  $L'=L(\pi')$, $\pi'=\pi^{1/(q^{i_0}-1)}$ and
  $c:=\pi'^{-v(a_{i_0})}$. Clearly,
  $v(c)=-v(a_{i_0})/(q^{i_0}-1)=-\nu$. Thus, one has 
\[ v(a_{i}c^{q^i-1})=v(a_i)-(q^i-1)\nu\ge 0,\quad 0\le i\le r,\quad
  \text{and} \quad v(a_{{i_0}}c^{q^{i_0}-1})= 0. \] 
  The morphism $\varphi'_t:=\varphi^c_{t}$ defines a generalized Drinfeld
  $A$-module $(E',\varphi')$ of rank $\le r$ over $R'$, which has the
  desired property.  
 
  Now $F$ is arbitrary. Choose $t\in A\setminus\Fq$ and put
  $A_0:=\Fq[t]$ with fraction field $F_0$. 
  Then the restriction of $\varphi$ to $A_0$ gives rise to a Drinfeld
  $A_0$-module $(E_0,\varphi_0)$ over $L$ of rank $r_0=n r$ , where
  $n=[F:F_0]$. We have shown that there is a tamely ramified extension
  $L'/L$, a generalized Drinfeld $A_0$-module $(E'_0,\varphi'_0)$ of 
  rank $\le r_0$ over $R'$, and an isomorphism 
  $\alpha:(E,\varphi)\otimes_L L'\simeq (E',\varphi')$ of Drinfeld
  $A_0$-modules. Since $A$ commutes with $A_0$, each element
  $\varphi_{a}$ for $a\in A$, can be viewed as an element
  $\varphi'_{a}$ in
  $\Hom((E_0',\varphi'_0)|_{L'},(E_0',\varphi'_0)|_{L'})$. By
  \cite[Proposition~3.7]{pink}, 
  this homomorphism extends uniquely to a homomorphism
  $\varphi'_a$ over $R'$, i.e. $\varphi'_a\in R'\{\tau\}$. Thus, we
  have a generalized Drinfeld $A$-module $(E',\varphi')$ of rank $\le
  r$ over $R'$ satisfying the desired property. \qed        
\end{proof}

\section{The arithmetic Satake compactification}
\label{sec:S}

We keep the notation in the previous section. 
In this section we construct the arithmetic Satake compactification of
the Drinfeld moduli scheme $\bfM^r(\grn)$ over $A[\grn^{-1}]$. 
The Satake compactification of the generic fiber
$M^r(\grn)=\bfM^r(\grn)\otimes_{A[\grn^{-1}]} F$ has been constructed
by Kapranov \cite{kapranov} for $F=\Fq(t)$ and 
$A=\Fq[t]$ using the analytic construction.
Pink \cite{pink} gave a different
construction for the Satake compactification of the generic 
fiber $M^r_{K}:=\bfM^r_{K}\otimes_{A[\grn_K^{-1}]} F$ 
for any global function field $F$ and any fine subgroup 
$K\subset \GL_r(\wh A)$ using generalized Drinfeld modules. 
The arithmetic construction for $\bfM^r_K$ over $A[\grn_K^{-1}]$ presented here follows directly 
along the line of Pink's construction. 

\begin{defn}\label{def:satake}
  For any fine open compact subgroup $K$, a dominant open immersion 
  $\bfM^r_K \embed \ol{\bfM}^r_K$ over $A[\grn^{-1}]$, where
  $\grn=\grn_K$,  
  with the properties
\begin{enumerate}
\item[(a)] $\ol{\bfM}^r_K$ is a normal integral scheme which is
  proper flat over $\Spec A[\grn^{-1}]$, and 
\item[(b)] the universal family over $\bfM^r_K$ extends to a weakly 
 separating generalized Drinfeld $A$-module $(\olE, \ol \varphi)$ 
 over $\ol{\bfM}^r_K$. 
\end{enumerate}
is called an (\emph{arithmetic}) \emph{Satake-Pink}, or \emph{Satake compactification} of
$\bfM^r_K$. By abuse 
of terminology we call $(\olE,\ol \varphi)$ the universal family on
$\ol{\bfM}^r_K$. 
\end{defn}

Here our word ``arithmetic'' refers to the fact that our compactification is proper flat over the ``arithmetic'' base scheme $\Spec A[\grn^{-1}]$, while Pink considered the ``algebraic'' case over the fraction field $F$ of $A[\grn^{-1}]$. As far as we know, it was expected that the 
universal Drinfeld module extends 
to a generalized Drinfeld module over a certain
compactification of $\bfM_K^r$. For rank $r=2$, 
where $\bfM^r_K$ is a relative curve, this was proved by Drinfeld 
\cite[Section 9]{drinfeld:1}. 
However, for arbitrary ranks the precise description (as
Definition~\ref{def:satake}) first appeared in  
the work of Pink \cite{pink}. As shown in
loc.~cit., Pink's formulation and approach give a rather simple way for
compactifying Drinfeld modular varieties. One advantage is that 
one does not need to construct the boundary components and analyze how
to glue them, which was done for $A=\Fq[t]$ by Kapranov
\cite{kapranov} using the analytic construction. 
In \cite[Remark 4.9]{pink}, Pink suggested
to compare Kapranov's compactification and the Satake-Pink
compactification. Indeed, H\"aberli \cite{haberli} verifed that 
the universal family over the generic fiber $M^r(\grn)$ 
extends to a weakly separating generalized Drinfeld
module over Kapranov's compactification, and hence both compactifications coincide.

\begin{thm}\label{S.2}
(1)  For every fine open compact subgroup $K$, the moduli scheme $\bfM^r_K$
  possesses a projective arithmetic Satake compactification 
  $\ol{\bfM}^r_K$. The Satake compactification and
  its universal family are unique up to unique isomorphism. 
  The dual $\omega_K:=\Lie(\olE_K)^\vee$ of 
  the relative Lie algebra of the universal family $\olE_K$
  over $\ol{\bfM}^r_K$ is ample.

(2) If $\wt K\subset K$ are fine open compact subgroups, then ``forgetting the level'' induces a finite surjective and open morphism 
\[
\pi_{\wt K,K}\colon\ol\bfM^r_{\wt K}\longrightarrow\ol\bfM^r_K\otimes_{A[\grn_K^{-1}]}A[\grn_{\wt K}^{-1}]
\]
over $\Spec A[\grn_{\wt K}^{-1}]$ which satisfies $\pi_{\wt K,K}^*(\olE_K,\ol\varphi_K)=(\olE_{\wt K},\ol\varphi_{\wt K})$ and $\pi_{\wt K,K}^*(\omega_K)=\omega_{\wt K}$.
Moreover, if $\wt K\lhd K$ is normal, then $\pi_{\wt K,K}$ identifies $\ol\bfM^r_K\otimes_{A[\grn_K^{-1}]}A[\grn_{\wt K}^{-1}]$ with the quotient $(\ol\bfM^r_{\wt K})^{(K/\wt K)}$ of $\ol\bfM^r_{\wt K}$ by the action of the finite group $K/\wt K$.
\end{thm}

We begin the proof of the theorem with the following lemma. 
 
\begin{lemma}\label{S.3}
  If a Satake compactification $\ol{\bfM}^r_K$ and the universal
  family $(\olE,\ol \varphi)$ exist, then they are unique up to
  unique isomorphism. 
\end{lemma}
\begin{proof}
  The proof is the same as that of \cite[Lemma 4.3]{pink}. We include
  it for the sake of completeness.
  Abbreviate $\ol \bfM:=\ol{\bfM}^r_K $, and let $\ol{\bfM'}$ be
  another Satake compactification of $\bfM:=\bfM^r_K$ with
  universal family $(\ol{E'}, \ol{\varphi'})$. Let $\wt \bfM$ be the
  normalization of the Zariski closure $\bfM^{\rm zar}$ of 
  the diagonal embedding
  $\bfM\embed \ol \bfM \times_A \ol{\bfM'}$. Then we have two
  projections 
\[ 
\begin{CD}
  \ol \bfM @<{\pi}<< \wt \bfM @>{\pi'}>> \ol{\bfM'}
\end{CD} \]
  which are proper and are the identity when restricted to $\bfM$. The
  morphism $(\pi,\pi'): \wt \bfM\to \ol \bfM\times_A \ol{\bfM'}$ is finite,
  because $\bfM^{\rm zar}$ is excellent, see 
  \cite[IV$_2$, Scholie~7.8.3 (ii), (v)]{EGA}. One obtains two 
  generalized Drinfeld modules $\pi^*(\olE,\ol \varphi)$ and 
  $\pi'^*(\ol{E'},\ol \varphi')$ over $\wt \bfM$. By
  \cite[Proposition~3.7 and 3.8]{pink}, the identity on the universal family
  over $\bfM\subset \wt \bfM$
  extends to an isomorphism $\pi^*(\olE,\ol \varphi)\isoto 
  \pi'^*(\ol{E'},\ol \varphi')$ over $\wt \bfM$. 
  For any geometric point $x\in \ol
  \bfM(\olL)$ where $\olL$ is an algebraically closed field, we restrict the map $\pi':\wt \bfM(\olL) \to \ol{\bfM'}(\olL)$ to the subset $\pi^{-1}(x)\subset \wt \bfM(\olL)$. One
  obtains a finite-to-one map
  \begin{equation}
    \label{eq:finite}
\begin{CD}
\pi':\pi^{-1}(x) @>{(\pi,\pi')}>> \{x\}\times \pi'(\pi^{-1}(x)) @>{\rm
    pr_2}>{\sim}> \pi'(\pi^{-1}(x)).   
\end{CD}    
  \end{equation}
  Since the fibers $(\ol{E'}, \ol{ \varphi'})$ over points 
  $y \in \pi'(\pi^{-1}(x))$ are isomorphic to $(\olE_x,\ol \varphi_x)$
  and $(\ol{E'}, \ol{ \varphi'})$ is weakly separating, 
  the set $\pi'(\pi^{-1}(x))$ is finite. By \eqref{eq:finite},
  $\pi^{-1}(x)$ is finite and hence the morphism $\pi$ is finite. 
As $\pi$ is birational and both $\wt \bfM$ and $\ol
  \bfM$ are normal, the morphism $\pi$ is an isomorphism by Zariski's
  Main Theorem \cite[\S\,III.9, I. Original form, p.~209]{Mumford}. Similarly,
  one proves that $\pi'$ is also an isomorphism. Thus one obtains 
  unique isomorphisms $\xi:=\pi'\circ \pi^{-1}: \ol 
  \bfM \isoto \ol{\bfM'}$
  and $(\olE, \ol \varphi)\isoto \xi^*(\ol{E'}, \ol{\varphi'})$. \qed    
\end{proof}

\begin{prop}\label{S.4}
  Let $f:X\to S=\Spec R$ be a quasi-projective scheme over 
  a Noetherian ring $R$, 
  together with a right action by a finite group $G$ over $R$. 
  Then

  (1) The quotient scheme $Y=X/G$ exists and the canonical
   morphism $\pi: X \to Y$ is finite surjective and open.

  (2) The quotient scheme $Y$ is quasi-projective over $R$. It is 
  projective over $R$ if and only if $X$ is so.

  (3) If the group $G$ acts freely in the sense that the
  morphism $X\times G \to X\times_S X, (x,g)\mapsto (x,g.x)$ 
  is a monomorphism, then $\pi:X\to Y$ is faithfully flat.

  (4) Assume that $Y$ is regular. Then $X$ is Cohen-Macaulay if and only if $\pi:X\to Y$ is flat. 
\end{prop}
\begin{proof}
(1) and (3) By \cite[Exp.~V, Theorem~7.1]{SGA3} the quotient scheme exists, is of finite type over $R$ and the morphism $\pi\colon X\to Y$ is surjective, open, proper and quasi-finite, hence finite. Moreover, $\pi$ is flat if $G$ acts freely.

(2) 
Since $f$ is quasi-projective, we can choose an $f$-very ample
  invertible sheaf $\calL$ on $X$. Replacing
  $\calL$ by $\otimes_{g\in G} g^*\calL$, one may assume that $\calL$ is
  $G$-equivariant and still $f$-ample by \cite[II, Corollaire~4.6.10]{EGA}. Some power $\calL^{\otimes n}$ with $n>0$ will be $f$-very ample by \cite[II, Proposition~4.6.11]{EGA}. Then $H^0(X,\calL^{\otimes n})$ gives rise to a $G$-equivariant embedding $X\embed \bbP^{N-1}$ where $G$ acts on $\bbP^{N-1}$ linearly. The closure $\ol X$ of $X$ is equal to $\Proj T$ of a graded ring $T=\oplus_{n\ge 0} T_n$ that is generated by $T_1$ over $T_0$. $T_0$ is an $R$-algebra which is finite as an $R$-module, and therefore it is Noetherian. By the construction of $\ol X/G$ and $X/G$, we know that $\ol Y:=\ol X/G$ is isomorphic to $\Proj T^G$ and $Y$ is an open subscheme of  $\ol Y$. So it suffices to show that $\ol Y$ is projective. Since $T$ is of finite type over a Noetherian ring $T_0$ and $G$ is finite, the invariant subring $T^G$ is again of finite type over $T_0$; see \cite[Theorem 1.2]{hashimoto:geom_quot2004}. 
Suppose that $T^G$ is generated by elements of degree $\ge d$ for some $d$. Then the graded $T':=\oplus_{n\ge 0} T^G_{dn}$ is generated by elements in $T_{d}$, which are of degree 1 in the new graded ring. By \cite[II, Proposition~2.4.7]{EGA}, the induced morphism $\Proj T' \to \Proj T^G$ is an isomorphism. This shows that $\ol Y$ is projective over $T_0$ (and hence over $R$).  

(4) This follows from \cite[Theorem~18.16 and Corrollary~18.17]{eisenbud}. 
\qed
\end{proof}

\begin{remark}\label{S.5}
  (1) The finite surjective 
  morphism $X\to Y$ in Proposition~\ref{S.4} may not be flat in general.
  Singularities of $X$ and $Y$ play a crucial role. 
  Here is a counter-example where
  $X$ is regular and $Y$ is normal. 

  Take $X=\Spec B$ with $B=\C[x_1, x_2]$ and let 
  the finite group $G=\{\pm 1\}$ act on $B$ by
  $-1:(x_1,x_2)\mapsto (-x_1,-x_2)$. Then we have 
  $C:=B^G=\C[x_1^2,x_1 x_2, x_2^2]$ and $Y=\Spec C$. 
  We show that $B$ is not $C$-flat.
  To see this, consider the maximal ideal $\grm=(x_1^2,x_1 x_2, x_2^2)$ of
  $C$ and put $L:=\Frac(C)$. One computes $\dim_{\C} B\otimes_C C/\grm
  =\dim_\C B/\grm B=3$ and $\dim_{L}B\otimes_C L=2$. Thus, $B$ can not be $C$-flat.
  
  (2) We thank David Rydh for explaining the proof of 
  Proposition~\ref{S.4}(2) to us and pointing out the reference \cite[IV Proposition~1.5, p.~180]{knutson}.  
  Rydh proves a more general result for algebraic spaces with finite flat groupoids; 
  see \cite[Proposition~4.7 (B')]{rydh:jag2013} for more details. Proposition~\ref{S.4}(2) also follows from \cite[Theorem~2.9]{AltmanKleiman}, because a strongly quasi-projective scheme over an affine base is quasi-projective.
\end{remark}

\begin{lemma}\label{S.6}
  Suppose that $K$ is fine and $K(\grn)\subset K$. Suppose 
  that  $\bfM^r_{K(\grn)}$
  has a projective Satake compactification $\ol \bfM^r_{K(\grn)}$ with
  universal family $(\olE,\ol \varphi)$ for which the dual
  $\omega_{K(\grn)}$  
  of the relative Lie algebra of $\olE$ is ample. 
Then

{\rm (1)} The action of $K/K(\grn)$ on $\bfM^r_{K(\grn)}$ and $(E,
  \varphi)$ 
  as in \eqref{eq:action} extends uniquely to an
  action on  $\ol \bfM^r_{K(\grn)}$ and $(\olE,\ol \varphi)$,
  respectively. 

{\rm (2)} The quotient 
  $\ol \bfM^r_{K(\grn)}/(K/K(\grn))$
  furnishes the projective Satake compactification
  $\ol{\bfM}^r_K[\grn^{-1}]$ of $\bfM^r_K[\grn^{-1}]$ with universal
  family $(\olE_K, \ol \varphi_K)$, where $\olE_K:=\olE/(K/K(\grn))$ is the
  quotient and $\ol \varphi_K$ is the
  Drinfeld $A$-module structure on $\olE_K$ 
  descended from $\ol \varphi$.

{\rm (3)} The dual $\omega_K:=\Lie(\olE_K)^\vee$ of 
  the relative Lie algebra 
  of $\olE_K$ over $\ol{\bfM}^r_K[\grn^{-1}]$ is ample. 
\end{lemma}

\begin{proof}
  (1)  Write $\bfM=\bfM^r_{K(\grn)}$ and $\ol\bfM=\ol\bfM^r_{K(\grn)}$. 
  For any $g\in K$, the action of $g$ gives the following commutative 
  diagrams:
  \begin{equation}
    \label{eq:IgJg}
    \begin{CD}
  ( E,\varphi) @>>> \bfM \\
  @V{I_{\!g\,}}VV @V{J_{\!g\,}}VV \\
  ( E,\varphi) @>>> \bfM 
\end{CD}\qquad \text{and}\qquad
\begin{CD}
  (\olE,\ol \varphi) @>>> \ol \bfM \\
   @V{?\,\ol I_{\!g\,}}VV  @V{?\,\ol J_{\!g\,}} VV     \\
  (\olE,\ol \varphi) @>>> \ol \bfM. 
\end{CD}
\end{equation}
Let $\wt \bfM$ be the normalization of the Zariski closure of the 
graph of $J_{\!g\,}$ in $\ol \bfM\times_{A[\grn^{-1}]} \ol \bfM$. 
It is equipped with two projections $\pi:\wt \bfM\to \ol \bfM$ and 
$\pi':\wt \bfM \to \ol \bfM$. The argument of Lemma~\ref{S.3} shows
that $\pi$ and $\pi'$ are isomorphisms and the isomorphism $I_{\!g\,}$ extends
to an isomorphism 
$\pi^*(\olE, \ol \varphi)\isoto \pi'^* (\olE, \ol \varphi)$. Since
$\pi$ and $\pi'$ are isomorphisms, the morphisms  $J_{\!g\,}$ and $I_{\!g\,}$ 
extend to morphisms $\ol J_{\!g\,}$ and $\ol I_{\!g\,}$ on $ \ol \bfM$ and $(\olE,\ol \varphi)$, respectively. 

(2) By Proposition~\ref{S.4}, the quotients $\ol\bfM/(K/K(\grn))$ and 
    $\olE/(K/K(\grn))$ exist as schemes and the quotient 
    $\bfM^r_K[\grn^{-1}]$ of $\bfM$ is open in $\ol\bfM/(K/K(\grn))$.
    Moreover, $\ol\bfM/(K/K(\grn))$ is projective over $A[\grn^{-1}]$.
    We next show that $\ol \bfM/(K/K(\grn))$ is a normal integral
    scheme proper flat over $A[\grn^{-1}]$. Note that 
    if $R$ is a normal domain with quotient field $Q$ with
    an action by a finite group $G$, then $R^G$ is again a normal
    domain. To see this, suppose that $x\in Q^G$ is integral over
    $R^G$. Then $x\in R \cap Q^G=R^G$. The normality of $\ol
    \bfM/(K/K(\grn))$ follows. 
    To show the flatness of $\ol \bfM/(K/K(\grn))$, one
    must show that the generic point of $\ol \bfM/(K/K(\grn))$ dominates
    $\Spec A[\grn^{-1}]$; see \cite[Proposition~III.9.7]{hartshorne}. This
    follows from that the generic point of $\ol \bfM\otimes_{A[\grn^{-1}]}F$ maps to that of
    $\ol \bfM/(K/K(\grn))$ and it dominates $\Spec 
    A[\grn^{-1}]$.  

    As $K$ is fine, it is proved in \cite[Lemma
  4.4]{pink} that $\olE_K:=\olE/(K/K(\grn))$ is again a line bundle over $\ol
    \bfM$. Since $\ol \varphi$ is invariant under the action of $K$,
    it descends to a Drinfeld module structure $\ol \varphi_K$ on $\olE_K$. 
    Since $(\olE,\ol\varphi)$ is weakly separating, also 
    $(\olE_K,\ol\varphi_K)$ trivially is weakly separating.

(3) Observe that $\omega_{K(\grn)}=g^*\omega_K$ under the finite
surjective quotient morphism $g:X':=\ol \bfM^r_{K(\grn)}
\longrightarrow X:=\ol\bfM^r_{K(\grn)}/(K/K(\grn))$, because 
$g^*\olE_K = \olE$. Since the base $Y:=\Spec A[\grn^{-1}]$ is affine,
$\omega_K$ is ample on $X$, if and only if it is ample relative to
$Y$, by \cite[II, Corollary~4.6.6]{EGA}. Since $\omega_{K(\grn)}$ is ample
on $\ol \bfM^r_{K(\grn)}$ we will now apply \cite[II, Corollary~6.6.3]{EGA}
to conclude that $\omega_K$ is ample on $X$. We may apply loc.\ cit.\
because condition (II bis) in \cite[II, Proposition~6.6.1]{EGA} is satisfied
for $(X,\calO_X)$ and $g_*\calO_{X'}$. Namely, let $\eta\in X$ and
$qf(X):=\calO_{X,\eta}$ be the generic point and the function field of
$X$ and similarly for $X'$. Then condition (II bis), stated on
page~126 of loc.\ cit., requires that for any affine open subscheme
$U\subset X$ and any section $f\in (g_*\calO_{X'})(U)$ the
characteristic polynomial $T^n -
\sigma_1(f)T^{n-1}+\ldots+(-1)^n\sigma_n(f)\in qf(X)[T]$ of the
multiplication by $f$ on the $qf(X)$-vector space
$g_*\calO_{X'}\otimes_{\calO_X}qf(X)=qf(X')$ has all its coefficients
$\sigma_i(f)$ in $\calO_X(U)$. Since $X$ is normal, $\calO_X(U)$
equals the intersection in $qf(X)$ of the local rings $\calO_{X,x}$
for all points $x\in X$ of codimension one. All these local rings are
discrete valuation rings. Since
$g_*\calO_{X'}\otimes_{\calO_X}\calO_{X,x}$ equals the normalization
of $\calO_{X,x}$ in $qf(X')$, it is a free $\calO_{X,x}$-module and
this implies that all the coefficients $\sigma_i(f)$ lie in
$\calO_{X,x}$. So condition (II bis) is indeed satisfied. \qed
\end{proof}

Let $F'$ be a finite field extension of $F$ with only one place $\infty'$
over $\infty$, and let $A'$ be the integral closure of $A$ in $F'$. 
Then $A'$ consists of all elements in $F'$ regular away from
$\infty'$, and $A'$ is a projective finite $A$-module of rank $[F':F]$. 
Let $r$ and $r'$ be positive integers with $r=r'[F':F]$. 
Let $\grn\subset A$ be a non-zero proper ideal, and put $\grn':=\grn A'$. 
Note that $\grn'{}^{-1}/A'=A'\otimes_{A} (\grn^{-1}/A)$ which is non-canonically isomorphic to $(\grn^{-1}/A)^{\oplus[F':F]}$ as a module over the principal ring $A/\grn$. Thus, we can fix an isomorphism 
$(\grn'{}^{-1}/A')^{r'}\simeq (\grn^{-1}/A)^r$ of $A$-modules.  
Then we have a natural finite morphism 
\begin{equation} \label{eq:S.1}
  I_b: \bfM^{r'}_{A'}(\grn') \to \bfM^{r}_{A}(\grn)
\end{equation}
sending each $(E', \varphi',\lambda')$ to $(E', \varphi'|_A,
\lambda')$, which fits into the commutative diagram
\begin{equation} 
  \begin{CD}
    \bfM^{r'}_{A'}(\grn') @>{I_b}>> \bfM^{r}_{A}(\grn) \\
    @VVV @VVV \\
   \Spec A'[\grn'^{-1}] @>>> \Spec A[\grn^{-1}].
  \end{CD}
\end{equation}
The morphism $I_b$ is finite, because it is proper by \cite[Proposition~3.2]{Breuer} and affine by \cite[Theorem~2.2]{HartlHendler}. It is not surjective when $F'\ne F$ by reason of dimensions.

\begin{lemma}\label{S.7}
  Suppose that the Satake compactification $\ol\bfM$ of $\bfM:=\bfM^r_A(\grn)$
  over $A[\grn^{-1}]$
  exists. Then the Satake compactification $\ol{\bfM'}$ of 
  $\bfM':=\bfM^{r'}_{A'}(\grn')$ 
  over $A[\grn^{-1}]$ exists and the morphism $I_b$ extends uniquely 
  to a finite morphism $\ol I_b: \ol{\bfM'} \to \ol
  {\bfM}$. Moreover, if the dual $\omega_A$ of the Lie algebra of the universal Drinfeld module over $\ol\bfM$ is ample, then also the dual $\omega_{A'}$ of the Lie algebra of the universal Drinfeld module over $\ol{\bfM'}$ is ample.
\end{lemma}

\begin{proof}
  The statement of Lemma~\ref{S.7} for the generic fiber 
  is proved in \cite[Lemma
  4.5]{pink}, and the proof also works in the present situation. We
  sketch the construction for the reader's convenience. 
  Let $I_b(\bfM')^{\rm zar}$ be the Zariski closure of $I_b(\bfM')$ in
  $\ol \bfM$, and 
  let $\ol{\bfM'}$ be the
  normalization of $I_b(\bfM')^{\rm zar}$ in the
  function field of $\bfM'$. 
  $\ol{\bfM'}$ is a normal integral
  scheme flat over $A'[\grn'^{-1}]$, using \cite[Proposition~III.9.7]{hartshorne} again. 
  Since $I_b(\bfM')^{\rm zar}$ is excellent we have a natural finite morphism 
  $\ol I_b: \ol{\bfM'} \to \ol \bfM$ extending $I_b$ by
  \cite[IV$_2$, Scholie~7.8.3 (ii), (v)]{EGA}.
  The pull-back of the universal family on $\ol \bfM$ gives a
  generalized Drinfeld $A$-module $(\ol{E'},\wt \varphi)$ over $\ol
  {\bfM'}$ where
  the $A$-action $\wt \varphi$ extends to the $A'$-action $\varphi'$ 
  on the open subscheme $\bfM'$. We can view $\varphi'_{a'}$, for
  $a'\in A'$, as an endomorphism of $(\ol{E'},\wt \varphi)$ over
  $\bfM'$, which extends to an endomorphism $\ol \varphi'_{a'}$ 
  over $\ol {\bfM'}$ by \cite[Proposition~3.7]{pink}. 
  Since the morphism $\ol {\bfM'}\to \ol {\bfM}$ is finite, 
  it follows that $(\ol {E'}, \ol {\varphi}')$ is a weakly separating 
  family on $\ol {\bfM'}$. That $\omega_{A'}$ is ample on $\ol{\bfM'}$ follows from the equality $\omega_{A'}=(\ol I_b)^*\omega_A$ by \cite[II, Proposition~5.1.12 and Corollary~4.6.6]{EGA}. \qed
\end{proof}

\begin{lemma}\label{S.8}
  If $A=\Fq[t]$, then a projective Satake compactification $\ol{\bfM}^r_{K(t)}$ for $\bfM^r_{K(t)}$ exists for any $r\ge 1$ and the dual $\omega$ of the Lie algebra of the universal Drinfeld module over $\ol{\bfM}^r_{K(t)}$ is ample.
\end{lemma}
\begin{proof}
  See Proposition~\ref{S.10} below for a  more detailed statement and the
  proof. \qed 
\end{proof}

\npr {\bf Proof of Theorem~\ref{S.2}.} 
(1) Choose $\grn:=\grn_K^m=t A$ for some $m\in \bbN$. Put
$A_0:=\Fq[t]\subset A$ 
and $F_0:=\Frac(A_0)$. The moduli scheme $\bfM^r_{K(\grn)}$ is defined
over $A[\grn_K^{-1}]$ and one has a morphism $I_b: \bfM^r_{K(\grn)}\to
\bfM^{r[F:F_0]}_{A_0, K(t)}$ as in (\ref{eq:S.1}). 
By Lemma~\ref{S.8}, the moduli scheme 
$\bfM^{r[F:F_0]}_{A_0, K(t)}$ admits a projective 
Satake compactification over $A_0[1/t]$. 
It follows from Lemma~\ref{S.7} that 
the moduli scheme $\bfM^r_{K(\grn)}$ admits a projective 
Satake compactification over $A[\grn_K^{-1}]$. 
As $K(\grn)\subset K$, it follows from
Lemma~\ref{S.6} that the moduli scheme $\bfM^r_K$ admits a projective 
Satake compactification over $A[\grn_K^{-1}]$. This proves part (1) of
Theorem~\ref{S.2}. 

Part (2) follows from Lemma~\ref{S.6}. Namely, let $\grn:=\grn_{\wt K}$, so that $K(\grn)\subset\wt K$, and let $\ol{\bfM}^r_{K(\grn)}$ be the Satake compactification of $\bfM^r_{K(\grn)}$. By Lemma~\ref{S.6} the quotients $\ol{\bfM}^r_{\wt K}:=\ol{\bfM}^r_{K(\grn)}\big/\bigl(\wt K/K(\grn)\bigr)$ and $\ol{\bfM}^r_{K}\otimes_{A[\grn_K^{-1}]}A[\grn_{\wt K}^{-1}]:=\ol{\bfM}^r_{K(\grn)}\big/\bigl(K/K(\grn)\bigr)$ are the Satake compactifications of $\bfM^r_{\wt K}$ and $\bfM^r_K\otimes_{A[\grn_K^{-1}]}A[\grn_{\wt K}^{-1}]$, respectively. The generalized Drinfeld $A$-module $(\olE_{K(\grn)},\ol\phi_{K(\grn)})$ descends to generalized Drinfeld $A$-modules $(\olE_{\wt K},\ol\phi_{\wt K})$ and $(\olE_K,\ol\phi_K)$ on $\ol{\bfM}^r_{\wt K}$ and $\ol{\bfM}^r_{K}\otimes_{A[\grn_K^{-1}]}A[\grn_{\wt K}^{-1}]$, respectively. The forgetful quotient morphisms $\pi_{K(\grn),\wt K}$ and $\pi_{K(\grn),K}$ are finite surjective and open by Proposition~\ref{S.4}(1). Therefore, also $\pi_{\wt K,K}$ is finite surjective an open and satisfies $\pi_{\wt K,K}^*(\olE_K,\ol\varphi_K)=(\olE_{\wt K},\ol\varphi_{\wt K})$ and $\pi_{\wt K,K}^*(\omega_K)=\omega_{\wt K}$.
\qed

\begin{prop}\label{reduced}
Suppose that $K$ is fine. At every place $v\nmid \grn_K$ the fiber $\ol{\bfM}^r_K \otimes_{A[\grn_K^{-1}]} \F_v$ is
  geometrically reduced. 
\end{prop}

\begin{proof}
  Write $\scrM^r_K=\bfM^r_K \otimes_{A[\grn_K^{-1}]} \F_v$ and $\ol \scrM^r_K=\ol{\bfM}^r_K \otimes_{A[\grn_K^{-1}]} \F_v$. Since the residue field $\F_v$ is perfect, it suffices to show that $\ol\scrM^r_K$ is reduced. Using Serre's criterion \cite[IV$_2$, Proposition~5.8.5]{EGA} we prove this by showing that every point $x\in\ol\scrM^r_K$ has an open affine neighborhood $\Spec B\subset\ol\scrM^r_K$ for which the ring $B$ satisfies conditions $(R_0)$ and $(S_1)$:

$(R_0)$ For every minimal prime $\grp\subset B$, the local ring $B_\grp$ is regular.

$(S_1)$ Every prime ideal $\grp\subset B$ with ${\rm depth}\,B_\grp=0$ has codimension $0$.

Let $\Spec\bfB$ be an open affine neighborhood of $x$ in $\ol{\bfM}^r_K$. Since $\bfM^r_K$ is dense in $\ol{\bfM}^r_K$ the dimension of the integral domain $\bfB$ is $r$ by Proposition~\ref{2.5}(2). Also the dimension of $\bfB_{A_{(v)}}:=\bfB\otimes_{A[\grn_K^{-1}]} A_{(v)}$ is $r$ where $A_{(v)}$ is the localization of $A$ at the place $v$. Let $a\in A$ be an element with $v(a)=1$ and let $B:=\bfB\otimes_{A[\grn_K^{-1}]} \F_v=\bfB_{A_{(v)}}/\gamma(a)\bfB_{A_{(v)}}$. Then $\Spec B$ is an affine open neighborhood of $x$ in $\ol\scrM^r_K$. Let $\grp\subset B$ be a minimal prime ideal and let $\grP\subset \bfB$ be the preimage of $\grp$ in $\bfB$. Then $\grP\bfB_{A_{(v)}}$ is the preimage of $\grp$ in $\bfB_{A_{(v)}}$ and is minimal among primes containing $\gamma(a)$. So $\bfB_\grP$, which equals the localization of $\bfB_{A_{(v)}}$ at $\grP\bfB_{A_{(v)}}$, has dimension $\le 1$ by Krull's principal ideal theorem \cite[Theorem~10.2]{eisenbud}. On the other hand $\dim B=r-1$, because $\gamma(a)$ is not a unit and not a zero-divisor in $\bfB_{A_{(v)}}$ due to the flatness of $\ol{\bfM}^r_K$ over $A[\grn_K^{-1}]$. Therefore, $\dim \bfB/\grP=\dim B/\grp \le\dim B=r-1$. Since $\bfB$ is a finitely generated $\Fq$-algebra we have $r=\dim\bfB=\dim \bfB_\grP+\dim\bfB/\grP$ by \cite[Corollary~13.4]{eisenbud}, and hence $\dim B/\grp=r-1$. Since $\dim(\ol\scrM^r_K \setminus \scrM^r_K)\le r-2$ by Proposition~\ref{2.7}, the point $\grp\in\Spec B\subset\ol\scrM^r_K$ must lie in $\scrM^r_K$. And since $\scrM^r_K$ is smooth over $\F_v$ the local ring $B_\grp$ is regular, proving $(R_0)$.

To prove $(S_1)$ let $\grp\subset B$ be a prime ideal with ${\rm depth}\,B_\grp=0$ and let $\grP\subset \bfB$ be the preimage of $\grp$ in $\bfB$. Since $\gamma(a)$ is a non-zero-divisor in $\bfB_\grP$ and $B_\grp=\bfB_\grP/(\gamma(a))$, we have ${\rm depth}\, \bfB_\grP=1$ by \cite[0$_{\rm IV}$, Proposition~16.4.6]{EGA}. Since $\bfB$ is normal we conclude $\dim \bfB_\grP=1$, and hence $\dim B_\grp=0$ as desired. This proves $(S_1)$ and the proposition.
\qed
\end{proof}

In the remainder of this section we let $A=\Fq[t]$,
$F=\Fq(t)$ and $\grn=(t)$. Let $r\ge 1$ be a positive integer, and 
write $\bfM=\bfM^r(t)$. The aim is to
construct the Satake compactification of $\bfM$ over $A[1/t]$ and
prove Lemma~\ref{S.8}. It turns
out that methods and proofs in the construction for the Satake
compactification $\ol M$ of the generic fibre $M:=\bfM \otimes_{A[1/t]}
F$ already suffice for our purpose. 

Set $V_r:=\Fq^r$ and identify it with the $\Fq$-vector space
$(t^{-1}A/A)^r$. Put $V^0_r:=V_r \setminus \{0\}$. 
Let $(E, \varphi,\lambda)$ be a Drinfeld $A$-module
of rank $r$ with level-$t$ structure 
over an $A[1/t]$-scheme $S$. Then the level structure
$\lambda$ induces an $\Fq$-linear map $\lambda:V_r\to E(S)$ which is
\emph{fiber-wise injective}, i.e. for any point $s\in S$ the induced map
$V_r\to E(s)$ is injective. In particular, for any $v\in V^0_r$, the
section $\lambda(v)$ is nowhere zero.  

\begin{lemma}\label{S.9}
  For any line bundle $E$ over an $A[1/t]$-scheme $S$ and any 
  fiber-wise injective
  $\Fq$-linear map $\lambda:V_r\to E(S)$, there exists a unique
  homomorphism $\varphi: A\to \End(E)$ turning $(E,\varphi, \lambda)$
  into a Drinfeld $A$-module of rank $r$ with level-$t$ structure over
  $S$. 
\end{lemma}

\begin{proof}
  The assertion is local on $S$. Thus, we may assume that $S=\Spec R$
  is connected and $E=\bbG_{a,S}=\Spec R[X]$. For any 
  $v\in V^0_r:=V_r \setminus \{0\}$, one has $\lambda(v)\in R^\times$, as it is
  non-zero everywhere. Put $f(X):=\prod_{v\in V_r}
  (X-\lambda(v))$. By \cite[Cor. 1.2.2]{goss}, $f\in \End(\bbG_{a,S})$ 
  is an $\Fq$-linear endomorphism of degree $q^r$ in $X$ over
  $R$. Thus, $\ker f\subset S$ is a finite constant
  group over $S$ of order $q^r$, which is the union of the image of the
  sections $\lambda(v)$. 
  Note that $(-1)^{q^r-1}=+1$ if $q$ is odd, and also if $q$ is even 
  when $-1=+1$. Therefore,
  \begin{equation}
    \label{eq:S.3}
    \begin{split}
    \varphi_t:&=\gamma(t)  \prod_{v\in V^0_r} \lambda(v)^{-1}
  f(X)=\gamma(t) X \prod_{v\in V^0_r} (1-\lambda(v)^{-1} X) \\
  &=\gamma(t)\tau^0 + \dots + \gamma(t) \prod_{v}
  \lambda(v)^{-1} \tau^r, \quad (\tau^i=X^{q^i})
    \end{split}
  \end{equation}
 defines a Drinfeld $A$-module of rank $r$ over $S$. 
 Note that $\lambda:V_r\isoto (\ker f)(S)=\varphi[t](S)\subset R$ is
 an $\Fq$-linear isomorphism, which is also $A$-linear as $t$
 annihilates both sides. Thus,  
 $\lambda$ is a level-$t$ structure on $(\bbG_{a,S},\varphi)$.
 Note that $\varphi_t$ in
 (\ref{eq:S.3}) is the
 unique polynomial such that the coefficient of $X$ is $\gamma(t)$ and $(\ker
 \varphi_t)(S)=\lambda(V_r)$. Therefore, the homomorphism 
 $\varphi$ is uniquely determined
 by $\lambda$. \qed      
\end{proof}

Write $\bbP^{r-1}:=\bbP^{r-1}_{\Fq}=\Proj S_r$, where
$S_r:=\Fq[x_0,\dots, x_{r-1}]$ is the graded polynomial ring over $\Fq$ 
with degree one on each $x_i$. 
We identify $S_r$ with the symmetric
algebra $\Sym_{\Fq} V_r$ of $V_r=\Fq^r$ 
by sending $x_0,\dots, x_{r-1}$ to 
the standard basis of $V_r$.
As is well known \cite[Proposition~II.7.12]{hartshorne}, 
$\bbP^{r-1}$ represents the
functor that associates to any $\Fq$-scheme $T$ the set of isomorphism
classes of $(E,e_0,e_1,\dots, e_{r-1})$ consisting of a line bundle
$E$ over $T$ and sections $e_0,\dots, e_{r-1}\in E(T)$ that generate
$E$. Given such a tuple $(E,e_0, \dots, e_{r-1})$, one associates 
an $\Fq$-linear map $\lambda:V_r\to E(T)$, sending $x_i$ to $e_i$,
which induces a surjective map $\calO_T^r\to E$. 
Clearly, the datum $(e_0, \dots, e_{r-1})$ is determined by $\lambda$.  
The universal family on $\bbP^{r-1}$ is $(\calO_{\bfP^{r-1}}(1), x_0,\dots,
x_{r-1})$, or equivalently by $(\calO_{\bfP^{r-1}}(1), \lambda_{\bbP^{r-1}})$, where $\lambda_{\bbP^{r-1}}: V_r\to
\calO_{\bbP^{r-1}}(1)(\bbP^{r-1})$ is the identity map.

Let $\Omega_r$ be the open subscheme of $\bbP^{r-1}$ obtained by
removing all $\Fq$-rational hyperplanes. By definition, $\Omega_r$ is
the largest open subset $U$ such that $v$ is nowhere zero on $U$ 
for any $v\in V_r^0$, or equivalently, the restriction on $U$ of  
$\lambda_{\bbP^{r-1}}: V_r \to
\calO_{\bbP^{r-1}}(1)(U)$ is fiber-wise injective. Thus, $\Omega_r$
represents the functor $\calF$ which associates to each 
$\Fq$-scheme $T$ the
pairs $(E,\lambda)$ in $\bbP^{r-1}(T)$ with the property that
 the $\Fq$-linear map 
$\lambda:V_r\to E(T)$ is fiber-wise injective. 
On the other hand,
by Lemma~\ref{S.9}, the functor $\calF$ restricted
to the category of $A[1/t]$-schemes is the same as the representable
functor associated to $\bfM$. Thus, one obtains 
$\bfM=\Omega_{r,A[1/t]}:=\Omega_r\otimes_{\Fq} A[1/t]$.

Let $E=E_{\Omega_{r},A[1/t]}$ be the
line bundle over $\Omega_{r,A[1/t]}$ corresponding to the invertible sheaf
$\calO_{\bbP^{r-1}}(1)\otimes_{\Fq} A[1/t]$ on $\Omega_{r,A[1/t]}$. 
The map $\lambda_{\bbP^{r-1}}$ induces a fiber-wise injective
$\Fq$-linear map $\lambda:V_r\to E(\Omega_{r,A[1/t]})$.
Let $\varphi:A\to \End(E)$ be the homomorphism defined by 
\begin{equation}
  \label{eq:S.5}
  \varphi_t=\gamma(t)X \prod_{v\in V^0_r} \left(1-\frac{X}{\lambda(v)}\right)=
  \gamma(t) \tau^0+\sum_{i=1}^r \varphi_{t,i} \tau^i,  
\end{equation}
where $\tau^i=X^{q^i}$ and 
\begin{equation}
  \label{eq:S.6}
  \varphi_{t,i}=\sum_{v_1,\dots, v_{q^i-1}\in V_r^0, v_i\neq v_j}
\frac{\gamma(t)}{\lambda(v_1) \cdots \lambda(v_{q^i-1})} \in \Gamma(\Omega_r,
E^{\otimes(1-q^{i})}),
\end{equation}
where we use again that $(-1)^{q^i-1}=+1$ if $q$ is odd, and also if $q$ is even 
  when $-1=+1$. From Lemma~\ref{S.9}, $(E,\varphi,\lambda)$ is
the universal family on $\bfM$.

Denote the quotient field of $S_r=\Fq[x_0,\dots,x_{r-1}]$ 
by $K_r$. Let $R_r$ be the
$\Fq$-subalgebra of $K_r$ generated by $1/v$ for all $v\in V^0_r$. 
Impose a graded structure on $R_r$ by assigning degree one to each
$1/v$, and define $Q_r:=\Proj R_r$. We change the graded structure on $S_r$ by now assigning degree $-1$ to each $x_i$ and to each $v\in V^0_r$. Let $RS_r$ be the graded 
$\Fq$-subalgebra of 
$K_r$ generated by $R_r$ and $S_r$, and $RS_{r,0}\subset RS_{r}$ be the $\Fq$-subalgebra consisting of homogeneous elements of degree zero.  
Then $\Omega_r=\Spec(RS_{r,0})$ is
the open subscheme in $Q_r$ by removing the hyperplane sections
defined by $1/x_i$ for $i=0,\dots, r-1$.

\begin{prop}\label{S.10}
  The scheme $Q_{r,A[1/t]}=Q_r\otimes_{\Fq} A[1/t]$ 
  is a projective Satake compactification
  of $\bfM$. It is Cohen-Macaulay. The dual $\omega$ of 
  the relative Lie algebra of the universal family
  on $Q_{r,A[1/t]}$ is $\calO_{Q_{r,A[1/r]}}(1)$, in particular,
  $\omega$ is very ample relative to $A[1/t]$.
\end{prop}
\begin{proof}
$Q_r$ is Cohen-Macaulay by results of Pink and Schieder \cite[Theorem~1.11]{pink-schieder}. This implies that the morphisms $Q_r\to\Spec\Fq$ and $Q_{r,A[1/t]}\to\Spec\Fq$ are Cohen-Macaulay by \cite[IV$_2$, Proposition~6.8.3]{EGA}. Therefore, $Q_{r,A[1/t]}$ is Cohen-Macaulay.
Let $\olE$ be the line bundle whose sheaf of sections
is $\calO_{r,A[1/t]}(-1)$. 
Define the homomorphism $\ol \varphi:A\to \End(\olE)$ by setting
\begin{equation}
  \label{eq:S.7}
  \ol \varphi_t=\gamma(t) +\sum_{i=1}^r \ol \varphi_{t,i} \tau^i, \quad \ol
  \varphi_{t,i}=\varphi_{t,i} \in  \calO_{Q_{r,A[1/t]}}(q^i-1)=\olE^{\otimes(1-q^i)}(Q_{r,A[1/t]}), 
\end{equation}
where $\varphi_{t,i}$ is defined in (\ref{eq:S.6}). By \eqref{eq:S.6},
if $\varphi_{t,i}(x)=0$ at some point $x$ for every $i=1,\dots, r$  
then $\tfrac{1}{v}(x)=0$ for every $v\in V^0_r$, which is not possible. Thus,  
$(\olE, \ol \varphi)$ has rank $r\ge 1$ everywhere and 
it is a generalized Drinfeld $A$-module of
rank $\le r$ on $Q_{r,A[1/t]}$ which extends the universal family
$(E,\varphi)$ over $\bfM$. The proof of \cite[Proposition 7.2]{pink} shows
that $(\olE,\ol \varphi)$ is weakly separating. Thus, $Q_{r,A[1/t]}$
is a projective Satake compactification of $\bfM$. Note that the
relative Lie algebra $\Lie(\olE)$ is $\calO_{Q_{r,A[1/r]}}(-1)$ and its
dual $\omega:=\Lie(\olE)^\vee$ is $\calO_{Q_{r,A[1/r]}}(1)$,
particularly $\omega$ is very ample relative to $A[1/t]$. \qed
\end{proof}

\begin{remark}\label{S.11}
  Note that the generic fiber $\ol M^r_K:=\ol \bfM^r_K\otimes_{A[\grn_K^{-1}]} F$ satisfies the
  characterizing properties of the Satake compactification 
  \cite[Definition 4.1 (a) and (b)]{pink}. Thus, the generic fiber 
  $\ol M^r_K$ of $\ol \bfM^r_K$ is the
  Satake compactification of $M^r_K$ constructed by Pink. 
\end{remark}

\begin{remark}\label{RemCM}
In a forthcomming work Pink~\cite{PinkDetMorph} shows that for
$A=\Fq[t]$ and level $K=K(t^n)$ the Satake compactification
$\ol{\bfM}^r_K$ is Cohen-Macaulay. More generally, we give the
following small evidence for Conjecture~\ref{ConjCM}. 
Assume that the boundary $C:=\ol{\bfM}^r_K\setminus \bfM^r_K$ with the
reduced scheme structure is \emph{$F$-split}. This means that the
(injective) Frobenius homomorphism $\O_{C} \hookrightarrow ({\rm
  Frob}_{p,C}){}_*\O_{C}$ has an \emph{$F$-splitting}, that is a
section as $\O_{C}$-modules, where ${\rm Frob}_{p,C}\colon C \to C$ is
the absolute $p$-Frobenius of $C$ which is the identity on points and
the $p$-power map on the structure sheaf; see
\cite{MethaRamanathan,BrionKumar}. If $C$ is $F$-split, then
$\ol{\bfM}^r_K$ is Cohen-Macaulay by \cite[discussion before Lemma~2.7
on page~727]{enescu-hochster} and \cite[Theorem~A.3]{HMS14}, because
$\bfM^r_K$ is regular and $C$ is cut out locally by a
non-zero-divisor. 

Maybe by the following approach one can prove that $C$ is $F$-split,
at least for the cofinal system $K=K(\grn)$ of principal level
subgroups for which $\bfM^r_K$ is smooth, see
Theorem~\ref{ThmDrinfeld}. Assume that a smooth (e.g.~toroidal)
compactification $X:=\ol{\bfM}^{r,\rm sm}_K$ of $\bfM^r_K$ has been
constructed and let $\wt C:=\ol{\bfM}^{r,\rm sm}_K\setminus\bfM^r_K$
be the boundary with the reduced scheme structure. Assume further that
$(X,\wt C)$ are \emph{compatibly $F$-split}. This means that $X$ has
an $F$-splitting $s\colon({\rm Frob}_{p,X})_*\O_X\onto\O_X$ which
induces an $F$-splitting $s\colon({\rm Frob}_{p,\wt C})_*\O_{\wt
  C}\onto\O_{\wt C}$ of $\wt C$. Maybe such a compatible $F$-splitting
can be constructed from an explicit combinatorial description of the
boundary components of $\wt C\subset\ol{\bfM}^{r,\rm sm}_K$.
Then under the canonical proper, birational morphism $f\colon X=\ol{\bfM}^{r,\rm sm}_K \to \ol{\bfM}^r_K=:Y$ we have $f_*\O_X = \O_Y$, because $\ol{\bfM}^r_K$ is normal, and hence $(Y,C)$ are compatibly $F$-split, too, by \cite[Proposition~4]{MethaRamanathan}. In particular, $C$ is $F$-split. 
\end{remark}

\section{Drinfeld modular forms and Hecke operators}
\label{sec:MF}

\subsection{Drinfeld modular forms over $A[\grn^{-1}]$}
\label{sec:MF.1}
Let $A,F,\infty$ be as in previous sections, and let $G:=\GL_r$. 
For a fine subgroup $K$,
let $(\olE_K,\ol \varphi_K)$ be the universal family over the Satake
compactification $\ol \bfM^r_{K}$ over $A[\grn_K^{-1}]$. 
By Theorem~\ref{S.2},
\[ \omega_K:=\Lie(\olE_K)^\vee \]
is an ample invertible sheaf on $\ol \bfM^r_K$. We also write $\omega_K$
for its restriction on $\bfM_K^r$.  

\begin{defn}\label{MF.1}
  (1) For any integer $k\ge 0$, 
  fine open compact subgroup $K\subset K(1)=\GL_r(\wh A)$
  and $A[\grn_K^{-1}]$-algebra $L$, denote by
  \begin{equation}\label{eq:MF.1}
    M_k(r,K,L):=H^0(\ol \bfM^r_{K}\otimes_{A[\grn_K^{-1}]} L,\omega_K^{\otimes k}\otimes {L})
  \end{equation}
the $L$-module of \emph{algebraic Drinfeld modular forms of rank $r$, weight
$k$, level $K$ over $L$}. 
The terminology ``algebraic'' is meant to distinguish them from the
modular forms which are defined using the analytic theory for $L=\C_\infty$.

(2) Denote by
  \begin{equation}\label{eq:MF.2}
    M(r,K,L):=\bigoplus_{k\ge 0} M_k(r,K,L)
  \end{equation}
the \emph{graded ring of algebraic Drinfeld modular forms of rank $r$, level
$K$ over $L$}.  
\end{defn}

\begin{remark}\label{Remk<0}
Note that for $k<0$ the analogously defined $L$-module $M_k(r,K,L)$ is zero by the following well-known lemma (applied with $Y=\Spec L$) and Proposition~\ref{reduced}.
\end{remark}

\begin{lemma}\label{Lemma_DualOfAmple}
Let $\pi\colon X\to Y$ be a proper and flat morphism of noetherian schemes and let $\calL$ be an invertible sheaf on $X$ which is relatively ample over $Y$. Assume that for every $y\in Y$ the fiber $X_y$ is reduced and all irreducible components of $X_y$ have dimension at least one. Then $\pi_*\calL^{\otimes k}=(0)$ for every $k<0$.
\end{lemma}

We include a proof, because we could not find a reference. Note that the condition that the fibers are reduced is crucial, as one sees from Example~\ref{Ex_DualOfAmple} below.

\begin{proof}
Fix a $k<0$. By the theory of cohomology and base change, it suffices to treat the case when $Y$ is the spectrum of a field. More precisely, by \cite[Chapter~II, \S\,5, Theorem on page~44]{MumfordAbVar} there is a complex of finite locally free sheaves $P^\bullet\colon 0\longrightarrow P^0\xrightarrow{\;d^0\,} P^1\xrightarrow{\;d^1\,}\ldots$ on $Y$ such that for every $Y$-scheme $Y'$
\[
(\pi\times\id_{Y'})_*(\calL^{\otimes k}\otimes_{\calO_Y}\calO_{Y'}) \simeq H^0(P^\bullet\otimes_{\calO_Y}\calO_{Y'}):=\ker(d^0\otimes_{\calO_Y}\calO_{Y'})\,.
\]
Locally on $Y$ we can choose bases of $P^0$ and $P^1$ and write $d^0$ as an $n_1\times n_0$-matrix, where $n_i$ is the rank of $P^i$. If at some point $y\in Y$ 
\[
(0)=H^0(X_y,\calL^{\otimes k}\otimes_{\calO_Y}\kappa(y))=(\pi\times\id_{y})_*(\calL^{\otimes k}\otimes_{\calO_Y}\kappa(y)) \simeq H^0(P^\bullet\otimes_{\calO_Y}\kappa(y)), 
\]
then $d^0\otimes \kappa(y)$ is injective. This means that $n_0\le n_1$ and there is an $n_0\times n_0$-minor in the matrix $d^0$ whose image in $d^0\otimes \kappa(y)$ has invertible determinant. Then the determinant of this minor is already invertible in $\calO_{Y,y}$ and so $d^0\otimes\calO_{Y,y}$ is injective. If this holds at every point $y\in Y$, then $d^0$ is injective and $\pi_*\calL^{\otimes k}=(0)$.

Note that $\calL^{\otimes k}\otimes_{\calO_Y}\kappa(y)$ on $X_y$ is relatively ample over $\kappa(y)$ by \cite[II, Proposititon~4.6.13 (iii)]{EGA}. So we may replace $Y$ by $\Spec\kappa(y)$ for a point $y\in Y$, and thus assume that $Y$ is the spectrum of a field. Then we must show that $H^0(X,\calL^{\otimes k})=(0)$. Assume that there is a non-zero global section $0\ne s\in H^0(X,\calL^{\otimes k})$. Let $U\subset X$ be an open subset with $0\ne s|_U$; use \cite[I$_{\rm new}$, Lemma~9.7.9.1]{EGA}. By shrinking $U$ we may assume that it is contained in exactly one irreducible component of $X$. Since $X$ is reduced, the scheme theoretic closure $\ol U$ of $U$ in $X$ is still reduced and also irreducible. It contains $U$ as an open subscheme. Let $i\colon \ol U\hookrightarrow X$ be the corresponding closed immersion. Then $0\ne i^*s\in H^0(\ol U,i^*\calL^{\otimes k})$ is a regular section and defines an effective Cartier divisor $D$ on $\ol U$ by \cite[Proposition~11.32]{GW}. Since $\dim\ol U\ge 1$ and the support \[
{\rm Supp}(D):=\{x\in\ol U\;|\; D_x\ne 1\}=\{x\in\ol U\;|\;\calO_{\ol U,x}\cdot (i^*s)_x\subsetneq (i^*\calL)^{\otimes k}_x\}
\]
of $D$ is strictly contained in $\ol U$ by \cite[Lemma~11.33]{GW}, we
may choose a proper curve $C\subset \ol U$ (that is, an irreducible
and reduced closed subscheme of dimension one) which is not contained
in ${\rm Supp}(D)$. Let $\wt C$ be the normalization of $C$ and let
$f\colon \wt C\to X$ be the induced map. Then $f$ is finite and
$f^*\calL$ is ample on $\wt C$ by \cite[II, Corollaire~4.6.6 and
Proposition~5.1.12]{EGA}. Therefore $\deg (f^*\calL)>0$ by
\cite[Corollary~IV.3.3]{hartshorne} (noting the degree remains the
same after a field extension base change) and $k\cdot\deg (f^*\calL)<0$,
because $k<0$. On the other hand, since $C\not\subset{\rm Supp}(D)$
and $i^*s$ generates $i^*\calL^{\otimes k}$ on $\ol U\setminus{\rm Supp}(D)$,
we have $0\ne f^*s\in H^0(\wt C,f^*\calL^{\otimes k})$. Therefore,
$k\cdot\deg (f^*\calL)=\deg (f^*\calL^{\otimes k})\ge0$ by
\cite[Lemma~IV.1.2]{hartshorne}. This is a contradiction and proves
the lemma. \qed 
\end{proof}

\begin{eg}\label{Ex_DualOfAmple}
We show that the condition on the reduced fibers in Lemma~\ref{Lemma_DualOfAmple} cannot be dropped. Let $k$ be a field and let $X=\Proj k[S,T,U]/(TU,U^2)=V(TU,U^2)\subset\Proj k[S,T,U]=\bbP^2_k$. Then $X$ is non-reduced at the point $V(T,U)$. The line bundle $\calO(1)$ is ample on $X$. But its dual $\calO(-1)$ has the non-zero global section $\tfrac{U}{S^2}\in H^0(X,\calO(-1))$, which vanishes outside $V(T,U)$.
\end{eg}

Since $\omega_K$ is ample and $\ol \bfM^r_K$ is proper over
$A[\grn_K^{-1}]$, by \cite[Tags 01CV and 01Q1]{stacks} 
there is a canonical isomorphism
\begin{equation}
  \label{eq:MF.3}
  \ol \bfM^r_K \isoto \Proj M(r,K,A[\grn_K^{-1}]).
\end{equation}
Applying the base change theorem for cohomology groups to 
$\Spec L\to \Spec A[\grn_K^{-1}]$, we obtain canonical maps
\begin{equation}
  \label{eq:MF.4}
  M_k(r,K,A[\grn_K^{-1}])\otimes L \to M_k(r,K,L),\quad k=0,1,\dots.
\end{equation}
and these are isomorphisms when $L$ is flat over $A[\grn_K^{-1}]$; see
\cite[Proposition~III.9.3]{hartshorne}. 
We will need the following well known 
 
\begin{prop}\label{MF.2}
  Let $Y$ be a Noetherian scheme, $f\colon X\to Y$ a projective morphism,
  $\calF$ a coherent $\calO_X$-module which is flat over $Y$. Assume
  that for some $i$ the cohomology in the fiber 
  \begin{equation}
    \label{eq:MF.45}
    H^i(X_y,\calF\otimes k(y))=0
  \end{equation}
for all points $y\in Y$. Then we have
\begin{equation}
  \label{eq:MF.46}
  (R^{i-1} f_* \calF)\otimes_{\calO_Y} B \simeq H^{i-1}(X\times_Y \Spec B,
  \calF\otimes B)
\end{equation}
for any $Y$-scheme $\Spec B$.  
\end{prop}
\begin{proof}
  This is proved in the same way as \cite[Theorem~III.12.11]{hartshorne} by combining 
  \cite[Chapter~III, Propositions~12.4, 12.5, 12.7 and 12.10]{hartshorne}. \qed
\end{proof}

\begin{cor}\label{MF.3}
  There is a positive integer $k_0$ such that for all $k\ge k_0$, the
  canonical map in \eqref{eq:MF.4} is an isomorphism and $M_k(r,K,A[\grn_K^{-1}])$ is a finite projective $A[\grn_K^{-1}]$-module.
\end{cor}
\begin{proof}
Let $f:X:=\ol \bfM^r_K\to Y:=\Spec A[\grn_K^{-1}]$.
Since $\omega_K$ is ample, by \cite[Proposition~III.5.3]{hartshorne} there
exists an integer $k_0$ such that $R^if_*(\omega_K^{\otimes
  k})=0$ on $Y$ for all $i>0$ and $k\ge k_0$. 
We now prove that for each point $y\in
Y$ the natural map  
\[ \theta^i(y): R^if_*(\omega_K^{\otimes
  k})\otimes_{\calO_Y} k(y)\to H^i(X_y, \omega_K^{\otimes k}\otimes k(y)) \]
is surjective for all $i>0$. This is true for $i>\dim X+1$ as the
target is zero. Since
$R^if_*( \omega_K^{\otimes k})=0$ is locally free everywhere for all
$i>0$, by
\cite[Theorem~III.12.11]{hartshorne} 
and by induction on $i$ decreasingly, the map
$\theta^i(y)$ is surjective for $i=\dim X+1,\dots,1$. 
Using $R^1f_*( \omega_K^{\otimes k})=0$ and that $\theta^1(y)$ is
surjective again, we show $H^1(X_y,\omega_K^{\otimes k}\otimes k(y))=0$
for all $y\in Y$. Then \eqref{eq:MF.4} is an isomorphism by 
Proposition~\ref{MF.2}, and the projectivity of $M_k(r,K,A[\grn_K^{-1}])$ follows
from \cite[Theorem~III.12.11]{hartshorne}. \qed
\end{proof}

\begin{lemma}\label{MF.35}
  Let $\wt K\subset K\subset K(1)$ be two fine open compact
  subgroups and let $\pi:=\pi_{\wt K, K}:\ol\bfM^r_{\wt
  K}\to \ol \bfM^r_{K}\otimes_{A[\grn_{K}^{-1}]}A[\grn_{\wt K}^{-1}]$ be the finite cover from Theorem~\ref{S.2}(2).
\begin{enumerate}
\item Let $L$ be an $A[\grn_{\wt K}^{-1}]$-module. Then the map of quasi-coherent sheaves 
  \begin{equation}\label{EqLemma4.4}
\omega_K^{\otimes k}\otimes_{A[\grn_{K}^{-1}]}L \longrightarrow 
  \pi_{\wt K,K\,*}\bigl(\omega_{\wt K}^{\otimes k}\otimes_{A[\grn_{\wt K}^{-1}]}L\bigr)
 = \pi_{\wt K,K\,*}\pi_{\wt K,K}^*\bigl(\omega_K^{\otimes k}\otimes_{A[\grn_K^{-1}]}L\bigr)
\end{equation}
on $\ol{\bfM}^r_{K}\otimes_{A[\grn_{K}^{-1}]}A[\grn_{\wt K}^{-1}]$ is injective.
\item
For any $A[\grn_{\wt K}^{-1}]$-algebra $L$ the pullback map
  $\pi_{\wt K,K}^*: M_k(r,K ,L)\to M_k(r,\wt K,L)$ is
  injective. 
\item Moreover, if $\wt K\lhd K$ is normal and $L$ is a flat $A[\grn_{\wt K}^{-1}]$-algebra, then $\pi_{\wt K,K}$ induces an isomorphism of $L$-modules
\begin{equation}
  \label{eq:MF.65}
  M_k(r,K,L)\isoto M_k(r,\wt K,L)^{K/\wt K}.
\end{equation}
\end{enumerate}
\end{lemma}
\begin{proof}
Statement (2) follows from (1) by taking global sections on $\ol{\bfM}^r_K\otimes_{A[\grn_{K}^{-1}]}A[\grn_{\wt K}^{-1}]$, because $M_k(r,K ,L)=H^0\bigl(\ol{\bfM}^r_K\otimes_{A[\grn_{K}^{-1}]}A[\grn_{\wt K}^{-1}],\omega_K^{\otimes k}\otimes_{A[\grn_{K}^{-1}]}L\bigr)$.

(1) Note that $\omega_K^{\otimes k}$ and $\omega_{\wt K}^{\otimes k}$ are flat over $A[\grn_{\wt K}^{-1}]$ and the direct image functor $\pi_{\wt K,K\,*}$ is exact by \cite[III$_1$, Corollaire~1.3.2]{EGA} as $\pi_{\wt K,K}$ is finite. Therefore, any exact sequence $0\to L'\to L\to L''\to 0$ of $A[\grn_{\wt K}^{-1}]$-modules induces a commutative diagram with exact rows
\[
\xymatrix @C=1.5pc {
0 \ar[r] & \omega_K^{\otimes k}\otimes_{A[\grn_{K}^{-1}]}L' \ar[r] \ar[d] & \omega_K^{\otimes k}\otimes_{A[\grn_{K}^{-1}]}L \ar[r] \ar[d] & \omega_K^{\otimes k}\otimes_{A[\grn_{K}^{-1}]}L'' \ar[r] \ar[d] & 0\; \\
0 \ar[r] & \pi_*\bigl(\omega_{\wt K}^{\otimes k}\otimes_{A[\grn_{\wt K}^{-1}]}L'\bigr) \ar[r] & \pi_*\bigl(\omega_{\wt K}^{\otimes k}\otimes_{A[\grn_{\wt K}^{-1}]}L\bigr) \ar[r] & \pi_*\bigl(\omega_{\wt K}^{\otimes k}\otimes_{A[\grn_{\wt K}^{-1}]}L''\bigr) \ar[r] & 0\,.
}
\]   
Assume that there is a section $w:=\sum_i w_i\otimes\ell_i\in H^0(U,\omega_K^{\otimes k}\otimes_{A[\grn_{K}^{-1}]}L)$ over an open affine subset $U\subset\ol{\bfM}^r_K\otimes_{A[\grn_{K}^{-1}]}A[\grn_{\wt K}^{-1}]$ in the kernel of \eqref{EqLemma4.4}, where $w_i\in H^0(U,\omega_K^{\otimes k})$ and $\ell_i\in L$ for all $i$. Let $L'$ be the $A[\grn_{\wt K}^{-1}]$-submodule of $L$ generated by the finitely many $\ell_i$. Then $w\in H^0(U,\omega_K^{\otimes k}\otimes_{A[\grn_{K}^{-1}]}L')$. Replacing $L$ by $L'$ we reduce to the case that $L$ is a finitely generated module over the Dedekind domain $A[\grn_{\wt K}^{-1}]$. By the structure theory of such modules we have $L=P\oplus\bigoplus_j A[\grn_{\wt K}^{-1}]/\grp_j^{n_j}$ for a finite projective $A[\grn_{\wt K}^{-1}]$-module $P$ and maximal ideals $\grp_i\subset A[\grn_{\wt K}^{-1}]$ and integers $n_j>0$. Using the commutative diagram above, we thus reduce to the cases where $L=A[\grn_{\wt K}^{-1}]/\grp^n$ or where $L=P$ is finite projective over $A[\grn_{\wt K}^{-1}]$. 

In the first case we choose a uniformizer $z$ of $\grp$ and consider the exact sequence $0 \to A[\grn_{\wt K}^{-1}]/\grp^{n-1}\xrightarrow{\; z\cdot\,} A[\grn_{\wt K}^{-1}]/\grp^n \longrightarrow A[\grn_{\wt K}^{-1}]/\grp \to 0$. Using the diagram again we reduce to the case where $L=A[\grn_{\wt K}^{-1}]/\grp=\F_v$ for the place $v$ of $A$ corresponding to $\grp$. For $L=\F_v$ the sheaf $\omega_K^{\otimes k}\otimes_{A[\grn_K^{-1}]}L$ is (the pushforward to $\ol{\bfM}^r_K\otimes_{A[\grn_{K}^{-1}]}A[\grn_{\wt K}^{-1}]$ of) an invertible sheaf on $\ol\scrM_K^r:=\ol\bfM^r_K \otimes_{A[\grn_K^{-1}]} \F_v$. We consider the kernel sheaf
\[
\calI := \ker\bigl(\calO_{\ol\scrM^r_K} \longrightarrow \pi_{\wt K,K\,*}\pi_{\wt K,K}^*\calO_{\ol\scrM^r_K}\bigr)\,.
\]
The vanishing locus of $\calI$ in $\ol\scrM^r_K$ is the scheme theoretic image of $\pi_{\wt K,K}\colon \ol\scrM^r_{\wt K}\to\ol\scrM^r_K$. Since the set theoretic image equals the entire space $\ol\scrM^r_K$, which is the vanishing locus of the zero ideal, the radicals coincide $\sqrt{\calI}=\sqrt{(0)}$. But the latter equals $(0)$, because $\ol\scrM^r_K$ is reduced by Proposition~\ref{reduced}. This shows that $\calI=(0)$. Tensoring with the flat $\calO_{\ol\scrM^r_K}$-module $\omega_K^{\otimes k}\otimes_{A[\grn_K^{-1}]}L$ proves the injectivity of \eqref{EqLemma4.4} for $L=\F_v$.

We treat the second case, where $L$ is a finite projective $A[\grn_{\wt K}^{-1}]$-module simultaneously to assertion (3). By the construction of $\ol\bfM_{\wt K}^r$ as $\ol\bfM_{K(\grn_{\wt K})}^r\big/\bigl(\wt K/K(\grn_{\wt K})\bigr)$ in Lemma~\ref{S.6}(2) we have an exact sequence of sheaves on $\ol\bfM_{\wt K}^r$
\[
0\longrightarrow \calO_{\ol\bfM_{\wt K}^r} \longrightarrow \pi_{K(\grn_{\wt K}),\wt K\,*}\calO_{\ol\bfM_{K(\grn_{\wt K})}^r} \overset{\textstyle\longrightarrow}{\longrightarrow} \prod_{\wt K/K(\grn_{\wt K})}\pi_{K(\grn_{\wt K}),\wt K\,*}\calO_{\ol\bfM_{K(\grn_{\wt K})}^r}\,,
\]
where the two maps on the right are the diagonal inclusion, and the action of $\wt K/K(\grn_{\wt K})$, respectively. Using the analogous reasoning for $K$ and replacing $\wt K/K(\grn_{\wt K})$ by $K/K(\grn_{\wt K})$ we get a similar description of $\calO_{\ol\bfM_K^r}\otimes_{A[\grn_{K}^{-1}]}A[\grn_{\wt K}^{-1}]$. Putting both together we obtain the exact sequence of sheaves on $\ol\bfM_{\wt K}^r\otimes_{A[\grn_{K}^{-1}]}A[\grn_{\wt K}^{-1}]$
\[
0\longrightarrow \calO_{\ol\bfM_K^r}\otimes_{A[\grn_{K}^{-1}]}A[\grn_{\wt K}^{-1}] \longrightarrow \pi_{\wt K,K\,*}\calO_{\ol\bfM_{\wt K}^r} \overset{\textstyle\longrightarrow}{\longrightarrow} \prod_{K/\wt K}\pi_{\wt K,K\,*}\calO_{\ol\bfM_{\wt K}^r}\,.
\]
If $L$ is flat over $A[\grn_{\wt K}^{-1}]$ we may tensor the latter sequence with the flat $\calO_{\ol\bfM_K^r}\otimes_{A[\grn_{K}^{-1}]}A[\grn_{\wt K}^{-1}]$-module $\omega_K^{\otimes k}\otimes_{A[\grn_{K}^{-1}]}L$ to obtain the exact sequence 
\[
0\longrightarrow \omega_K^{\otimes k}\otimes_{A[\grn_{K}^{-1}]}L \longrightarrow \pi_{\wt K,K\,*}\bigl(\omega_{\wt K}^{\otimes k}\otimes_{A[\grn_{\wt K}^{-1}]}L\bigr) \overset{\textstyle\longrightarrow}{\longrightarrow} \prod_{K/\wt K}\pi_{\wt K,K\,*}\bigl(\omega_{\wt K}^{\otimes k}\otimes_{A[\grn_{\wt K}^{-1}]}L\bigr)\,.
\]
This proves the injectivity of \eqref{EqLemma4.4} for flat $L$ and finishes the proof of (2). Here we use the projection formula $(\pi_{\wt K,K\,*}\O_{\ol\bfM_{\wt K}^r})\otimes_{\O_{\ol\bfM_K^r}}(\omega_K^{\otimes k}\otimes_{A[\grn_K^{-1}]}L) = \pi_{\wt K,K\,*}(\pi_{\wt K,K}^*\omega_K^{\otimes k}\otimes_{A[\grn_K^{-1}]}L)$ from \cite[I$_{\rm new}$, Corollaire~9.3.9]{EGA}

To prove (3) we use that $L$ is a flat $A[\grn_{\wt K}^{-1}]$-algebra. Taking global sections on $\ol\bfM_{\wt K}^r\otimes_{A[\grn_{K}^{-1}]}A[\grn_{\wt K}^{-1}]$, which is a left exact functor yields the exact sequence
\[
0\longrightarrow M_k(r,K,L) \longrightarrow M_k(r,\wt K,L)\overset{\textstyle\longrightarrow}{\longrightarrow} \prod_{K/\wt K} M_k(r,\wt K,L)\,,
\]
and hence the isomorphism~\eqref{eq:MF.65}.
 \qed   
\end{proof}

If the $A[\grn_{\wt K}^{-1}]$-algebra $L$ is not flat we do not know whether assertion (3) of the previous lemma still holds true. However, if the special fibers $\ol\scrM_K^r:=\ol\bfM^r_K \otimes_{A[\grn_K^{-1}]} \F_v$ and $\ol\scrM_{\wt K}^r:=\ol\bfM^r_{\wt K} \otimes_{A[\grn_{\wt K}^{-1}]} \F_v$ are \emph{normal} at all places $v\nmid\grn_{\wt K}$ then (3) can be proved along the same lines as (2) by using the following proposition.

\begin{remark}
Let $\ol\scrM_K^{r,\rm nor}$ be the normalization of $\ol\scrM_K^r:=\ol\bfM^r_K \otimes_{A[\grn_K^{-1}]} \F_v$. Since $\ol\scrM_K^{r,\rm nor}\to\ol\scrM_K^r$ is finite, the pullback of $\omega_K$ to $\ol\scrM_K^{r,\rm nor}$, which we again denote $\omega_K$, is ample by \cite[II, Corollaire~4.6.6 and Proposition~5.1.12]{EGA}.
\end{remark}

\begin{prop}\label{Prop4.5}
Let $\wt K\lhd K\subset K(1)$ be two fine open compact subgroups such that $\wt K$ is normal in $K$, and let $L$ be an $\F_v$-algebra for a place $v\nmid\grn_{\wt K}$. The morphism $\pi_{\wt K,K}\colon \ol\scrM_{\wt K}^r\to\ol\scrM_K^r$ induces a morphism $\pi_{\wt K,K}\colon \ol\scrM_{\wt K}^{r,\rm nor}\to\ol\scrM_K^{r,\rm nor}$ which is finite. The latter induces for any $k$ an isomorphism of $L$-modules
\begin{equation}
  \label{eq:MF.65b}
  H^0\bigl(\ol\scrM_K^{r,\rm nor},\omega_K^{\otimes k}\otimes_{A[\grn_K^{-1}]}L\bigr)\isoto H^0\bigl(\ol\scrM_{\wt K}^{r,\rm nor},\omega_{\wt K}^{\otimes k}\otimes_{A[\grn_{\wt K}^{-1}]}L\bigr)^{K/\wt K}
\end{equation}
for the natural action of the group $K/\wt K$.
Moreover, the natural map 
\[
M_k(r,K,L)=H^0\bigl(\ol\scrM_K^{r},\omega_K^{\otimes k}\otimes_{A[\grn_K^{-1}]}L\bigr)\longrightarrow H^0\bigl(\ol\scrM_K^{r,\rm nor},\omega_K^{\otimes k}\otimes_{A[\grn_K^{-1}]}L\bigr)
\]
is injective.
\end{prop}

\begin{proof}
The morphism $\pi_{\wt K,K}\colon \ol\scrM_{\wt K}^{r,\rm nor}\to\ol\scrM_K^{r,\rm nor}$ is obtained by the universal property of the normalization. It is finite, because $\pi_{\wt K,K}\colon \ol\scrM_{\wt K}^r\to\ol\scrM_K^r$ is finite by Theorem~\ref{S.2}(2).

Like in the previous lemma, the proof of the isomorphism \eqref{eq:MF.65b} follows by considering the sequence of coherent sheaves on $\ol\scrM_K^{r,\rm nor}$
\begin{equation}\label{EqSeqProp4.5}
0\longrightarrow \O_{\ol\scrM_K^{r,\rm nor}} \longrightarrow \pi_{\wt K,K\,*}\O_{\ol\scrM_{\wt K}^{r,\rm nor}} \overset{\textstyle\longrightarrow}{\longrightarrow} \prod_{K/\wt K}\pi_{\wt K,K\,*}\O_{\ol\scrM_{\wt K}^{r,\rm nor}}\,,
\end{equation}
tensoring it with the flat $\O_{\ol\scrM_K^{r,\rm nor}}$-module $\omega_K^{\otimes k}\otimes_{\F_v}L$ and taking global sections. It thus remains to prove that the sequence~\eqref{EqSeqProp4.5} is exact. 

Exactness on the left follows as in the previous lemma, because the scheme theoretic image of $\pi_{\wt K,K}\colon \ol\scrM_K^{r,\rm nor}\to\ol\scrM_{\wt K}^{r,\rm nor}$ is the entire space $\ol\scrM_K^{r,\rm nor}$ which is reduced.

To prove exactness in the middle let $\ol U=\Spec R\subset \ol\scrM_K^{r,\rm nor}$ be an affine open subset and let $U:=\ol U\cap \scrM_K^r$, where we observe that $\scrM_K^r:=\bfM^r_{\wt K} \otimes_{A[\grn_{\wt K}^{-1}]} \F_v$ is smooth over $\F_v$, hence normal, and hence an (affine) open subset of $\ol\scrM_K^{r,\rm nor}$. By Proposition~\ref{2.5}(2) the scheme $U\times_{\ol\scrM_K^{r,\rm nor}}\ol\scrM_{\wt K}^{r,\rm nor}$ is a finite \'etale Galois cover of $U$ with Galois group $K/\wt K$. Thus the restriction of the sequence \eqref{EqSeqProp4.5} to the dense open $U$ is exact. Let $f\in H^0(\ol U, \pi_{\wt K,K\,*}\O_{\ol\scrM_{\wt K}^{r,\rm nor}}) = H^0(\ol U\times_{\ol\scrM_K^{r,\rm nor}}\ol\scrM_{\wt K}^{r,\rm nor}, \O_{\ol\scrM_{\wt K}^{r,\rm nor}})$ lie in the equalizer of the two morphisms on the right. By the exactness of \eqref{EqSeqProp4.5} on $U$ the restriction $f|_U$ of $f$ to $U$ lies in $H^0(U, \O_{\ol U})$. Since $\ol U$ is normal the ring $R=H^0(\ol U,\O_{\ol U})$ equals the intersection $R_\grp$ of its local rings at height one primes $\grp\subset R$ by \cite[Theorem~11.5(ii)]{matsumura:crt89}. Thus it suffices to prove that $f\in R_\grp$ for all such $\grp$, or equivalently that $v_\grp(f)\ge0$ where $v_\grp$ is the valuation of the discrete valuation ring $R_\grp$. The scheme $\Spec R_\grp\times_{\ol\scrM_K^{r,\rm nor}}\ol\scrM_{\wt K}^{r,\rm nor}$ is finite over $R_\grp$. Let $\grq$ be a point in that scheme lying above $\grp$. The local ring at $\grq$ is a finite extesion of $R_\grp$, and hence also a discrete valuation ring with valuation $v_\grq$ extending $v_\grp$. Since $f\in H^0(\ol U\times_{\ol\scrM_K^{r,\rm nor}}\ol\scrM_{\wt K}^{r,\rm nor}, \O_{\ol\scrM_{\wt K}^{r,\rm nor}})$, we have $v_\grq(f)\ge0$, and hence $v_\grp(f)\ge0$. This proves that $f\in H^0(\ol U, \O_{\ol U})$, whence the exactness of \eqref{EqSeqProp4.5}.

The final statement follows as in the previous lemma, because the scheme theoretic image of $\ol\scrM_K^{r,\rm nor}\to\ol\scrM_K^r$ is the entire space $\ol\scrM_K^r$ which is reduced.
\qed
\end{proof}

\begin{remark}
The cokernel of 
\begin{equation} \label{EqCokerNormaliz}
M_k(r,K,L)=H^0\bigl(\ol\scrM_K^{r},\omega_K^{\otimes k}\otimes_{A[\grn_K^{-1}]}L\bigr)\longrightarrow H^0\bigl(\ol\scrM_K^{r,\rm nor},\omega_K^{\otimes k}\otimes_{A[\grn_K^{-1}]}L\bigr)
\end{equation}
can be described as follows. Let $f\colon\ol\scrM_K^{r,\rm nor}\to\ol\scrM_K^r$ and consider the cokernel sheaf $\mathcal F$ in the following exact sequence
\[
0\longrightarrow \O_{\ol\scrM_K^r} \longrightarrow f_*\O_{\ol\scrM_{K}^{r,\rm nor}} \longrightarrow \mathcal F \longrightarrow 0.
\]
$\mathcal F$ is supported on the locus, where $\ol\scrM_K^r$ is not normal. This locus is closed and contained in $\ol\scrM_K^r \setminus \scrM_K^r$. If Conjecture~\ref{ConjCM} holds, then it has codimension one and is a union of irreducible components of $\ol\scrM_K^r \setminus \scrM_K^r$. The cokernel of \eqref{EqCokerNormaliz} is then equal to the kernel of
\[
H^0\bigl(\ol\scrM_K^r,\omega_K^{\otimes k}\otimes\mathcal F\otimes_{A[\grn_K^{-1}]}L\bigr) \longrightarrow H^1\bigl(\ol\scrM_K^{r},\omega_K^{\otimes k}\otimes_{A[\grn_K^{-1}]}L\bigr).
\]
For $k\gg0$ the $H^1$-term vanishes and the cokernel equals $H^0\bigl(\ol\scrM_K^r,\omega_K^{\otimes k}\otimes\mathcal F\otimes_{A[\grn_K^{-1}]}L\bigr)$. To compute this further one would need an understanding of the singularities of $\ol\scrM_K^r$.
\end{remark}

For any arithmetic subgroup $\Gamma\subset G(F)=\GL_r(F)$, Basson, Breuer
and Pink \cite[Definition 6.1]{BBP} and Gekeler~\cite[Definition 1.11]{gekeler:I} have defined the space $M_k(\Gamma)$ of Drinfeld modular
forms of weight $k$ and level $\Gamma$ over $\C_\infty$. These are
$\C_\infty$-valued (rigid analytic) holomorphic functions 
on the Drinfeld period domain $\Omega^r$ that satisfy the usual conditions defined by
automorphy factors (i.e.~are weakly modular forms) and are required to 
be holomorphic ``at infinity''.  
Basson, Breuer and Pink proved the following comparison theorem.

\begin{thm}[{\cite[Theorem 10.9]{BBP}}]\label{MF.comp}
There is an isomorphism 
\begin{equation}
  \label{eq:decomp}
  \bfM^r_K\otimes_{A[\grn_K^{-1}]} \C_\infty
\simeq \coprod_{i=1}^h \Gamma_{g_i} \backslash
\Omega^r,\quad \Gamma_{g_i}:=G(F)\cap g_i K {g_i}^{-1},
\end{equation}
where $g_1,\dots, g_h$ are complete representatives for the double
coset space $G(F)\backslash 
G(\A^\infty)/K.$
When $K\subset K(1)$ is fine, 
there is a natural isomorphism of $\C_\infty$-vector spaces
\begin{equation}
  \label{eq:compare}
  M_k(r,K,\C_\infty)\isoto \bigoplus_{i=1}^h M_k(\Gamma_{g_i}).
\end{equation}
\end{thm}

\subsection{Drinfeld modular forms over $A_{(v)}$}
\label{sec:MF.2}
Let $v$ be a finite place of $F$. 
We say that an open compact subgroup $K\subset
G(\A^{\infty})$ is {\it fine} if it is conjugate to a fine subgroup 
of $K(1)$. Let $A_{(v)}$ denote the localization of $A$ at
the place $v$, and let $\wh A^{(v)}:=\prod_{v'\neq v} A_{v'}$. 
In this subsection
we shall define algebraic Drinfeld modular forms of rank $r$ over an
$A_{(v)}$-algebra $L$
and prime-to-$v$ Hecke operators on these modular forms. 
For this purpose we need to extend the action of $G(\wh A)$ on $\bfM^r_K$ from \eqref{eq:action} to $G(\A^{\infty v})$. The conceptual best way for this is to (re-)define the moduli schemes $\bfM^r_K$ over $A_{(v)}$
and to construct their Satake compactification 
for fine level subgroups $K\subset G(\A^{\infty})$ 
of the form $K=K_v K^v$, where 
$K_v=K_v(1)=G(A_v)$ is a fixed maximal open compact subgroup, 
and $K^v\subset G(\A^{\infty v})$ is an
open compact subgroup not necessarily contained in $G(\wh A^{(v)})$
which can vary. Correspondingly, we need to (re-)define $K$-level structures on Drinfeld $A$-modules for such subgroups $K\subset G(\A^{\infty v})$. 
For the remainder of this article unless stated otherwise $K_v$ and
$K^v$ are as above and $K=K_v K^v$.  

For a Drinfeld $A$-module $(E,\varphi)$ over an $A_{(v)}$-field $L$, 
the \emph{prime-to-$v$ Tate module} and \emph{Tate space} of $(E,\varphi)$
are defined as 
\[ T^{(v)}(\varphi):=\Hom_A(\A^{\infty v}/\widehat A^{(v)}, \varphi(L^\sep)), 
\quad V^{(v)}(\varphi):=\A^{\infty v}\otimes_{\wh A^{(v)}} 
T^{(v)}(\varphi), \]
where $L^\sep$ denotes a separable closure of $L$.

\begin{defn}
Let $S$ be a \emph{connected} locally 
Noetherian $A_{(v)}$-scheme, and let
$(E,\varphi)$ be a
Drinfeld $A$-module of rank $r$ over $S$. A \emph{level-$K^v$ structure} on
$(E,\varphi)$ is a $K^v$-orbit $\ol \eta=\eta K^v$ 
of $\A^{\infty v}$-linear isomorphisms 
\[  \eta: (\A^{\infty v})^r\isoto V^{(v)}(\varphi_{\bar s}) \]
which is $\pi_1(S,\bar s)$-invariant, where $\bar s$ is a geometric
point of $S$. Here $K^v\subset\GL_r(\A^{\infty v})$ acts on $(\A^{\infty v})^r$ 
and $\pi_1(S,\bar s)$ on $V^{(v)}(\varphi_{\bar s})$. 
If $K=K(\grn)\subset K(1)$ and $\eta$ maps 
$(\wh A^{(v)})^r$ onto $T^{(v)}(\varphi_{\bar s})$, then $\ol \eta$
is nothing but a level-$\grn$ structure on $(E,\varphi)$; see
\eqref{eq:level}. For general $S$, a level-$K^v$ structure on
$(E,\varphi)$ is a tuple $\ol \eta=(\ol \eta_{S_i})_{S_i\in \pi_0(S)}$, 
where each $\ol \eta_{S_i}$ is a level-$K^v$ structure on
$(E,\varphi)$ over $S_i$.  
\end{defn}

Recall that a morphism $\alpha: (E_1,\varphi_1)\to
(E_2,\varphi_2)$ of two Drinfeld $A$-modules over $S$ is called an
\emph{isogeny} if it is surjective with finite flat kernel. 
A morphism $\alpha$ is an isogeny if and only if $\alpha\neq 0$ above
every connected component of $S$; see \cite[Proposition
5.4]{hartl:isog}. For every isogeny $\alpha$, by 
\cite[Corollary 5.15]{hartl:isog} there is an element $a\in A$,
$a\neq 0$ and an isogeny $\beta: (E_2,\varphi_2)\to (E_1,\varphi_1)$
with $\alpha \beta=a\cdot {\rm id}_{E_2}$ and $\beta \alpha=a\cdot
{\rm id}_{E_1}$, and $\ker (\alpha)\subset \varphi_1[a]$. 
We say that an isogeny $\alpha: (E_1,\varphi_1)\to (E_2,\varphi_2)$ is
{\it prime-to-$v$} if there is an element $a\in A$ with $v\nmid a$
such that $\ker \alpha\subset \varphi_1[a]$. Equivalently, let $\gra$
be the kernel of the ring homomorphism $A\to \End(\ker(\alpha))$,
sending $a\mapsto \varphi_a$;  
then $\alpha$ is prime-to-$v$ if and only if $v\nmid \gra$.
Since $\ker\alpha\subset \varphi_1[a]$ for some
$a\neq 0$, we have $a\in \gra$ and $\gra\neq (0)$. 

\begin{defn}\label{def:MF.10} 
Let $\bfM^r_{K_v K^v}$ denote the moduli functor over $A_{(v)}$
classifying equivalence classes of Drinfeld $A$-modules
$(E,\varphi, \bar \eta)$ with level-$K^v$ structure. Here two objects
$(E_1,\varphi_1,\ol \eta_1)$ and $(E_2,\varphi_2,\ol \eta_2)$
over a base scheme $S$ are said to be {\it equivalent} if there is 
a prime-to-$v$ isogeny $\alpha:
(E_1,\varphi_1)\to (E_2,\varphi_2)$ over $S$ such that $\alpha_* \ol
\eta_1=\ol \eta_2$, where 
$\alpha_*: V^{(v)}(\varphi_{1,\bar s})\to V^{(v)}(\varphi_{2,\bar s})$ 
is induced from 
$\alpha:  \varphi_{1,\bar s}[\grn]\to \varphi_{2,\bar s}[\grn]$.
\end{defn}

We claim that when $K=K_vK^v=G(A_v)K^v\subset G(\A^\infty)$ is fine, 
the functor $\bfM^r_{K_v K^v}$ is representable by an
affine and smooth $A_{(v)}$-scheme, which we denote again 
by $\bfM^r_{K_v K^v}$, with a universal
family $(E_{K}, \varphi_{K}, \ol \eta_{K})$. To see this,    
for any element $g\in G(\A^{\infty v})$, the right translation by $g$
gives an isomorphism of functors:
\[ J_{\!g\,}: \bfM^r_{K_v K^v} \isoto \bfM^r_{(K_v K^v)^g}, \quad
(E,\varphi,\ol \eta)\mapsto (E,\varphi, \ol{\eta g}). \]
Here we write $K^g:=\Int(g^{-1}) K=g^{-1} K g$. In particular, if $g\in K^v$
then $K^g=K$ and $J_{\!g\,}$ is the identity on $\bfM^r_{K_v K^v}$.
One can choose $g\in\{1\}_v\times G(\A^{\infty v})$
such that $(K^v)^g\subset  G(\wh A^{(v)})$, and hence that $K^g\subset
K(1)$. 
Then the natural map 
$i_v:\bfM^r_{K^g}\otimes_{A[\grn_{K^g}^{-1}]} A_{(v)} \to \bfM^r_{(K_v K^v)^g}$,
sending $(E,\varphi,\lambda K^g) \mapsto (E,\varphi,\bar \eta)$, where
$\eta$ is a lifting of $\lambda$, is 
an isomorphism, and hence the representability of $\bfM^r_{(K_v K^v)^g}$,
and hence of $\bfM^r_{K_v K^v}$ via the isomorphism $J_{\!g\,}$, is obtained from Proposition~\ref{2.5}(2). 
We then transport the Satake
compactification $\ol \bfM^r_{K^g}$ of $\bfM^r_{K^g}$ and the universal 
family $(\olE_{K^g},\ol \varphi_{K^g})$ over $\ol\bfM^r_{K^g}$ to 
$\bfM^r_{K_v K^v}$. So we obtain the Satake compactification 
$\ol \bfM^r_{K_v K^v}$ of $\bfM^r_{K_v K^v}$ over $A_{(v)}$, and the 
universal family $(\olE_K,\ol\varphi_K)$ over $\ol \bfM^r_{K_v K^v}$.  
By
abuse of notation, we also write $\bfM^r_{K}$ for $\bfM^r_{K_v K^v}$ and  
$\ol\bfM^r_K$ for $\ol\bfM^r_{K_v K^v}$, understanding that they are schemes 
over $A_{(v)}$, not over $A[\grn_K^{-1}]$.

If $\wt K^v\subset K^v\subset G(\A^{\infty v})$ are fine open compact subgroups and $\wt K=G(A_v)\wt K^v$ and $K=G(A_v)K^v$, then ``forgetting the level'' induces by Theorem~\ref{S.2} a finite surjective and open morphism 
\[
\pi_{\wt K,K}\colon\ol\bfM^r_{\wt K}\longrightarrow\ol\bfM^r_K\otimes_{A[\grn_K^{-1}]}A[\grn_{\wt K}^{-1}]
\]
over $\Spec A[\grn_{\wt K}^{-1}]$ which satisfies $\pi_{\wt K,K}^*(\olE_K,\ol\varphi_K)=(\olE_{\wt K},\ol\varphi_{\wt K})$ and $\pi_{\wt K,K}^*(\omega_K)=\omega_{\wt K}$.

The construction of the Satake compactification in Section~\ref{sec:S}
also shows the following functorial property (see \eqref{eq:IgJg}): 
For any $g\in G(\A^{\infty v})$, the Hecke translation
$J_{\!g\,}:\bfM^r_{K}\to \bfM^r_{K^g}$, and the
the canonical isomorphism
$I_{\!g\,}:(E_K,\varphi_K)\to (E_{g^{-1}Kg},\varphi_{g^{-1}Kg})$ which lifts $J_{\!g\,}$
extend uniquely to isomorphisms $\ol J_{\!g\,}$ and $\ol I_{\!g\,}$, respectively, 
that fit into the following commutative diagram:
\begin{equation}
  \label{eq:MF.7}
  \begin{CD}
  (\olE_K, \ol \varphi_K) @>{\ol I_{\!g\,}}>> (\olE_{K^g}, \ol
  \varphi_{K^g}) \\ 
  @VVV @VVV \\
  \ol \bfM^r_{K} @>{\ol J_{\!g\,}}>>\ol \bfM^r_{K^g}.    
  \end{CD}
\end{equation}
In particular, if $g\in K^v$ then $K^g=K$ and $\ol J_{\!g\,}$ and $\ol I_{\!g\,}$ are the identity.
Similarly, write $\omega_K:=\Lie(\olE_K)^\vee$, which is 
an ample invertible sheaf on $\ol \bfM^r_{K}$ over $A_{(v)}$. 

\begin{defn}\label{MF.5}
  (1) For any integer $k\ge 0$, {\it fine} open compact subgroup 
  $K=K_vK^v \subset G(\A^{\infty})$ with $K_v=G(A_v)$
  and $A_{(v)}$-algebra $L$, denote by
  \begin{equation}\label{eq:MF.8}
    M_k(r,K,L):=H^0(\ol \bfM^r_{K}\otimes_{A_{(v)}} L,\;\omega_{K}^{\otimes k}\otimes {L})
  \end{equation}
  the $L$-module of \emph{algebraic Drinfeld  modular forms of rank $r$,
  weight $k$, level $K$ over $L$}. The definition of $M_k(r,K,L)$ in
\eqref{eq:MF.8} agrees with that in Definition~\ref{MF.1} (noting 
$\ol \bfM^r_{K}\otimes_{A[\grn_K^{-1}]}\otimes L=\ol \bfM^r_{K_v K^v}
\otimes_{A_{(v)}} L$).
Thus, there should be no danger of confusion. 

(2) For $\wt K^v\subset K^v\subset G(\A^{\infty v})$ fine open compact
  subgroups,  
  the pull-back under $\pi_{\wt K, K}:\ol\bfM^r_{K_v \wt
  K^v}\to \ol \bfM^r_{K_v K^v}$ yields a map
$\pi_{\wt K,K}^*: M_k(r,K_vK^v,L)\to M_k(r,K_v\wt K^v,L)$, which is
injective by Lemma~\ref{MF.35}.  Define  
\begin{equation}
  \label{eq:MF.9}
   M_k(r, K_v,L):=\indlimw{K^v} M_k(r, K_v K^v, L), \quad
  M(r,K_v,L):=\bigoplus_{k\ge 0} M_k(r,K_v,L).
\end{equation} 

(3) Let $k\ge 0$ and $K$ be as in (1). If $L$ is a flat $A_{(v)}$-algebra we set $\wt M_k(r,K,L):=M_k(r,K,L)$ and $\wt M_k(r,K_v,L):=M_k(r,K_v,L)$. On the other hand, if $L$ is an $\F_v$-algebra we consider the normalization $\ol\scrM_K^{r,\rm nor}$ of $\ol\scrM_K^r:=\ol\bfM^r_K \otimes_{A_{(v)}} \F_v$ and set 
  \begin{equation}\label{eq:MF.8b}
    \wt M_k(r,K,L):=H^0(\ol\scrM_K^{r,\rm nor}\otimes_{\F_v} L,\;\omega_{K}^{\otimes k}\otimes {L})\,.
  \end{equation}
By Proposition~\ref{Prop4.5} we have $M_k(r,K,L)\subset \wt M_k(r,K,L)$. Moreover, by Proposition~\ref{Prop4.5} the map
$\pi_{\wt K,K}^*:\wt M_k(r,K_vK^v,L)\to \wt M_k(r,K_v\wt K^v,L)$ is 
injective, and we define  
\begin{equation}
  \label{eq:wtMk}
  \wt M_k(r, K_v,L):=\indlimw{K^v}\wt M_k(r, K_v K^v, L)\,.
\end{equation} 
Then $M_k(r, K_v,L)\subset\wt M_k(r, K_v,L)$. We even have equality if $L$ is flat over $A_{(v)}$. As in Remark~\ref{Remk<0} we have $\wt M_k(r,K,L)=(0)$ for $k<0$ by Lemma~\ref{Lemma_DualOfAmple}.
\end{defn}

\begin{remark}
Take the projective limit of the tower of schemes $\ol \bfM^r_{K_vK^v}$
\begin{equation}
  \label{eq:MF.projlim}
  \ol \bfM^r_{K_v}:=\projlimw{K^v} \ol \bfM^r_{K_vK^v}.
\end{equation}
Note that $\ol \bfM^r_{K_v}$ is a scheme as the transition maps are
affine. It has a continuous right action of $G(\A^{\infty v})$. 
Then $M_k(r,K_v,L)=H^0(\ol \bfM^r_{K_v}\otimes_{A_{(v)}}L, \omega^{\otimes k}\otimes_{A_{(v)}}L)$.
Moreover, for any fine open compact subgroup $K=K_vK^v$, one has 
$\ol \bfM^r_{K_v}/K^v\simeq \ol \bfM^r_{K_vK^v}$ by Theorem~\ref{S.2}(2).
\end{remark}

We now describe the left action of $G(\A^{\infty v})$ on 
the $L$-modules $M_k(r,K_v,L)$ and $\wt M_k(r,K_v,L)$. 
For any element $g\in G(\A^{\infty v})$, the canonical isomorphisms in
(\ref{eq:MF.7}) gives a canonical isomorphism 
$\omega^{\otimes k}_{K}\simeq \ol J_{\!g\,}^* \omega^{\otimes
  k}_{K^g}$.
Using the adjoint 
isomorphism $\omega^{\otimes k}_{K^g}\isoto \ol J_{\!g*\,}  
\ol J_{\!g\,}^*\omega^{\otimes k}_{K^g}$, we get isomorphisms  
\begin{equation}
  \label{eq:MF.10}
  \begin{split}
   & H^0(\ol \bfM^r_{K^g}\otimes_{A_{(v)}} L,\,\omega_{K^g}^{\otimes
   k}\otimes L)   \isoto H^0(\ol \bfM^r_{K^g}\otimes_{A_{(v)}} L,\,\ol J_{\!g*\,} \ol J_{\!g\,}^*
   \omega_{K^g}^{\otimes k}\otimes L) \\[2mm]
   & \isoto H^0(\ol \bfM^r_{K}\otimes_{A_{(v)}} L,\, \ol J_{\!g\,}^* \omega_{K^g}^{\otimes k}\otimes L)
    \isoto H^0(\ol \bfM^r_{K}\otimes_{A_{(v)}} L,\, 
   \omega_{K}^{\otimes k}\otimes L), 
  \end{split}
\end{equation}
and an isomorphism $T_g:M_k(r,K^g,L)\isoto M_k(r,K,L)$. For
$g_1,g_2\in G(\A^{\infty v})$, one has $T_{g_1}\circ T_{g_2}=T_{g_1 g_2}$.
Taking the inductive limit, the group $G(\A^{\infty v})$ naturally
acts on $M_k(r,K_v,L)$. The group $G(\A^{\infty v})$ acts as
automorphisms of the graded rings on $M(r,K_v,L)$, that is, one has
\begin{equation}
  \label{eq:MF.11}
  T_g(f_1\cdot f_2)=T_g(f_1)\cdot T_g(f_2), \quad \forall\, f_1, f_2\in
  M(r,K_v,L). 
\end{equation}
Likewise, applying base change to $\F_v$ and normalization to the isomorphisms in
(\ref{eq:MF.7}) yields a commutative diagram 
\begin{equation*}
  \begin{CD}
  (\olE_K, \ol \varphi_K) @>{\ol I_{\!g\,}}>> (\olE_{K^g}, \ol
  \varphi_{K^g}) \\ 
  @VVV @VVV \\
  \ol\scrM_K^{r,\rm nor} @>{\ol J_{\!g\,}}>>\ol\scrM_{K^g}^{r,\rm nor}.    
  \end{CD}
\end{equation*}
in which the horizontal maps $\ol I_{\!g\,}$ and $\ol J_{\!g\,}$ are isomorphisms. Here we abuse notation and denote the pullback of $(\olE_K, \ol \varphi_K)$ from $\ol\bfM^r_K$ to $\ol\scrM_K^{r,\rm nor}$ again by $(\olE_K, \ol \varphi_K)$. If now $L$ is an $\F_v$-algebra, the same reasoning as above applied to $\ol\scrM_K^{r,\rm nor}$ produces an isomorphism $T_g:\wt M_k(r,K^g,L)\isoto \wt M_k(r,K,L)$ and an action of $G(\A^{\infty v})$ on $\wt M_k(r,K_v,L)$.

\begin{thm}\label{ThmSmoothHeckeAction}
(1) Let $L$ be a flat $A_{(v)}$-algebra or a noetherian $\F_v$-algebra. Then the actions of $G(\A^{\infty v})$ on $M_k(r,K_v,L)$ and $\wt M_k(r,K_v,L)$ are
\begin{itemize}
\item \emph{smooth} (in the sense that every element of 
$M_k(r,K_v,L)$ and $\wt M_k(r,K_v,L)$ has an open stabilizer) and 
\item \emph{admissible} (in the sense that for every open compact subgroup $K^v\subset G(\A^{\infty v})$ the fixed points $M_k(r,K_v,L)^{K^v}$ and $\wt M_k(r,K_v,L)^{K^v}$ form finitely generated $L$-modules).
\end{itemize}
In particular, for every open compact subgroup $K^v\subset G(\A^{\infty v})$ there is an open compact subgroup $\wt K^v\subset K^v$ such that $M_k(r,K_v,L)^{K^v}\subset M_k(r,K_v\wt K^v,L)$.

(2) If $L$ is a flat $A_{(v)}$-algebra then $M_k(r,K_vK^v,L)= M_k(r,K_v,L)^{K^v}$ for any  open compact subgroup $K^v\subset G(\A^{\infty v})$.

(3) If $L$ is an $\F_v$-algebra then $\wt M_k(r,K_vK^v,L)= \wt M_k(r,K_v,L)^{K^v}$ for any  open compact subgroup $K^v\subset G(\A^{\infty v})$.

\end{thm}

\begin{proof}
(2) and (3) follow from the isomorphisms \eqref{eq:MF.65} in Lemma~\ref{MF.35}(3) and \eqref{eq:MF.65b} in Proposition~\ref{Prop4.5}, respectively.

(1) To prove smoothness let $f\in M_k(r,K_v,L)$ or $f\in\wt M_k(r,K_v,L)$. Then there exists an open compact subgroup $K^v\subset G(\A^{\infty v})$ such that $f\in M(r,K_vK^v,L)$ or $f\in\wt M(r,K_vK^v,L)$, respectively, and then $K^v$ stabilizes $f$.

To prove admissibility, let $K^v\subset G(\A^{\infty v})$ be an open compact subgroup. If $L$ is a flat $A_{(v)}$-algebra then $\wt M_k(r,K_v,L)=M_k(r,K_v,L)$ and (2) implies that $M_k(r,K_v,L)^{K^v}=M_k(r,K_vK^v,L)$ is a finitely generated $L$-module. On the other hand, if $L$ is an $\F_v$-algebra then (3) implies that $\wt M_k(r,K_v,L)^{K^v}=\wt M_k(r,K_vK^v,L)$ is a finitely generated $L$-module. If $L$ is moreover noetherian, then $M_k(r,K_v,L)^{K^v}\subset\wt M_k(r,K_v,L)^{K^v}=\wt M_k(r,K_vK^v,L)$ is also finitely generated.

The last assertion of (1) follows from the finiteness of $M_k(r,K_v,L)^{K^v}$ as $L$-submodule of the inductive limit $M_k(r, K_v,L)=\lim\limits_{\longrightarrow} M_k(r, K_v\wt K^v, L)$. \qed
\end{proof}

\subsection{Systems of Hecke eigenvalues in $\ol \F_v$ }
\label{sec:MF.3}

Since most texts only consider Hecke algebras with values in $\Q$-algebras, we review, for the readers convenience, the theory of the Hecke algebras for $\GL_r$ with values in arbitrary commutative rings. From \cite{Herzig11,Henniart-Vigneras} we recall the following

\begin{defn}\label{DefHeckeAlg}
Set $\bbG:=G(\A^{\infty v})$, let $K^v\subset\bbG$ be an open compact subgroup, and let $R$ be any commutative ring. The \emph{Hecke algebra} $\calH_R(G(\A^{\infty v}),K^v)$ is the convolution $R$-algebra 
\begin{align*}
\calH_R(G(\A^{\infty v}),K^v) := \bigl\{\,h & \colon\bbG \longrightarrow R\text{ functions with compact support }\big| \\[1mm]
& h(k_1gk_2)=h(g)\text{ for all }g\in G(\A^{\infty v})\text{ and }k_1,k_2\in K^v\,\bigr\}\,.
\end{align*}
Here compact support means that $h$ is zero outside a finite union of cosets $K^vgK^v$. The multiplication is defined by \emph{convolution}
\[
(\tilde h*h)(g) \;:=\;\sum_{\tilde g\in \bbG/K^v} \tilde h(\tilde g)\cdot h(\tilde g^{-1}g)\;=\;\int_{\bbG}\tilde h(\tilde g)\cdot h(\tilde g^{-1}g)\;d\tilde g\,,
\]
where $d\tilde g$ is the left-invariant Haar measure on $\bbG$ with $\vol(K^v)=1$. The characteristic function ${\bf1}_{K^v}$ of $K^v$ is the unit element.

The Hecke algebra is isomorphic to the endomorphism algebra $\End_{R[\bbG]}\bigl(\cInd_{K^v}^\bbG(\UOne)\bigr)$ of the compact induction
\[
\cInd_{K^v}^\bbG(\UOne)\;:=\;\bigl\{\,f\colon\bbG\to R\text{ functions with compact support }\big|\ f(kg)=f(g) \text{ for all }k\in K^v, g\in\bbG\,\bigr\}
\]
of the trivial $K^v$-representation $\UOne=R$. Here $\cInd_{K^v}^\bbG(\UOne)$ is an $R[\bbG]$-module by right translation, that is, $\tilde g\in\bbG$ maps $f\in \cInd_{K^v}^\bbG(\UOne)$ to $\rho_{\tilde g}(f)$ which is defined by $\rho_{\tilde g}(f)(g)=f(g\tilde g)$. The isomorphism $\calH_R(G(\A^{\infty v}),K^v)\isoto\End_{R[\bbG]}\bigl(\cInd_{K^v}^\bbG(\UOne)\bigr)$ sends $h\in \calH_R(G(\A^{\infty v}),K^v)$ to the endomorphism $h$ of $\cInd_{K^v}^\bbG(\UOne)$
\[
h\colon f\;\longmapsto\; h*f\,,\text{ defined by}\quad (h*f)(g)\;:=\;\sum_{\tilde g\in\bbG/K^v}h(\tilde g)\cdot f(\tilde g^{-1}g)\;=\;\int_{\bbG} h(\tilde g)\cdot f(\tilde g^{-1}g)\;d\tilde g\,.
\]
From now on we assume that $K^v=\prod_{v'\nmid \infty v}K_{v'}$ with open compact 
subgroups $K_{v'}\subset G(F_{v'})$ with $K_{v'}=G(A_{v'})$ for almost all $v'$.
Then the Hecke algebra $\calH_R(G(\A^{\infty v}),K^v)$ decomposes into local Hecke algebras
\[ \calH_R(G(\A^{\infty v}),K^v)\;=\;\bigotimes_{v'\nmid v\infty}^{}\!\! '\,  
  \calH_{R}(G(F_{v'}),K_{v'}) 
\]
which are defined analogously. Here $\otimes'$ denotes the restricted 
tensor product with respect to the unit elements 
${\bf1}_{K_{v'}}\in\calH_{R}(G(F_{v'}),K_{v'}) $ for almost all places $v'$.
\end{defn}

Let $\grn\subset A$ be a 
non-zero ideal, prime to $v$ such that $K(\grn)$ is contained in 
a conjugate of $K=K_vK^v$. When $v'\nmid \infty v \grn$, then the open compact subgroup $K_{v'}$ is maximal and hyperspecial,
and  $\calH_{R}(G(F_{v'}),K_{v'})$ is commutative
by \cite[\S\,1.5~Theorem and Remark]{Henniart-Vigneras}, or \cite[Theorem~2.6]{Herzig11}. This can also be seen in an elementary way called ``Gelfand's trick'' as follows. Let $g'\in G(F_{v'})$ be such that $g'{}^{-1}K_{v'}g'=G(A_{v'})$. Then there is an isomorphism of $R$-algebras
\begin{eqnarray*}
\calH_R(G(F_{v'}),K_{v'}) & \isoto & \calH_R\bigl(G(F_{v'}),G(A_{v'})\bigr) \\[1mm]
h \quad & \longmapsto & h\circ \Int_{g'} \qquad\text{where}\qquad h\circ \Int_{g'}\colon g\mapsto h(g'gg'{}^{-1})\,.
\end{eqnarray*}
The $R$-algebra $\calH_R\bigl(G(F_{v'}),G(A_{v'})\bigr)$ has an involution
\[
\iota\colon h \longmapsto {}^\iota h\qquad\text{where}\qquad{}^\iota h\colon g\mapsto h({}^tg)
\]
and ${}^tg\in G(F_{v'})$ denotes the transpose of $g\in G(F_{v'})$. This means that ${}^\iota(\tilde h*h)={}^\iota h*{}^\iota\tilde h$. By the elementary divisor theorem every double coset $G(A_{v'})gG(A_{v'})$ has a representative $g$ which is a diagonal matrix, and hence satisfies ${}^tg=g$. This shows that $\iota$ is the identity on $\calH_R\bigl(G(F_{v'}),G(A_{v'})\bigr)$ and proves the commutativity $\tilde h*h=h*\tilde h$ of $\calH_R\bigl(G(F_{v'}),G(A_{v'})\bigr)$.

\bigskip

Now we apply this to Drinfeld modular forms. For brevity we put $\calH_{L}^{\infty v}:=\calH_L(G(\A^{\infty v}),K^v)$ 
for any $\F_v$-algebra $L$. 
By Theorem~\ref{ThmSmoothHeckeAction}, the spaces
$M_k(r,K_v,L)^{K^v}$ and $\wt M_k(r,K_v,L)^{K^v}=\wt M_k(r,K_vK^v,L)$ are finite $L$-modules equipped with an
$\calH_L^{\infty v}$-module structure given for $f\in M_k(r,K_v,L)^{K^v}$ or $f\in\wt M_k(r,K_v,L)^{K^v}$ and $h=\sum_{i} n_i\cdot {\bf1}_{g_iK^v}\in \calH^{\infty v}_L$ with $n_i\in L$ by the rule
\begin{equation}
  \label{eq:MF.action}
  T(h)(f)\;=\;\sum_{\tilde g\in\bbG/K^v} h(\tilde g)\cdot T_{\tilde g}(f)\;=\;\sum_i n_i \cdot T_{g_i}(f)
\;=\;\int_{\bbG}h(\tilde g)\cdot T_{\tilde g}(f)\;d\tilde g
\end{equation}
where again $d\tilde g$ is the left-invariant Haar measure on $\bbG$ with $\vol(K^v)=1$. In particular, with $g_1:=\tilde gg$ and $g=\tilde g^{-1}g_1$ one computes as usual
\begin{eqnarray}
\label{eq:formula*}
T(\tilde  h)\bigl(T(h)(f)\bigr) & = &  \textstyle \int\limits_{\bbG}\tilde h(\tilde g)\cdot T_{\tilde g}\bigl(\int\limits_{\bbG} h(g)\cdot T_{g}(f)\;dg\bigr)\;d\tilde g \\
& = &  \textstyle \int\limits_{\bbG}\int\limits_{\bbG}\tilde h(\tilde g)\cdot h(g)\cdot T_{\tilde g}\circ T_{g}(f)\;dg\;d\tilde g \nonumber \\
& = & \textstyle \int\limits_{\bbG}\int\limits_{\bbG}\tilde h(\tilde g)\cdot h(\tilde g^{-1}g_1)\cdot T_{g_1}(f)\;d\tilde g\;dg_1 \nonumber \\
& = & T(\tilde h*h)(f)\,.  \nonumber
\end{eqnarray}

We now specialize to the case where $L=\ol \F_v$.
The commutativity of the Hecke algebra $\calH^{\infty v\grn}_{\ol\F_v}:=\bigotimes_{v'\nmid \infty v\grn}^{}\!\!\! '\,  \calH_{\ol\F_v}(G(F_{v'}),K_{v'})$
allows us to make the following definition. 

\begin{defn}\label{MF.7}
  Let $\grn$ and $K=K_v K^v$ be as above. 
  A Drinfeld modular form $0\ne f\in M_k(r,K_v,\ol \F_v)^{K^v}$ is said to
  be a \emph{Hecke eigenform for prime-to-$v\grn$ Hecke operators} if for
  any place $v'\nmid \infty v\grn$ and any element $h\in \calH_{v',\ol
  \F_v}:=\calH_{\ol \F_v}(G(F_{v'}),K_{v'})$, one has 
  \begin{equation}
    \label{eq:MF.13}
    T(h)(f)\;=\;a_{v'}(h) f,\quad \text{for some $a_{v'}(h)\in \ol \F_v$}. 
  \end{equation}
  Then by formula~\eqref{eq:formula*} the map $a_{v'}\colon \calH_{v',\ol \F_v}\to \ol \F_v$ is
  a homomorphism of $\ol\F_v$-algebras, and is called the {\it
  character} of $f$ at $v'$. 
  The collection $(a_{v'})_{v'\nmid \infty v\grn}$ 
  of characters $a_{v'}$, or equivalently, the character
  \[
(a_{v'})_{v'}:\calH_{\ol\F_v}^{\infty v\grn}:=\bigotimes_{v\nmid \infty v\grn}^{}\!\! ' \calH_{v',\ol \F_v}\longrightarrow  \ol \F_v\]
 is called the \emph{system of Hecke eigenvalues (or the
 Hecke eigensystem) of $f$}.    

We make the analogous definition for $f\in\wt M_k(r,K_v,\ol \F_v)^{K^v}=\wt M_k(r,K_vK^v,\ol \F_v)$.
\end{defn}

We are interested in Hecke eigensystems arising from algebraic Drinfeld
Hecke eigenforms over $\ol \F_v$ for \emph{all} weights and will determine them in Theorem~\ref{HI.6}. 

\begin{remark}
  Our discussions also cover the case where $L=\C_\infty$. 
  Namely, 
  we have the space $M_k(r, K,\C_\infty)$ of algebraic Drinfeld
  modular forms of rank $r$ and level $K$ over $\C_\infty$, and
  can consider Hecke eigenforms and Hecke eigensystems over
  $\C_\infty$. The comparison theorem for algebraic and analytic
  Drinfeld modular forms proved by Basson, Breuer and Pink
  \cite[Theorem 10.9]{BBP} (cf.~our Theorem~\ref{MF.comp})
  indicates 
  an alternative way of studying the action of Hecke operators 
  on these modular
  forms by analysis. This has been accomplished for $r=2$ by Gekeler and
  others.  

  We explain how to construct the ``reduction mod $v$'' of a 
  Hecke eigensystem arising from $M_k(r,
  K,\C_\infty)$. This gives an approach to study 
  Hecke eigensystems arising from $M_k(r,
  K,\C_\infty)$ through studying Hecke eigensystems of 
  $M_k(r,K,\ol \F_v)$. Note that such a construction is not obvious
  as the reduction modulo $v$ of a Hecke
  eigenform $f\in M_k(r,
  K,\C_\infty)$ may not be defined:   
  because $f$ may not be either defined over
  $\ol K$ or defined over the 
  integral ring $\ol A$ of $\ol K$, or even $f$ is defined over $\ol
  A$ but its reduction could be zero.   
  Suppose $f\in M_k(r,
  K,\C_\infty)$ is a prime-to-$v\grn$ eigenform and let
  $(a_{v'}): \calH_{\Z}^{\infty v \grn}\otimes_{\Z} \C_{\infty} \to 
  \C_{\infty}$ be the associated Hecke character, where 
  $\calH_\Z^{\infty v\grn}:=\calH_\Z(G(\A^{\infty v\grn}),K^{v\grn})$
  and $K^{v\grn}$ is the prime-to-$\grn$ part of $K^v$.
  Now $M_k(r,K,\C_\infty)=M_k(r,K,A_{(v)})\otimes_{A_{(v)}}
  \C_\infty$ by \cite[I$_{\rm new}$, Proposition~9.3.2]{EGA} and any 
  prime-to-$v\grn$ Hecke operator $T(h)$ for $h\in \calH_\Z^{\infty
  v\grn}$  
  leaves the $A_{(v)}$-module $M_k(r,K,A_{(v)})$ invariant, which is
  finite 
  by Theorem~\ref{ThmSmoothHeckeAction}. 
  Consider the $A_{(v)}$-subalgebra $\calT_{A_{(v)}}$ of
  $\End_{A_{(v)}}\bigl(M_k(r,K,A_{(v)})\bigr)$ generated by all $T(h)$
  for $h\in \calH_\Z^{\infty v\grn}$, or equivalently,
  $\calT_{A_{(v)}}$ is the image of $\calH_\Z^{\infty
  v\grn}\otimes_\bbZ
  A_{(v)}\to\End_{A_{(v)}}\bigl(M_k(r,K,A_{(v)})\bigr)$. Then
  $\calT_{A_{(v)}}$ is an $A_{(v)}$-algebra which is finite as an
  $A_{(v)}$-module, because the same holds for
  $\End_{A_{(v)}}\bigl(M_k(r,K,A_{(v)})\bigr)$ and $A_{(v)}$ is
  noetherian. It follows that
  $\calT_{A_{(v)}}\otimes_{A_{(v)}}\ol\F_v$ surjects onto the image of
  $\calH_\Z^{\infty v\grn}\otimes_\bbZ
  \ol\F_v\to\End_{\ol\F_v}\bigl(M_k(r,K,A_{(v)})\otimes_{A_{(v)}}\ol\F_v\bigr)$.
  Note that for $k\gg0$ we have
  $M_k(r,K,A_{(v)})\otimes_{A_{(v)}}\ol\F_v=M_k(r,K,\ol\F_v)$ and
  $M_k(r,K,A_{(v)})$ is finite projective by Corollary~\ref{MF.3}. In
  this case the kernel of $\calT_{A_{(v)}}\otimes_{A_{(v)}}\ol\F_v\to
  \End_{\ol\F_v}\bigl(M_k(r,K,\ol\F_v)\bigr)$ is a nilpotent ideal by
  \cite[Proposition~I.4.1]{Bellaiche10} and we denote the image by
  $\calT_{\ol\F_v}$.  

  For every $h\in \calH_\Z^{\infty v\grn}$ the image of $T(h)$ in
  $\calT_{A_{(v)}}$ satisfies a monic polynomial with coefficients  
  in $A_{(v)}$ by \cite[Theorem~4.3]{eisenbud}, and hence
  the eigenvalues of $T(h)$ on $M_k(r,K,\bbC_\infty)$ are all
  $v$-adically integral.  
  Let $F(f)\subset \C_\infty$ be the field generated by the
  eigenvalues of 
  all $T(h)$ on $f$. Since $\calT_{A_{(v)}}$ is finitely generated, $F(f)$ is a finite 
  field extension of $F$ and the character 
  $(a_{v'})_{v'}:\calH^{\infty v\grn}_\Z \to \C_\infty$ 
  factors through the integral closure $R_f$ of $A_v$ in $F(f)$. In other words, $f$ defines a character $\chi_f\colon\calT_{A_{(v)}}\to R_f$.
  Modulo any maximal ideal $\grm$ of $R_f$
  we obtain a Hecke eigensystem $(\bar a^\grm_{v'})_{v'}$ with
  values in $\ol \F_v$. We say the collection of characters $(\bar a^\grm_{v'})_{v'}$ for all maximal ideals $\grm$ of $R_f$ is the \emph{reduction modulo
  $v$} of $(a_{v'})_{v'}$. This collection is a subset of $\Hom_{A_{(v)}}(\calT_{A_{(v)}},\ol\F_v)=\Hom_{\ol\F_v}(\calT_{A_{(v)}}\otimes_{A_{(v)}}\ol\F_v,\ol\F_v)$. Assume now that $M_k(r,K,A_{(v)})\otimes_{A_{(v)}}\ol\F_v=M_k(r,K,\ol\F_v)$ and $M_k(r,K,A_{(v)})$ is finite projective. Then every character $(\bar a^\grm_{v'})_{v'}$ for fixed $\grm$ induces an element of $\Hom_{\ol\F_v}(\calT_{\ol\F_v},\ol\F_v)$, as the kernel of $\calT_{A_{(v)}}\otimes_{A_{(v)}}\ol\F_v\to\calT_{\ol\F_v}$ is nilpotent. Then $(\bar a^\grm_{v'})_{v'}$ is the Hecke eigensystem of a Hecke eigenform $\bar g$ in $M_k(r,K,\ol\F_v)$ by \cite[Theorem~I.5.9]{Bellaiche10}.
  But note that there is no reduction map of Hecke eigenforms in general producing $\bar g$ from $f$.
Many more such considerations can be found in \cite{Bellaiche10}, for instance on the question of lifting Hecke eigensystems from $\ol\F_v$ to finite extensions of the completion of $A_{(v)}$ in \cite[Lemma~I.7.9]{Bellaiche10}.
\end{remark}

\section{The Supersingular locus and Hecke modules}
\label{sec:SS}

Let $F$, $\infty$, $A$ be as in previous sections.
Let $v$ be a finite place and $\grp\subset A$ the corresponding prime
ideal. Denote by $\F_v:=A/\grp$ the residue field at $v$ and $\ol
\F_v$ its algebraic closure, regarded as $A$-fields. The
cardinality of $\F_v$ is denoted by $q_v$.

Let $K=K_v K^v$ where $K_v=G(A_v)$. As in Section~\ref{sec:MF.2} let $\bfM^r_{K}=\bfM_{K_vK^v}^r$ be the Drinfeld moduli scheme 
of rank $r$ over $A_{(v)}$ with level-$K^v$ structure and $\ol\bfM^r_{K}=\ol\bfM_{K_vK^v}^r$ its Satake compactification over $A_{(v)}$.
If $K$ is not fine then $\bfM^r_{K}$ is a coarse moduli scheme. Denote by 
$\scrM^r_{K}:=\bfM^r_{K} \otimes_{A_{(v)}}\F_v$ and $\ol\scrM^r_{K}:=\ol\bfM^r_{K} \otimes_{A_{(v)}}\F_v$ the special fibers of $\bfM^r_{K}$ and $\ol\bfM^r_{K}$, respectively.

\subsection{The $v$-rank stratification}
\label{sec:SS.1}

\begin{defn} \label{SS.1} 
  Let $\varphi$ be a Drinfeld $A$-module over an
  $\F_v$-field $L$. 

  (1)  The {\it $v$-rank} of $\varphi$, denoted by 
  $v\text{-rank}(\varphi)$, is the
  non-negative integer $j$ with $\varphi[\grp](\olL)\simeq
  (\grp^{-1}/A)^j$. The integer $h=r-j$ is called the {\it height} of
  $\varphi$, where $r$ is the rank of $\varphi$.

  (2) We call $\varphi$ {\it supersingular} if
      $v\text{-rank}(\varphi)=0$.   
\end{defn}

The $v$-rank $j$ of $\varphi$ satisfies $0\le j\le r-1$ and drops
under specialization as $j$ is the \'etale rank of $\varphi[\grp]$.
Let
\begin{equation}
  \label{eq:SS.1}
  (\scrM^r_K)_{\le j}(\ol \F_v):=
  \{(\varphi,\bar \eta)\in (\scrM^r_K)_{}(\ol \F_v) \, |\,
  v\text{-rank}(\varphi)\le j\}
\end{equation}
be the closed subset consisting of all Drinfeld $A$-modules 
of $v$-rank $\le j$. We regard $(\scrM^r_K)_{\le j}$ 
as a closed subscheme of
$\scrM^r_K$ with the induced reduced structure. 
One can show that each stratum $(\scrM^r_K)_{\le j}$ is stable under the
  $\Gal(\ol \F_v/\F_v)$-action. Thus, each $(\scrM^r_K)_{\le j}$ is defined
  over $\F_v$.
Put $(\scrM^r_K)_{(j)}:=(\scrM^r_K) _{\le j}\setminus(\scrM^r_K)_{\le j-1}$,
which is a reduced locally 
closed subscheme consisting of all Drinfeld $A$-modules of
$v$-rank $j$. One has the $v$-rank stratification
\begin{equation}
  \label{eq:SS.2}
  \scrM^r_K =\coprod_{0\le j\le r-1} (\scrM^r_K)_{(j)}.
\end{equation}

For each Drinfeld $A$-module $\varphi$ over $\ol \F_v$, the
associated $v$-divisible $A_v$-module $\varphi[\grp^\infty]=\varphi
[\grp^\infty]^{\rm loc}\oplus \varphi[\grp^\infty]^{\rm et}$
canonically decomposes into the local and \'etale parts. 
The \'etale part is isomorphic to $(F_v/A_v)^j$, where $j$ is the
$v$-rank of $\varphi$, while the local part
is a formal $A_v$-module of height $h$.  
By \cite[Proposition 1.17]{drinfeld:1}, any two formal
$A_v$-modules of the same height $h$ 
over $\ol \F_v$ 
are isomorphic. Thus, the associated $v$-divisible $A_v$-modules
$\varphi[\grp^\infty]$ are geometrically constant on each $v$-rank
stratum. So each $v$-rank stratum $(\scrM^r_K)_{(j)}$ is a \emph{central 
leaf} in the sense of Oort. 

\begin{thm}\label{SS.2}
  Assume that $K=K_v K^v$ is fine, where $K_v=G(A_v)$. 
  For any integer $j$ with $0\le j\le r-1$, the subscheme
  $(\scrM^r_K)_{(j)}$ is non-empty, smooth over $\F_v$, of pure dimension $j$. The
  closed scheme $(\scrM^r_K)_{\le j}$ is smooth over $\F_v$ of pure dimension $j$.
\end{thm}
\begin{proof}
  See \cite[Theorem 10.33]{boyer:llc}.\qed
\end{proof} 

We will see in Lemma~\ref{HI.3} that for each integer $h$ with $1\le h\le r$, the scheme $(\scrM^r_K)_{(r-h)}$ coincides with the scheme $(\scrM^r_K)^{(h)}$ from Definition~\ref{HI.4} on which the height is equal to $h$.
Then the $v$-rank stratification is
the same as the stratification by height: 
$\scrM^r_K=\coprod_{1\le h\le r} (\scrM^r_K)^{(h)}$. However, we will see 
in Section~\ref{sec:HI.2} that the height stratification
behaves better when we work with the Satake compactification $\ol
\scrM^r_K$ of $\scrM^r_K$.
Namely, on $\ol\scrM^r_K$ the $v$-rank goes down if the height increases, but also along the boundary $\ol\scrM^r_K\setminus\scrM^r_K$. So this depends on the leading coefficient of a Drinfeld module as in equation \eqref{eq:2.1} \emph{and} the lowest non-vanishing coefficient, while the height only depends on the lowest non-vanishing coefficient; see Lemma~\ref{HI.3}(2) for the precise statement.

We give another reason why the stratification by height is better behaved than the one by $v$-rank. We have $(\ol \scrM^r_K)_{(r-1)}=(\scrM^r_K)_{(r-1)}$, because $\rank(\varphi)\ge 1+v\text{-rank}(\varphi)$, and hence
the boundary $\partial \ol \scrM^r_K$ is contained in $(\ol
  \scrM^r_K)_{\le r-2}$. 
The $v$-rank stratum 
$(\ol \scrM^r_K)_{(r-2)}$ then consists of two parts:
$(\scrM^r_K)_{(r-2)}$ and $(\partial \ol\scrM^r_K)_{(r-2)}$. We will see that 
$(\scrM^r_K)_{(r-2)}$ is of dimension $r-2$. By (the proof of) 
Proposition~\ref{2.7}, the stratum $(\partial \ol\scrM^r_K)_{(r-2)}$ is of dimension $r-2$, too. Therefore, $(\scrM^r_K)_{(r-2)}$ will not be dense in $(\ol \scrM^r_K)_{(r-2)}$, while the corresponding density statement for the height stratification holds true by Theorem~\ref{HI.5}(1).

\subsection{Supersingular Drinfeld modules}
\label{sec:SS.2}

Let $\F_{v^m}\subset\ol\F_v$ be the finite extension of $\F_v$ of degree $m$.

\begin{prop}\label{SS.3}
  Let $\varphi$ be a Drinfeld $A$-module of rank $r$ 
  over one of the finite fields $\F_{v^m}$. The following statements are
  equivalent. 

  \begin{enumerate}
  \item [(a)] $\varphi$ is supersingular.
  \item [(b)] The endomorphism algebra 
  $D:=\End^0(\varphi\otimes \ol \F_v):=\End(\varphi\otimes \ol
  \F_v)\otimes_A F$ of $\varphi$ over $\ol \F_v$ is a central
  division $F$-algebra of dimension $r^2$.
  \item [(c)] Some power of the Frobenius endomorphism of $\varphi$
    lies in $A$.
  \end{enumerate}
In this case the Hasse invariants of $D$ are $\inv_v D=1/r$ and $\inv_\infty D=-1/r$.
\end{prop}
\begin{proof}
  See \cite[Proposition~4.1]{gekeler:finite}. \qed
\end{proof}

We recall the function field analogue of the Honda-Tate theorem proved
by Drinfeld \cite{drinfeld:2}.
We also refer to \cite{jkyu:ht} for a clear exposition with detailed
proofs.  

\begin{defn}\label{SS.4}
  A {\it Weil number over
  $\F_{v^m}$ of rank $r$} is an element $\pi$ of $\ol F$ satisfying
  the following property. 
\begin{enumerate}
\item $\pi$ is integral over $A$.
\item There is only one place $w$ of $F(\pi)$ which is a zero of
  $\pi$, i.e.~with $w(\pi)>0$. This place lies over $v$.
\item There is only one place of $F(\pi)$ lying over $\infty$.
\item $|\pi|_\infty=\#\F_v^{m/r}$, where $|\ |_\infty$ is the unique
  extension to $F(\pi)$ of the normalized absolute value 
  $|\ |_\infty$ on $F$.
\item $[F(\pi):F]$ divides $r$.
\end{enumerate}
\end{defn}

Let $W^r_{v^m}$ denote the set of Galois conjugacy classes of 
Weil numbers of rank $r$ over $\F_{v^m}$. The analogous Honda-Tate
theorem \cite{drinfeld:2,jkyu:ht} states that the map sending
$\varphi$ 
to its  
Frobenius endomorphism $\pi_\varphi$ induces a bijection 
\[ \{\text{isogeny classes of Drinfeld $A$-modules of rank $r$ over
  $\F_{v^m}$} \}\isoto W^r_{v^m} \]
Note that if $\pi$ is a Weil number of rank $r$ over $\F_{v^m}$ then
it is also a Weil number of rank $nr$ over $\F_{v^{nm}}$ for any
integer $n\ge 1$; namely, one has 
$W^{nr}_{v^{nm}}=W^r_{v^m}$. 

A Weil number $\pi$ is said to be {\it supersingular} if the
corresponding isogeny class of Drinfeld modules is supersingular. 
By Proposition~\ref{SS.3}, a Weil number $\pi$ 
is supersingular if and only if $\pi^n\in A$ for some $n\ge 1$. 

Observe that if $\pi\in W^r_{v^m}$, 
then the element $N_{F(\pi)/F}(\pi)\in A$
generates the ideal $\grp^{m[F(\pi):F]/r}$; this follows from
Definition~\ref{SS.4} (4) and the product formula. 
In particular, the ideal 
$\grp^m$ must be principal, 
because $[F(\pi):F]|r$ by (5) and $\grp^m$
is a power of the principal ideal $\grp^{m[F(\pi):F]/r}$. 
Therefore, if $\pi$ is a Weil number over $\F_v$ 
then $\grp$ is necessarily principal. 
Let $m_v$ be the
order of $\grp$ in the ideal class group $\Cl(A)$.
Then any generator $P$ of $\grp^{m_v}$ is a
supersingular Weil
number of rank 1 over $\F_{v^{m_v}}$, or a supersingular 
Weil number of rank $r$ over $\F_{v^{{m_v}r}}$ for any positive integer
$r$. In particular, there always exists a supersingular Drinfeld
$A$-module over $\ol \F_v$ of rank $r$ for any $r\ge 1$.

\begin{lemma}\label{SS.5}
  Let $P$ be a generator of the principal ideal $\grp^{m_v}$. Every
  supersingular Drinfeld $A$-module $\varphi$ of rank $r$ over $\ol \F_v$ 
  admits a unique model $\varphi'$ over $\F_{v^{{m_v}r}}$,  up to
  $\F_{v^{{m_v}r}}$-isomorphism, with Frobenius endomorphism
  $\pi_{\varphi'}=P$.     
\end{lemma}

Lemma~\ref{SS.5} is an analogous result of \cite[Proposition
6]{ghitza:thesis} that every supersingular elliptic curve over
$\Fpbar$ admits a canonical $\F_{p^2}$-model. 
The proof is similar and omitted. 
We call $\varphi'$ the canonical model of $\varphi$ over 
$\F_{v^{{m_v}r}}$.
It defines a natural $\F_{v^{m_v r}}$-structure on the space 
$\omega(\varphi)=\Lie(\varphi)^\vee$ of invariant differential forms
of $\varphi$. 
Lemma~\ref{SS.5} shows that  
the canonical model exists
over $\F_{v^{m_v r}}$ for all $\varphi$; $\F_{v^{m_v r}}$ is the
smallest field of definition by the Honda-Tate theory. 
We will show by another method 
in the next subsection 
that $\omega(\varphi)$ can actually have a natural
$\F_{v^r}$-structure for all $\varphi$. 

\subsection{Supersingular and algebraic Drinfeld modular forms mod $v$}
\label{sec:SS.3}

Let $\Sigma(r,v)$ be the set of isomorphism classes of
supersingular Drinfeld
$A$-modules of rank $r$ over $\ol \F_v$. 
We fix a member $\varphi_0$ in $\Sigma(r,v)$. Let
$\O_D:=\End(\varphi_0)$ and $D:=\End(\varphi_0)\otimes F$. 
Then $D$ is 
the central division $F$-algebra (unique up to isomorphism)
ramified precisely at $\infty$ and $v$, with invariants 
$\inv_\infty(D)=-1/r$ and
$\inv_v(D)=1/r$, see Proposition~\ref{SS.3}, and $\O_D$ is a
maximal $A$-order in $D$; see \cite[Proposition~1.7]{drinfeld:1} and
\cite[Theorem 1]{jkyu:ht}.
Denote by $G'$ the group scheme over $A$
associated to the multiplicative group of $\O_D$: For any $A$-algebra
$R$, the group of $R$-valued points of $G'$ is 
$G'(R)=(\O_D\otimes_A R)^\times$. By \cite[Theorem 4.3]{gekeler:finite}
(cf. \cite[Corollary 3.3]{yu-yu:ssd}), there is a natural bijection 
\begin{equation}
  \label{eq:SS.3}
  \Sigma(r,v)\simeq G'(F)\backslash G'(\A^{\infty})/G'(\wh A). 
\end{equation}

Let $\scrS_{K}\subset \scrM^r_{K}$ be the supersingular locus with
$K=K_vK^v$. It is an affine scheme, finite over $\F_v$.
Let $ \scrS_{K_v}:=\projlimw{K^v} \scrS_{K_v K^v}$, 
where $K^v$ runs through open compact subgroups of $G(\A^{\infty v})$. 
The scheme $\scrS_{K_v}$ has a right continuous action of
$G(\A^{\infty v})$, and for any $K^v\subset G(\A^{\infty v})$, 
one has $\scrS_{K_v}/K^v=\scrS_{K_v K^v}$. 
Fix an isomorphism $\eta_0:(\A^{\infty v})^r \isoto
V^{(v)}(\varphi_0)$ such that $\eta_0(\wh
A^{(v)})^r=T^{(v)}(\varphi_0)$. This isomorphism induces an
identification 
\begin{equation}
  \label{eq:SS.id}
  G(\A^{\infty v})\isoto G'(\A^{\infty v}), \quad g\mapsto g'=\eta_0 g
  \eta_0^{-1}, \quad g\in G(\A^{\infty v}), 
\end{equation}
and gives a base point $(\varphi_0,\eta_0)$ in $\scrS_{K_v}$. We shall
use it to identify the open
compact subgroups of $G(\A^{\infty v})$ and $G'(\A^{\infty v})$.
Let 
\begin{equation}
  \label{eq:SS_rig}
  \scrS_{K_v}^{\rm rig}:=\left \{ (\varphi,\eta,\alpha)\, \big| \,
  (\varphi,\eta)\in \scrS_{K_v}, \text{$\alpha:\varphi\to \varphi_0$
  is a quasi-isogeny}\, \right \}.
\end{equation}
Here ``rig'' indicates that the pairs $(\varphi,\eta)$ are rigidified
by a quasi-isogeny to the base object $\varphi_0$. 
This space admits a natural left action of $G'(F)$. 
We will describe $\scrS_{K_v}^{\rm rig}$ in terms of $G'(\A^\infty)$. Since 
any two supersingular Drinfeld modules are isogenous, the natural map 
$\scrS_{K_v}^{\rm rig}\to \scrS_{K_v}$ is surjective, and it
induces an isomorphism 
$G'(F)\backslash \scrS_{K_v}^{\rm rig}\simeq \scrS_{K_v}$, because for a given Drinfeld module $\varphi$ over $\ol\F_v$ the set $\{\alpha\colon\varphi\to\varphi_0\text{ a quasi-isogeny}\}$ is a principal homogeneous space under $\End^0(\varphi_0\otimes\ol\F_v)^\times=G'(F)$.

Write $A_v=\F_v[[ z_v]]$, where $z_v$ is a fixed uniformizer of $A_v$.
(One can of course choose $z_v\in A_{(v)}$, if necessary.)
The completions of the maximal unramified extensions
of $A_v$ and $F_v$ are $\ol \F_v[[ z_v]]$
and $\ol \F_v(( z_v))$, respectively. 
Let $\sigma_v$ be the Frobenius
map on $\ol \F_v[[ z_v]]$ and $\ol \F_v((z_v))$ induced by the 
map $x \mapsto x^{\#\F_v}$ on $\ol\F_v$ and satisfying $\sigma_v(z_v)=z_v$. 

\begin{defn}\label{DefDieuMod}
A \emph{covariant Dieudonn\'e module of rank $r$} over $\ol\F_v$ is a free $\ol\F_v[[z_v]]$-module $M$ of rank $r$ together with a $\sigma_v^{-1}$-semilinear map $V\colon M\to M$ such that $z_v\cdot M\subset V(M)$. Here $\sigma_v^{-1}$-semilinear means that $V(fm)=\sigma_v^{-1}(f)\cdot V(m)$ for $f\in \ol\F_v[[z_v]]$ and $m\in M$.
\end{defn}

The covariant equi-characteristic \dieu modules 
are the twisted linear duals of
contravariant \dieu modules defined in Laumon~\cite{laumon:1}. More precisely, for a covariant \dieu module $(M,V)$ as in Definition~\ref{DefDieuMod} the pair consisting of $M^\vee:=\Hom_{\ol\F_v[[z_v]]}(M,\ol\F_v[[z_v]])$ and $V^\vee\colon M^\vee\to M^\vee$ is a contravariant \dieu module as in \cite{laumon:1}. Here $V^\vee$ is $\sigma_v$-semilinear in the sense that $V^\vee(fm^\vee)=\sigma_v(f)\cdot V^\vee(m^\vee)$. In terms of \cite{HartlKim,HartlDict,HartlPSp} the pair $(M^\vee,V^\vee)$ is a \emph{local shtuka}. The following argument can also be reformulated in terms of local shtukas.

Let $(M_0,V_0)$ be the covariant \dieu module of $\varphi_0$ and
extend $V_0$ to $N_0:=M_0[1/z_v]$. One has $M_0/V_0M_0=\Lie(\varphi_0)$ and $\End(M_0)=\O_{D_v}=G'(A_v)$.
Let 
\[ X_v:=\{\,\text{$\ol\F_v[[z_v]]$-lattices $M$ in $N_0$, such that $(M,V_0)$ is a \dieu module}\,\}. \]
and let $X^v$ be the set of pairs $(L^{(v)},\eta)$, where $L^{(v)}\subset
  V^{(v)}(\varphi_0)$ is an $\wh A^{(v)}$-lattice and 
$\eta: (\wh A^{(v)})^r\isoto
  L^{(v)}$ is an isomorphism.

\begin{lemma}\label{LemmaG'Adele}
Let $F$ be a global function field with finite constant field $\F_q$ and let $D$ be a finite dimensional division algebra over $F$ with center $F$. Let $G'$ be the algebraic group over $F$ defined on $F$-algebras $R$ by $G'(R)=(D\otimes_F R)^\times$. Let $S$ be a non-empty set of places of $F$ and let $\bbA^S$ be the prime to $S$ adele ring of $F$. Then the topological space $G'(F)\backslash G'(\bbA^S)$ is compact.
\end{lemma}

\begin{proof}
Let $N_{D/F}\colon D\to F$ be the reduced norm of $D$ and let $G'(\bbA)_1$ be the kernel of the group homomorphism 
\[
|N_{D/F}(\,.\,)|\colon G'(\bbA) \longrightarrow q^\bbZ,\quad g=(g_x)_x \mapsto \prod_{x}|N_{D/F}(g_x)|_x
\]
where the product runs over all places $x$ of $F$, $g_x$ is the component of the adele $g$ at the place $x$, the norm $N_{D/F}$ is extended to $N_{D/F}\colon D\otimes_F F_x\to F_x$, and $|\,.\,|_x\colon F_x\to q^{\deg(x)\bbZ}$ is the normalized absolute value on $F_x$. By the product formula \cite[Chapter~II, \S\,12, Theorem]{CasselsFroehlich} the group $G'(F)$ is contained in $G'(\bbA)_1$. Since the center of $G'$ is the maximal $F$-split torus in $G'$, the quotient $G'(F)\backslash  G'(\bbA)_1$ is a compact topological space by \cite[Korollar~2.2.7]{Harder69}, see also \cite[Theorem~A.5.5(i)]{Conrad12}. Therefore also its quotient
\[
G'(F)\backslash  G'(\bbA)_1\big/\bigl(G'(\bbA)_1\cap\prod'_{x\in S}G'(F_x)\bigr)=G'(F)\backslash \bigl(G'(\bbA)_1\cdot\prod'_{x\in S}G'(F_x)\bigr)\big/\prod'_{x\in S}G'(F_x)
\]
is compact.

By \cite[(33.4) Theorem]{ReinerMO} the map $N_{D/F}\colon G'(F_x)\to F_x^\times$ is surjective for every place $x$. Thus the quotient $G'(\bbA)\big/\bigl(G'(\bbA)_1\cdot \prod'_{x\in S}G'(F_x)\bigr)\hookrightarrow q^{\bbZ/d\bbZ}$ is finite, where $d\ne0$ is the greatest common divisor of $\deg(x)$ for all $x\in S$. It follows that $G'(F)\backslash G'(\bbA^S)=G'(F)\backslash G'(\bbA)/\prod'_{x\in S}G'(F_x)$ is a finite disjoint union of cosets of the compact topological space 
\[
\textstyle G'(F)\backslash \bigl(G'(\bbA)_1\cdot\prod'_{x\in S} G'(F_x)\bigr)\big/\prod'_{x\in S}G'(F_x).
\]
This proves the lemma. \qed
\end{proof}

\begin{lemma}\label{SS.6}
  There are natural $G(\A^{\infty v})$-equivariant isomorphisms
  \begin{eqnarray}
    \xi^{\rm rig}: \scrS_{K_v}^{\rm rig} & \isoto & X_v\times X^v\;\isoto\;G'(\A^\infty)/G'(A_v)\qquad\text{and} \nonumber \\[1mm]
    \label{eq:SS.6}
    \xi: \scrS_{K_v} & \isoto & G'(F)\backslash G'(\A^\infty)/G'(A_v)
  \end{eqnarray}
  which send the
  base point $(\varphi_0,\eta_0)$ to the class of $1\in G'(\A^\infty)$. In
  particular, for any open compact subgroup $K^v\subset
  G(\A^{\infty v})\simeq G'(\A^{\infty v})$, there is an isomorphism 
  \begin{equation}
    \label{eq:SS.7}
    \xi_{K^v}:
  \scrS_{K}\isoto  G'(F)\backslash G'(\A^\infty)/G'(A_v)K^v
  \end{equation}
  which
  is compatible with the prime-to-$v$ Hecke action.
\end{lemma}
\begin{proof} 
  The proof is similar to that of \cite{ghitza:thesis} or
  \cite{yu:thesis, yu:smf}. For a member $(\varphi,\eta,\alpha)$ in 
  $\scrS^{\rm rig}_{K_v}$, we can replace $(\varphi, \alpha)$ by a
  prime-to-$v$ quasi-isogeny so that $\eta$ induces an isomorphism 
$\eta: (\wh A^{(v)})^r \isoto T^{(v)}(\varphi)$.
Then 
  $\scrS^{\rm rig}_{K_v}$ can be also interpreted as the set of
  isomorphism classes of such triples $(\varphi,\eta,\alpha)$ where $\eta$ satisfies the integrality condition $\eta\bigl((\wh A^{(v)})^r\bigr)= T^{(v)}(\varphi)$. 
  Taking the \dieu and
  prime-to-$v$ Tate modules, we obtain an isomorphism 
  $\scrS^{\rm rig}_{K_v}\simeq X_v\times X^v$. 
  Since any two supersingular \dieu modules are isomorphic
  \cite[Proposition 1.17]{drinfeld:1}, the
  action of $G'(F_v)$ on $X_v$ is transitive and one has
  an isomorphism $G'(F_v)/G'(A_v)\isoto X_v$, $g\mapsto gM_0$. 
  For each element $g\in
  G'(\A^{\infty v})$, one associates a pair $(L^{(v)},\eta)$ in $X^v$ 
  by taking $L^{(v)}:=g\cdot T^{(v)}(\varphi_0)$ and $\eta=g \eta_0:
  (\wh A^{(v)})^r \isoto L^{(v)}$. This gives an isomorphism 
  $G'(\A^{\infty v})\simeq X^v$ and we have proven  
  $\scrS^{\rm rig}_{K_v}\simeq G'(F_v)/G'(A_v)\times
  G'(\A^{\infty v})$. Everything else follows immediately. \qed
\end{proof}

For $M\in X_v$, we define the skeleton $M^\diamond$ of $M$ by
$M^\diamond:=\{m\in M | V_0^r m=z_v m \}$. This is an 
equi-\ch \dieu module over
$\F_{v^r}$ (as in Definition~\ref{DefDieuMod} but with $\ol\F_v$ replaced by $\F_{v^r}$, which is the field extension of $\F_v$ of degree $r$) and one has $M^\diamond \otimes_{A_{v^r}} \ol\F_v[[z_v]]=M$. The construction $M\mapsto M^\diamond$ is functorial and
  it defines an $\F_{v^r}$-subspace $\omega(M)^\diamond\subset 
\omega(M):=(M/V_0M)^\vee$. The endomorphism ring $\End(M_0)=\O_{D_v}$ acts
on $M/V_0M$ and this induces an isomorphism
$\F_{D_v}:=\O_{D_v}/{\rm rad}(\O_{D_v})\simeq
\End((M/V_0M)^\diamond)=\F_{v^r}$. 
Set 
\begin{equation}
  \label{eq:SS.8}
  G(v):=\F_{D_v}^\times\simeq\F_{v^r}^\times, \quad U(v):=\ker \bigl(G'(A_v)=\O_{D_v}^\times \to
G(v)\bigr).
\end{equation} 
The above isomorphism identifies $G(v)$ with $\wt G'(\F_v)$, 
where $\wt G'$ is the maximal reductive quotient of 
$G'\otimes_A (A_v/z_v)$. 

Consider the space $X_v^\omega$ which consists of
pairs $(M,e)$ where $M\in X_v$ and $e\in \omega(M)^\diamond$ is an $\F_{v^r}$-generator. 
Fix a base point $(M_0,e_0)\in X_v^\omega$. The group $G'(F_v)$ acts
transitively on $X_v^\omega$ and one has a bijection
$G'(F_v)/U(v)\simeq X_v^\omega$. 

For any finite-dimensional vector space $W$ over $\ol \F_v$, denote by 
$C^\infty(G'(F)\backslash  G'(\A^\infty),W)$ the space of locally
constant functions $f':G'(F)\backslash  G'(\A^\infty)\to W$. We
equip it with the right regular translation of $G'(\A^\infty)$, that is $(g\cdot f')(x):=f'(xg)$ for $g\in G'(\A^\infty)$ and $x\in G'(F)\backslash  G'(\A^\infty)$. Then
$C^\infty(G'(F)\backslash  G'(\A^\infty),W)$ is an admissible smooth
representation of $G'(\A^\infty)$.   
Indeed, the quotient $G'(F)\backslash  G'(\A^\infty)$ is a compact topological
space by Lemma~\ref{LemmaG'Adele}. Thus, every vector $f'\in C^\infty(G'(F)\backslash
G'(\A^\infty),W)$ takes on only finitely many values in $W$, and hence is fixed by an open subgroup of
$G'(\A^\infty)$. Moreover, for each open compact subgroup
$K'\subset G'(\A^\infty)$, the subspace $C^\infty(G'(F)\backslash
G'(\A^\infty),W)^{K'}$ of $K'$-fixed vectors
is equal to $C^\infty(G'(F)\backslash
G'(\A^\infty)/{K'},W)$, which is finite dimensional, because the set $G'(F)\backslash  G'(\A^\infty)/K'$ is finite.

Now assume that $W$ is equipped with a finite dimensional irreducible representation $\rho\colon G'(A_v)\to \Aut(W)$ of $G'(A_v)$. Following \cite{gross:amf}, we define
\emph{the space $M_\rho^{\rm alg}(G')$ of algebraic modular forms (mod $v$)
of weight $\rho$ on $G'$} by
\begin{equation}
  \label{eq:SS.9}
  \begin{split}
  M_\rho^{\rm alg}(G';W):=\{& f'\in C^\infty(G'(F)\backslash
  G'(\A^\infty),W)\ \big| \\ 
  & f'(xk_v)=\rho(k_v^{-1}) f'(x), \forall x \in
  G'(F)\backslash G'(\A^\infty), \ k_v\in G'(A_v)\}.    
  \end{split}
\end{equation}
If ${K^{v}} \subset G'(\A^{\infty v})=G(\A^{\infty v})$ 
is an open compact subgroup, we
write 
\[
\begin{split}
M_\rho^{\rm alg}(G',&{K^{v}};W):=M_\rho^{\rm
  alg}(G';W)^{{K^v}}=\\
& \{f'\in M_\rho^{\rm alg}(G';W)\ \big|\ f'(xk^v)=f'(x), 
\forall x \in G'(F)\backslash G'(\A^\infty), \ k^v\in K^v\}
\end{split}
\]
for the subspace of \emph{algebraic modular forms with level ${K^v}$}.

Let $S_k(r,K_vK^v,\ol \F_v):=H^0(\scrS_{K}\otimes_{\F_v}\ol\F_v,
i^* \omega_{K}^{\otimes k}\otimes\ol\F_v)$ be the space of supersingular 
Drinfeld modular forms of rank $r$, weight $k$ with level $K^v$ over
$\ol \F_v$, where $i:\scrS_{K}\to \scrM^r_{K}$ is the inclusion
map. Note that $\scrS_{K_v}/K^v=\scrS_{K_v K^v}$ implies 
\begin{equation}\label{eq:SkFixPts}
S_k(r,  K_vK^v,\ol \F_v)\simeq H^0(\scrS_{K_v}\otimes_{\F_v}\ol\F_v,i^* \omega^{\otimes k}\otimes\ol\F_v)^{K^v},
\end{equation}
and $H^0(\scrS_{K_v}\otimes_{\F_v}\ol\F_v,i^* \omega^{\otimes k}\otimes\ol\F_v)=\dirlim[\wt K^v] S_k(r,  K_v\wt K^v,\ol \F_v)$.

\begin{prop}\label{SS.7}
  Let $\chi:G'(A_v)\to \ol \F_v ^\times$ be the character of the
  $1$-dimensional representation $\omega(\varphi_0)$. For any integer
  $k\ge 1$ and open compact subgroup $K^v\subset G(\A^{\infty v})$,
  there is an isomorphism $S_k(r,
  K_vK^v,\ol \F_v) \simeq M^{\rm alg}_{\chi^k}(G',K^v;\ol\F_v)$
  which is compatible with the prime-to-$v$ Hecke action. 
\end{prop}
\begin{proof}
  By what was said before the proposition, it is equivalent to prove that there is a $G(\A^{\infty v})$-equivalent
  isomorphism $H^0(\scrS_{K_v}\otimes_{\F_v}\ol\F_v,i^* \omega^{\otimes k}\otimes\ol\F_v) 
  \simeq M^{\rm alg}_{\chi^k}(G';\ol\F_v)$. By Lemma~\ref{SS.6}, 
  the first space
  consists of all $G'(F)$-invariant locally constant sections $f$ on
  $\scrS_{K_v}^{\rm rig}= X_v\times G(\A^{\infty v})$ with $f(M,g^v)\in \omega(M)^{\otimes
  k}$. We lift each section $f$ to a function $f':X^\omega_v \times
  G(\A^{\infty v})\to \ol \F_v$ by $f'((M,e),g^v):=(e^{k})^{-1}
  f(M,g^v)$, where $e^k:=e\otimes \dots \otimes e$ ($k$ times) is an
  element in $\omega(M^\diamond)^{\otimes k}$ and it induces an
  isomorphism $e^k:\ol \F_v\isoto \omega(M)^{\otimes k}$. By
  $X_v^\omega=G'(F_v)/U(v)$, this defines
  an $G(\A^{\infty v})$-equivariant map 
\[ H^0(\scrS_{K_v}\otimes_{\F_v}\ol\F_v,i^* \omega^{\otimes k}\otimes\ol\F_v)\to C^\infty(G'(F)\backslash
  [G'(F_v)/U(v)\times G(\A^{\infty v})],\ol \F_v), \]
  which is injective, because $f$ can be recovered as 
  $f(M,g^v)=e^k\cdot f'((M,e),g^v)$. For $g_v\in
  G'(F_v)$ and $k_v\in G'(A_v)$, if $g_v(M_0,e_0)=(M,e)$ then 
$  g_v
  k_v(M_0,e_0)=(M, g \chi(k_v) e_0)=(M,\chi(k_v)g e_0)=(M,\chi(k_v)e)$.
  It is easy to see that 
  \begin{equation}
    \label{eq:Prop.SS.7}
  f'(g_v k_v, g^v)=(\chi(k_v)^k e^k)^{-1}\cdot f(M,g^v)=\chi^k(k_v)^{-1}f'(g_v, g^v).
  \end{equation}
  Therefore, we obtain
  an injection 
  $H^0(\scrS_{K_v}\otimes_{\F_v}\ol\F_v,i^* \omega^{\otimes k}\otimes\ol\F_v)\to M^{\rm
  alg}_{\chi^k}(G';\ol\F_v)$ which is  $G(\A^{\infty v})$-equivariant.
  To see that it is surjective, let $f'\in
  M_{\chi^k}^{\rm alg}(G';\ol\F_v)$, then \[ f'\in C^\infty(G'(F)\backslash
  [G'(F_v)/U(v)\times G(\A^{\infty v})], \ol \F_v)=
  C^\infty(X_v^\omega\times G(\A^{\infty v}),\ol \F_v)^{G'(F)}\] because
  $\chi^k(U(v))=\{1\}$. We define $f$ on $((M,e),g^v)$ as
  $f((M,e),g^v):=e^k (f'((M,e),g^v))\in \omega(M)^{\otimes k}$. As in
  \eqref{eq:Prop.SS.7} we see that this does not depend on the choice
  of $e$ because every other $\tilde e$ is of the form $c\cdot e$ with
  $c\in \F_{v^r}^\times=\chi(G'(A_v))$, that is, $c=\chi(k_v)$ for
  $k_v\in G'(A_v)$. So $f$ descends to a section $(f:(M,g^v)\mapsto
  f((M,e), g^v))\in H^0(\scrS_{K_v}\otimes_{\F_v}\ol\F_v, i^* \omega^{\otimes k}\otimes\ol\F_v)$. \qed 
\end{proof}

Since $\chi$ is of order $q_v^r-1=\#F_{v^r}^\times$, where $q_v:=\#\F_v$, the characters $\chi^k$ for
$k=1,\dots, q_v^r-1$ are all distinct irreducible 
representations of $G(v)$. Therefore,
\begin{equation}
  \label{eq:SS.11}
  C^{\infty}(G'(F)\backslash G'(\A^{\infty})/U(v)K^v,\ol
  \F_v)=\bigoplus_{k=1}^{q_v^r-1} M^{\rm alg}_{\chi^k} (G', K^v). 
\end{equation}
As a corollary of Proposition~\ref{SS.7}, we get a prime-to-$v$ Hecke
equivariant isomorphism 
\begin{equation}
  \label{eq:SS.12}
  \bigoplus_{k=1}^{q_v^r-1} S_k(r,K_vK^v,\ol \F_v)
  \simeq C^{\infty}(G'(F)\backslash G'(\A^{\infty})/U(v)K^v,\ol \F_v). 
\end{equation}
Let $\bf1$ be the constant function in $C^{\infty}(G'(F)\backslash
G'(\A^{\infty})/U(v)K^v,\ol \F_v)$ with value $1$ on each double coset.
This element maps under projection to an element still denoted by ${\bf1}\in
H^0(\scrS_{K}\otimes_{\F_v}\ol\F_v,i^* \omega_{K}^{\otimes q_v^r-1}\otimes\ol\F_v)$. Multiplication
by $\bf1$ gives a prime-to-$v$ Hecke equivariant isomorphism
\begin{equation}
  \label{eq:SS.13}
  {\bf1}: H^0(\scrS_{K}\otimes_{\F_v}\ol\F_v,i^* \omega_{K}^{\otimes k}\otimes\ol\F_v)\isoto
H^0(\scrS_{K}\otimes_{\F_v}\ol\F_v,i^* \omega_{K}^{\otimes k+q_v^r-1}\otimes\ol\F_v).
\end{equation}

\def\Mass{{\rm Mass}}

Recall that the \emph{mass} for $\Sigma(r,v)$ is defined by
\begin{equation}
  \label{eq:SS.14}
  \Mass(\Sigma(r,v)):=\sum_{\psi\in \Sigma(r,v)} \frac{1}{\#\Aut(\psi)}.
\end{equation}
For any open compact subgroup $K'\subset G'(A^\infty)$, the arithmetic
mass is defined by
\begin{equation}
  \label{eq:SS.arith_mass}
  \Mass(G',K'):=\sum_{i=1}^h \frac{1}{\# \Gamma_i}, \quad
  \Gamma_i:=c_i K' {c_i}^{-1}\cap G'(F), 
\end{equation}
where $c_1.\dots,c_h$ are coset representatives for the finite double coset
space $G'(F)\backslash G'(\A^\infty)/K'$. 
The mass formula (see \cite[2.5 and
5,11]{gekeler:mass} and \cite[Theorem~2.1]{yu-yu:ssd}, 
also see \cite{wei-yu:mass_division} and \cite{yu:mass_var} for
generalizations) states that
\begin{equation}
  \label{eq:SS.15}
\Mass(\Sigma(r,v))=\Mass(G',G'(\wh A))=\frac{h(A)}{q-1} \prod_{i=1}^{r-1}
\zeta_F^{\infty v}(-i),  
\end{equation}
where $h(A)$ is the class number of $A$ and
$\zeta_F^{\infty v}(s):=\prod_{w\neq \infty,v}=(1-(\#\F_w)^{-s})^{-1}$ is
the zeta function of $F$ with factors at $\infty$ and $v$ removed.

\begin{remark}\label{SS.8}
In \cite{gekeler:mass} Gekeler  
proved a recursive formula (referred as ``the transfer principle'') which computes 
explicitly the class number $h(\O_D)=\#\Sigma(v,r)$ 
for the case $F=\Fq(t)$ and $A=\Fq[t]$. 
Gekeler's transfer principle was generalized by F.-T Wei and the second author for an arbitrary hereditary $A$-order
$R$ in any central division $F$-algebra $D$ definite at $\infty$ 
(namely, $D_\infty$ is still a central division $F_\infty$-algebra 
but $D$ can be ramified at several finite places of $F$);
see \cite[Theorem 1.1]{wei-yu:classno}. Using the recursive formulas in
loc.~cit.,  one can compute the class number $h(R)$ of $R$ explicitly.  
\end{remark}

\begin{lemma}\label{SS.9}
Suppose $K=K_vK^v\subset G(\wh A)$ is fine with $K_v=G(A_v)$, then 
\begin{equation}
  \label{eq:SS.16}
  \begin{split}
  \dim_{\ol \F_v} &C^{\infty}(G'(F)\backslash
  G'(\A^{\infty})/U(v)K^v,\ol \F_v) \\
  & =[G(\wh A^{(v)}):K^v] h(A) \frac{q_v^r-1}{q-1} \prod_{i=1}^{r-1}
  \zeta_F^{ \infty v}(-i).    
  \end{split}
\end{equation}
\end{lemma}
\begin{proof}
  Clearly, we have \[ \begin{split}\dim_{\ol \F_v} C^{\infty}(G'(F)\backslash
  G'(\A^{\infty})/U(v)K^v,\ol \F_v)&=\# G'(F)\backslash
  G'(\A^{\infty})/U(v)K^v\\
  & = [G'(A_v):U(v)] \cdot [G(\wh A^{(v)}): K^v] \cdot \Mass(\Sigma(r,v)).
  \end{split}
  \] 
  As $G'(A_v)/U(v)\simeq \F_{q_v^r}^\times$, we have
  $[G'(A_v):U(v)]=q_v^r-1$. Then 
\[ \dim_{\ol \F_v} C^{\infty}(G'(F)\backslash
  G'(\A^{\infty})/U(v)K^v,\ol \F_v)=[G(\wh A^{(v)}):K^v] \cdot
  (q_v^r-1)\cdot \Mass(\Sigma(r,v)) \]
  and formula \eqref{eq:SS.16} follows from 
  the mass formula for
  $\Mass(\Sigma(r,v))$~\eqref{eq:SS.15}. \qed
\end{proof}

\section{Generalized Hasse invariants and $v$-rank strata}
\label{sec:HI} 

\subsection{Coefficient modular forms}
\label{sec:HI.1}
We keep the notation $F$, $\infty$, $A$, $v$ and $\grp\subset A$ from Section~\ref{sec:SS}. As in Section~\ref{sec:MF.2} let $(\olE_K,\ol \varphi_K)$ be the universal family on 
$\ol \bfM^r_{K}$ over $A_{(v)}$, 
where $K=K_vK^v\subset G(\A^\infty)$ 
is an open compact subgroup with $K_v=G(A_v)$. For any element $a\in A$, write
\[ 
\ol \varphi_{K,a}=\sum_{i=0}^{r\deg a} \ol \varphi_{K,a,i} \cdot \tau^i. 
\]
Then each $\ol \varphi_{K,a,i}\in H^0(\ol \bfM^r_{K},{\olE_K}^{1-q^i})= H^0(\ol
\bfM^r_{K},\omega_K^{\otimes q^i-1})$ is a Drinfeld modular form of rank $r$,
and weight $q^i-1$ over $A_{(v)}$. These are called \emph{coefficient
modular forms}. Coefficient modular forms of rank $2$ were studied
by Gekeler~\cite{gekeler:coeff} and of higher rank by Basson, Breuer
and Pink~\cite{BBP}.

Suppose $K(\grn)\subset K$ for a non-zero ideal
$\grn$ of $A$. Then the moduli scheme $\ol \bfM^r_{K}$ and 
Drinfeld modular forms $\ol \varphi_{K,a,i}$ are even defined over
$A[\grn^{-1}]$ and not just over $A_{(v)}$. 

Note that $\ol \bfM^r_K$ is normal and therefore 
the notions of Cartier divisors and Weil divisors of $\ol \bfM^r_K$ 
are the same.
One can consider the (Cartier) divisor $V(\ol
\varphi_{K,a,i})$ which is defined as the zero section of $\ol
\varphi_{K,a,i}$ on $\ol \bfM^r_{K}$, or the intersection of
several such divisors.  
For example, if $a\not\in A[\grn^{-1}]^\times$.
then $V(\ol \varphi_{K,a,0})=\ol \bfM^r_{K}\otimes_{A[\grn^{-1}]} A[\grn^{-1}]/(a)$, because $\varphi_{K,a,0}=\gamma(a)$. If the
prime ideal $\grp=(a,b)$ is generated by elements $a$ and $b$, then 
the intersection 
$V(\ol \varphi_{K,a,0})\cap V(\ol \varphi_{K,b,0})$ 
is the fiber of $\ol \bfM^r_K$ over $\Spec \F_v$.  

By Theorem~\ref{S.2}, we have $\pi^*_{\wt K,K} (\ol\varphi_{K,a,i})= \ol 
\varphi_{\wt K, a,i}$ for fine open compact subgroups $\wt K\subset
K$. Thus, 
the image of $\ol \varphi_{K,a,i}$ in $M_{q^i-1}(r,K_v, A_{(v)})$ is
well-defined, see \eqref{eq:MF.9} in Definition~\ref{MF.5}. We denote it by $\ol \varphi_{a,i}$.

Let $L$ be an $A_{(v)}$-algebra. Recall from Theorem~\ref{ThmSmoothHeckeAction} that $M_k(r,K_v,L)^{K^v}$ is an $\calH^{\infty v}_L$-module, where $\calH^{\infty v}_L:=\calH_{L}(G(\A^{\infty v}),K^v)$, and if $L$ is an $\F_v$-algebra, also $\wt M_k(r,K_vK^v,L)=\wt M_k(r,K_v,L)^{K^v}$ is an $\calH^{\infty v}_L$-module.

\begin{lemma} \label{HI.1}\ 

{\rm (1)} For any $a\in A$, $0\le i\le r \deg a $ and $g\in
G(\A^{\infty v})$, one has $T_g\cdot \ol \varphi_{K^g, a,i}
=\ol \varphi_{K,a,i}$. The Drinfeld modular form 
$\ol \varphi_{a,i}$ is fixed by $G(\A^{\infty v})$.
We have 
\[ {\bf1}_{K^v g K^v}* \ol \varphi_{K,a,i}=\#(K^v g K^v/K^v) 
\cdot \ol \varphi_{K,a,i}, \]
where ${\bf1}_{K^v g K^v}\in \calH^{\infty v}_{\Z}$ is the characteristic
function of $ {K^v g K^v}$.

{\rm (2)} For every $A_{(v)}$-algebra $L$ the multiplication by $\ol \varphi_{K,a,i}$ gives rise to
a morphism of Hecke modules
\begin{equation*}
  \ol \varphi_{K,a,i}: M_k(r, K_v, L)^{K^v}\to M_{k+q^i-1}(r, K_v, L)^{K^v}.
\end{equation*}  

{\rm (3)} For every $\F_v$-algebra $L$ the multiplication by $\ol \varphi_{K,a,i}$ gives rise to
a morphism of Hecke modules
\begin{equation*}
  \ol \varphi_{K,a,i}: \wt M_k(r, K_vK^v, L)\to \wt M_{k+q^i-1}(r, K_vK^v, L).
\end{equation*}  
\end{lemma}
\begin{proof}
  (1) The first statement follows from the functorial property of the
      Satake compactification; see \eqref{eq:MF.7} and
      \eqref{eq:MF.10}. 
     It follows from  $T_g\cdot \ol \varphi_{K^g, a,i}
=\ol \varphi_{K,a,i}$ that $T_g\cdot \ol \varphi_{a,i}
=\ol \varphi_{a,i}$, which proves the second statement. 
Write $K^v g K^v=\coprod_{j=1}^m 
      g_j K^v$, where $m=\#(K^v g K^v/K^v)$, then 
\[  {\bf1}_{K^v g K^v}* \ol \varphi_{K,a,i}=\sum_{j=1}^m T_{g_j}(\ol
      \varphi_{a,i})=m \cdot \ol \varphi_{K,a,i}. \]

  (2) Let $f\in M_k(r, K_v, L)^{K^v}$ and $h\in \calH^{\infty v}_{L}$. It
      suffices to check the case where $h$ is of the form ${\bf1}_{K^v g
      K^v}$, because these form an $L$-basis of 
      $\calH^{\infty v}_{L}$. Let $h={\bf1}_{K^v g K^v}$ and write $K^v g
      K^v=\coprod_{j=1}^m {g_j} K^v$, then we compute in $M_k(r,K_v,L)$
      which contains $M_k(r,K_v,L)^{K^v}$:
    \begin{equation}
      \label{eq:HI.2}
      \begin{split}
    h*(\ol \varphi_{a,i} \cdot f)&=\sum_{j=1}^m T_{g_j}(\ol
      \varphi_{a,i})\cdot T_{g_j}(f)=\sum_{j=1}^m \ol
      \varphi^{}_{a,i}\cdot T_{g_j}(f)\\
    &=\ol \varphi_{a,i}\cdot
      \sum_{j=1}^m T_{g_j}(f)=\ol \varphi_{a,i}\cdot (h*f). 
      \end{split}
    \end{equation} 

(3) is proved in the same way as (2). \qed
\end{proof}

One may ask how many Drinfeld modular forms are produced from
coefficient modular forms. Consider the modular forms over $F$
and write $\ol \varphi_{K,a,i}$ for their restriction to the generic
fiber $\ol M^r_K:=\ol\bfM^r_K\otimes_{A_{(v)}}F$. 
Let 
$M^c(r,K,F)\subset M(r,K,F)$ 
be the graded subring generated by all coefficient modular forms 
$\ol \varphi_{K,a,i}$.
As in Lemma~\ref{HI.1}, $T_g (\ol \varphi_{K,a,i})=\ol
\varphi_{K,a,i}$ for all $g\in K(1)$. Then 
\begin{equation}
  \label{eq:IH.coeff}
  M^c(r,K,F)\subset M(r,K,F)^{K(1)}=M(r,K(1),F)
\end{equation}
by Lemma~\ref{MF.35}(3).

Now suppose $A=\Fq[t]$ and $K=K(t)$. By
\cite[Theorem 7.4]{pink}, one has 
\begin{equation}
  \label{eq:HI.A0}
  M(r,K(t), F)=F \otimes_{\Fq} R_r=F[1/v;v\in V^{0}].
\end{equation} 

On the other hand, since $A$ is generated by $t$ over $\Fq$, we have
$M^c(r,K(t),F)=F[\ol\varphi_{K(t),t,1},\dots,
\ol\varphi_{K(t),t,r}]$. Note that the coefficient modular forms $\ol\varphi_{K(t),t,i}$ lie in $F[1/v;v\in V^{0}]$, because
\[
\ol\varphi_t(X) = t\cdot X\cdot \prod_{v\in V^{0}} (1-\tfrac{1}{v}\,X) = tX+\sum_{i=1}^r \ol\varphi_{K(t),t,i} X^{q^i}
\]
by \cite[(7.5)]{pink}. From this, we see that $M^c(r,K(t),F)\neq
M(r,K(t),F)$ at least when $(q,r) \neq (2,1)$, 
because $M^c(r,K(t),F)$ does not
contain all elements of degree one in $M(r,K(t),F)$.
We show that $M^c(r,K(t),F)$ contains
sufficiently many modular forms, in the sense that it generates
  a field of modular functions on $M^r_{K(t)}$ over $F$ that has the same
  transcendence degree over $F$ as the function field of
  $M^r_{K(t)}$.
By \cite[Theorem 8.1 (c) and Remark 8.3] {pink}, one has
$M^c(r,K(t),F)=M(r, K(1), F)=M(r,K(t),F)^{\GL_r(A/t)}$.
It is proven \cite[Theorem 1.7]{pink-schieder} that
$M(r,K(t),F)$ is a normal domain.
It follows that $M(r, K(t),F)$ is
the normalization of $M^c(r,K(t),F)$ in the quotient field ${\rm
  Frac}(M(r,K(t),F))$. 
Thus, $M(r,K(t),F)$ is a finite 
$M^c(r,K(t),F)$-module and hence $M^c(r,K(t),F)$ contains sufficiently many
Drinfeld modular forms. 

\subsection{Generalized Hasse invariants}
\label{sec:HI.2}

Recall that $\grp\subset A$ is the prime ideal corresponding to the
finite place $v$ and $\deg(v)=[\F_v:\F_q]$.
\begin{defn}\label{HI.2}
  For any element $a\in \grp$ and integer $0\le i\le r-1$, define the
  \emph{$i$-th $a$-Hasse invariant} on $\ol \bfM^r_{K}$ over $A_{(v)}$ by
  \begin{equation}
    \label{eq:HI.3}
    H^{a}_i\;:=\;\ol \varphi_{K,a,\,i\, v(a) \deg(v)}\;\in\;H^0(\ol
\bfM^r_{K},\omega_K^{\otimes q^{i\, v(a) \deg(v)}-1}).
  \end{equation}
\end{defn}

As in the previous sections, let $\ol \scrM^r_K:=\ol \bfM^r_{K}\otimes_{A_{(v)}} \F_v$. 
For any $1\le h\le r$ and $a\in \grp$,
let $V(H^a_1,\dots, H^a_{h-1})$ be the closed subscheme of $\ol
\scrM^r_K$ defined as the vanishing locus of $H^a_1,\cdots,H^a_{h-1}$, that is, of the sheaf of ideals $\sum_{i=1}^{h-1}H^a_i\cdot\omega_K^{\otimes 1-q^{i\, v(a) \deg(v)}}\subset\calO_{\ol \scrM^r_K}$.  
Note that $H^a_0=\gamma(a)$ is already zero in $\F_v$.

\begin{lemma}\label{LemmaHasseInv}
Let $(\olE,\ol\varphi)$ be a generalized Drinfeld $A$-module of rank $\le r$ over an $A$-scheme $S$ whose structure morphism $S\to\Spec A$ factors through $\F_v$. Let $a\in A$ with $v(a)=1$ and write $\ol\varphi_a=\sum_{i=0}^{r\cdot\deg(a)}\ol\varphi_{a,i} \tau^i$ with $\ol\varphi_{a,i}\in H^0(S,\olE^{\otimes 1-q^i})$. Then for every $i,j$ with $(i-1)\deg(v)\le j<i\deg(v)$ the coefficient $\ol\varphi_{a,j}$ lies in the subsheaf of $\olE^{\otimes 1-q^j}$ generated by $(H^a_0\cdot \olE^{\otimes 1-q^j},\ldots,H^a_{i-1}\cdot \olE^{\otimes q^{(i-1)\deg v}-q^j})$, where the $H^a_i:=\ol\varphi_{a,i\deg(v)}$ are defined as in \eqref{HI.2}.
\end{lemma}

\begin{proof}
The usual proof also works for generalized Drinfeld modules. First, the statement is local on $S$, so we may assume that $S=\Spec R$ is affine and that there is an isomorphism $\olE\simeq\calO_{\Spec R}$. We use it to view all $\ol\varphi_{a,i}$, and in particular the $H^a_i=\ol \varphi_{K,a,\,i\deg(v)}$ as elements of $R$. The statement is equivalent to showing that for any $j<i\deg(v)$ we have $\ol\varphi_{a,j}=0$ in $\olR_i:=R/(H^a_0,\ldots H^a_{i-1})$. Let $b\in A$ such that the image $\bar b$ of $b$ in $\F_v$ generates the multiplicative group $\F_v^\times$. Then $\bar b^n=1$ in $\F_v$ if and only if $(q^{\deg(v)}-1)|n$. Let $n=\min\{j\colon \ol\varphi_{a,j}\ne0 \text{ in }\olR_i\}$. From 
\begin{eqnarray*}
\ol\varphi_{ab} \enspace =\enspace \ol\varphi_a\ol\varphi_b & = & (\ol\varphi_{a,n}\tau^n+\ldots)(\gamma(\bar b)\tau^0+\ldots) \enspace = \enspace \ol\varphi_{a,n}\cdot\gamma(\bar b)^{q^n}\tau^n+\ldots \\[2mm]
\ol\varphi_{ba} \enspace =\enspace \ol\varphi_b\ol\varphi_a & = & (\gamma(\bar b)\tau^0+\ldots)(\ol\varphi_{a,n}\tau^n+\ldots) \enspace = \enspace \gamma(\bar b)\cdot\ol\varphi_{a,n}\tau^n+\ldots
\end{eqnarray*}
we deduce $(\gamma(\bar b)^{q^n}-\gamma(\bar b))\cdot \ol\varphi_{a,n}=0$ in $\olR_i$. Now write $n=k\deg(v)+m$ with $k\in\bbZ$ and $0\le m<\deg(v)$ and use $\bar b^{q^{\deg(v)}}=\bar b$. If $m\ne0$ then $\bar b^{q^n}-\bar b=\bar b^{(q^{\deg(v)})^k q^m}-\bar b=\bar b^{q^m}-\bar b\in\F_v^\times$, whence $\gamma(\bar b)^{q^n}-\gamma(\bar b)\in R^\times$ and $\ol\varphi_{a,n}=0$. Since $\ol\varphi_{a,k\deg(v)}=H^a_k=0$ in $\olR_i$ for $0\le k<i$, it follows that $n\ge i\deg(v)$ and the lemma is proved.
\qed
\end{proof}

\begin{lemma}\label{HI.3}
  Suppose $v(a)=1$. 

{\rm (1)} For any integer $1\le h\le r$, the closed
  subscheme $V(H^a_1,\dots, H^a_{h-1})$ of $\ol \scrM^r_{K}$ is
  independent from the choice of $a$ (satisfying $v(a)=1$). 

{\rm (2)} For any point $x$ in $\ol \scrM^r_{K}$, the
Drinfeld $A$-module $\ol \varphi_{K,x}$ over the point $x$ 
has height $\ge h$ if and only
if $x\in V(H^a_1,\dots, H^a_{h-1})$.
\end{lemma}
\begin{proof}
  Statement (2) follows immediately from the
  definition of the height of a Drinfeld module $\varphi$, 
  as the exponent of the leading term of
  $\ol \varphi_a=\ol\varphi_{a,h} \tau^h+\cdots$ in the associated 
  formal $A$-module $\ol \varphi$ when
  $a$ is a uniformizer; see \cite[(1.1), p.~529]{strauch}. 

(1) Suppose $\wt a\in A$ is another element
  with $v(\wt a)=1$. Recall that for the universal generalized Drinfeld module $(\olE,\ol\varphi)$ on $M^r_K$ we write $\omega_K:=\olE^{\otimes -1}$. On $\ol \scrM^r_K$, where $H^a_0=0$, we have  
\[ \ol \varphi_a:=\ol \varphi_{K,a}=H^a_1 \tau^{\deg v}+\dots+ H^a_2
  \tau^{2\deg v}+\cdots,\]   
such that for every $j\in\{(i-1)\deg(v),\ldots,i\deg(v)-1\}$ the 
coefficient $\ol\varphi_{K,a,j}$, which is a section of $\omega_K^{\otimes q^j-1}$, 
lies in the submodule generated by
$(H^a_0\cdot\omega_K^{\otimes q^j-1},\ldots,H^a_{i-1}\cdot\omega_K^{\otimes q^j-q^{(i-1)\deg v}})$
by Lemma~\ref{LemmaHasseInv}.
Then $a/\wt a\in A_{(v)}^\times$ and $a/\wt a=c/\wt c$ with $c,\wt c\in
A$ and $v(c)=v(\wt c)=0$. Write $\ol\varphi_{\wt c}=\gamma(\wt c)\tau^0+\ldots$ and likewise for $\ol\varphi_c$. We have 
\begin{alignat*}{2}
  \ol \varphi_{\wt c a} & \;=\;\ol \varphi_{\wt c}\ol \varphi_{a} && \;=\;\gamma(\wt c) H^a_1 \cdot\tau^{\deg v}+\cdots\\   
  \ol \varphi_{c \wt a} & \;=\;\ol \varphi_{c} \ol \varphi_{\wt a} && \;=\;\gamma(c) H^{\wt a}_1\cdot \tau^{\deg v}+\cdots.
\end{alignat*}
From $\wt c a=c \wt a$ and $\gamma(c),\gamma(\wt c)\in \F_v^\times$,
  we get $\gamma(\wt c) H^a_1=\gamma(c) H^{\wt 
  a}_1$ and $V(H^{a}_1)=V(H^{\wt a}_1)$. We now proceed by
  induction. For $1\le i \le h-1$, we have
\begin{equation}
  \label{eq:HI.ac2}
  \begin{split}
  \ol \varphi_{\wt c a} \mod (H^a_0,\dots, H^a_{i-1})&=\gamma(\wt c) 
  H^a_i \cdot\tau^{i \deg v}+\cdots \\   
  \ol \varphi_{c \wt a} \mod (H^{\wt a}_0,\dots, H^{\wt a}_{i-1}) 
  &=\gamma(c) H^{\wt a}_i\cdot \tau^{i \deg v}+\cdots.\\   
  \end{split}
\end{equation}
By $\ol \varphi_{\wt c a}=\ol \varphi_{c \wt a}$ and the induction hypothesis $(H^a_1,\dots, H^a_{i-1})=(H^{\wt a}_1,\dots,
H^{\wt a}_{i-1})$, 
we obtain the equality $(H^a_1,\dots, H^a_{i})=(H^{\wt a}_1,\dots,
H^{\wt a}_{i})$. Therefore, $(\ol \scrM^r_K)^{\ge h}:=V(H^a_1,\dots,
H^a_{h-1})$ is independent of $a$. \qed
\end{proof}

\begin{defn}\label{HI.4}
  For $1\le h \le r$, we define $(\ol \scrM^r_{K})^{\ge h}$ (resp.~ $( \scrM^r_{K})^{\ge h}$) as the closed
  subscheme of $\ol \scrM^r_K$ (resp.~$\scrM^r_K$) defined by the Hasse invariants 
  $H^a_{1},\dots, H^a_{h-1}$ for any $a\in \grp$ with $v(a)=1$. Let 
  $(\ol \scrM^r_{K})^{(h)}:=(\ol \scrM^r_{K})^{\ge h}\setminus(\ol \scrM^r_{K})^{\ge
  h+1}$ and $(\scrM^r_{K})^{(h)}:=(\scrM^r_{K})^{\ge h}\setminus(\scrM^r_{K})^{\ge
  h+1}$ be the locally closed subschemes. 
\end{defn}

In particular, $(\scrM^r_{K})^{(h)}$ equals the $v$-rank stratum $(\scrM^r_{K})_{(r-h)}$ from Section~\ref{sec:SS.1} and $(\ol\scrM^r_K)^{\ge 1}
      =\ol \scrM^r_{K}=\ol\bfM^r_K \otimes_{A[\grn_K^{-1}]} \F_v$.

\begin{thm}\label{HI.5}
  Let $h$ be an integer with $1\le h\le r$ and let $a\in\grp$ with $v(a)=1$.

{\rm (1)} The subschemes 
  $(\ol \scrM^r_{K})^{\ge h}$ and $(\ol \scrM^r_{K})^{(h)}$ are  
  of pure dimension $r-h$ and $(\scrM^r_{K})^{(h)}$ is
   Zariski dense in $(\ol \scrM^r_{K})^{\ge h}$, in $(\scrM^r_{K})^{\ge h}$ and in $(\ol \scrM^r_{K})^{(h)}$.

{\rm (2)} The subschemes $(\ol \scrM^r_{K})^{(h)}$ and
$(\scrM^r_{K})^{(h)}$ are affine.

{\rm (3)} If $\ol \bfM^r_K$ is Cohen-Macaulay then $(H^a_0,\ldots,H^a_{r-1})$ is a regular sequence on $\ol\bfM^r_K$.

{\rm (4)} If $\ol \bfM^r_K$ is Cohen-Macaulay then for every $1\le h\le r$ the closed subscheme $X_h:=V(H^a_1,\ldots,H^a_h)$ of $\ol\bfM^r_K$ is flat over $A_{(v)}$.

{\rm (5)} For every $h< r$, every (geometric) irreducible component of $(\ol \scrM^r_K)^{\ge h}$ contains a (geometric) irreducible component of $(\ol \scrM^r_K)^{\ge h+1}$. Likewise, every (geometric) irreducible component of $(\scrM^r_K)^{\ge h}$ contains a (geometric) irreducible component of $(\scrM^r_K)^{\ge h+1}$.

{\rm (6)} If $\ol \bfM^r_K$ is Cohen-Macaulay, then so is each subscheme 
 $(\ol \scrM^r_{K})^{\ge h}$.

 {\rm (7)} If $\ol \bfM^r_K$ is Cohen-Macaulay and $h< r-1$, then the natural maps $\pi_0((\ol \scrM^r_K)^{\ge h+1})\to\pi_0((\ol \scrM^r_K)^{\ge h})$ and $\pi_0((\ol \scrM^r_K)^{\ge h+1}\otimes_{\F_v} \ol \F_v)\to\pi_0((\ol \scrM^r_K)^{\ge h}\otimes_{\F_v} \ol \F_v)$ of (geometric) connected components are bijective.

{\rm (8)} The moduli space $\ol\scrM^r_{K}=(\ol\scrM^r_{K})^{\ge 1}$ has
$\bigl|(\A^\infty)^\times/(F^\times\cdot \det K)\bigr|$ geometric
connected components 
and $\bigl|(\A^\infty)^\times/(F^\times\cdot \det
K)\bigr|/f_v$ connected components, 
where $f_v$ is the order of the Frobenius element 
$(\grp, F_{\det
  K}/F)\in \Gal(F_{\det K}/F)$ of the class field $F_{\det K}$ of $F$
  corresponding to the open subgroup 
  $F^\times \det K \subset (\A^\infty)^\times$ by class field theory.
If $K=K(\grn)$ for a nonzero proper ideal $\grn\subset A$, then
$\bigl|(\A^\infty)^\times/(F^\times\cdot \det
K)\bigr|/f_v=h(A) \cdot|(A/\grn)^\times| /\bigl((q-1)f_1f_2\bigr)$, where $h(A)$ is the class number of $A$, where $f_1$ is the smallest
positive integer such that $\grp^{f_1}=(b)$ is a principal ideal and
$f_2$ is the order of the image of $b$ in
$(A/\grn)^\times/{\Fq}^\times$. 

{\rm (9)} If $\ol \bfM^r_K$ is Cohen-Macaulay then statement (8) holds also true for $(\ol\scrM^r_{K})^{\ge h}$ for all $h<r$ (and not just $h=1$).
\end{thm}

\begin{remark}
  (1) The closed stratum $(\ol \scrM^r_K)^{\ge r}$ is the
zero dimensional supersingular locus and it has many connected 
components (points) as given by the mass formula \eqref{eq:SS.15}; see Section~\ref{sec:SS.3}.

  (2)  When $A=\Fq[t]$ and $K=K(t)$, Pink and Schieder \cite[Theorem~1.11]{pink-schieder} showed that $\ol \bfM^r_K$ is Cohen-Macaulay; cf. our Proposition~\ref{S.10}. We expect this is also true for arbitrary $\ol \bfM^r_K$; see Conjecture~\ref{ConjCM} and Remark~\ref{RemCM}.
\end{remark}

\begin{proof}[of Theorem~\ref{HI.5}]
  (1) We first show that $(\ol \scrM^r_{K})^{\ge h}$ is of pure dimension
      $r-h$. When $h=1$, consider for a moment the scheme $\ol\bfM^r_K$ over $A[\grn_K^{-1}]$ from Theorem~\ref{S.2} and let $\Spec \wt A:=\Spec A[\grn_K^{-1}a^{-1}]\cup\{\grp\}$. Then the projective variety $(\ol\scrM^r_K)^{\ge 1}
      =\ol \scrM^r_{K}=\ol\bfM^r_K \otimes_{A[\grn_K^{-1}]} \F_v$ is
      of pure codimension $1$ in $\ol\bfM^r_K\otimes_{A[\grn_K^{-1}]}\wt A$ by 
      \cite[IV$_4$, Corollaire~21.12.7]{EGA}, because the inclusion of its 
      complement $\ol\bfM^r_K\otimes_{A[\grn_K^{-1}]}\wt A[1/a]\hookrightarrow \ol\bfM^r_K\otimes_{A[\grn_K^{-1}]}\wt A$ is an affine morphism. Since $\ol\bfM^r_K\otimes_{A[\grn_K^{-1}]}\wt A$ is irreducible of dimension $r$ we conclude that $\ol \scrM^r_{K}$ is pure of dimension $r-1$; use \cite[Corollary~13.4]{eisenbud}.
      Since the subvariety $(\ol \scrM^r_{K})^{\ge h}$ is cut out
      by $h-1$ equations from $\ol \scrM^r_{K}$, every irreducible component has
      dimension $\ge r-h$ by \cite[Theorem~I.7.2]{hartshorne}.  
      As in the proof of Proposition~\ref{2.7} we stratify the scheme $\ol \scrM^r_{K}=\coprod_{1\le r'\le r}
      S_{r'}$ by ranks $r'$, that is, $S_{r'}$ is the locally closed reduced
      subscheme consisting of all points where the universal generalized Drinfeld
      $A$-module has rank $r'$, and hence is a (genuine) Drinfeld module of rank $r'$. 
      By adding a level-$\grn$ structure on $S_{r'}$ for each $r'$, 
      there exist a finite \'etale cover $\wt S_{r'}$ of $S_{r'}$ and 
      a morphism $\wt S_{r'}\to \scrM^{r'}_{K'(\grn)}$ for 
      $K'(\grn)=\ker\bigl(\GL_{r'}(\wh A)\to\GL_{r'}(A/\grn)\bigr)$ 
      induced by the universal property of the moduli scheme 
      $\scrM^{r'}_{K'(\grn)}$.
      The latter morphism is quasi-finite, because $\ol\varphi$ is
      weakly separating on $\ol\scrM^r_K$. Since the stratum
      $(\scrM^{r'}_{K'(\grn)})^{\ge h}$ is of pure dimension $r'-h$ by
      Theorem~\ref{SS.2},
      the stratum $S^{\ge h}_{r'}= (\ol \scrM^r_{K})^{\ge h}\cap S_{r'}$
      has dimension $\le r'-h\le r-h$. Therefore, $\dim 
      (\ol \scrM^r_{K})^{\ge h}=r-h$ and every irreducible component has the same
      dimension. 
      Moreover, the complement of $(\scrM^r_K)^{(h)}$ in $(\ol\scrM^r_K)^{\ge h}$ equals $(\scrM^r_K)^{\ge h+1}\cup\coprod_{1\le r'<r}S^{\ge h}_{r'}$ and is of dimension $<r-h$ by the above. So every irreducible component of $(\ol\scrM^r_K)^{\ge h}$ meets $(\scrM^r_K)^{(h)}$ and the Zariski-density is proved.

  (2) 
  Since the stratum $(\ol \scrM^r_{K})^{(h)}$ 
  is the complement of
      an effective 
      ample divisor defined by $H^a_h=0$ in the projective scheme 
    $V(H^a_1,\dots,
      H^a_{h-1})$, it is affine. Here we use that on the latter scheme
    $\omega_K$ is ample by \cite[II, Proposition~4.6.13 (i bis)]{EGA}. 
      Since $\scrM^r_K$ is affine and $\ol\scrM^r_K$ is separated, the 
    intersection $(\scrM^r_K)^{(h)}=\scrM^r_K \cap (\ol \scrM^r_{K})^{(h)}$ is also affine. 

(3) For any point $x\in (\ol \scrM^r_K)^{\ge h}$, one has $\calO_{(\ol
    \scrM^r_K)^{\ge h}, x}=\calO_{\ol \bfM^r_K,
    x}/(H^a_0,\dots, H^a_h)$. It follows from (1) that for any 
     $0\le i\le r-2$, one has 
\[ \dim \calO_{\ol \bfM^r_K,
    x}/(H^a_0,\dots, H^a_i)=\dim \calO_{\ol \bfM^r_K,
    x}/(H^a_0,\dots, H^a_{i+1})+1. \] 
    Thus by \cite[Theorem~II.8.21A]{hartshorne},  $H^a_0,\dots, H^a_{r-1}$
    form a regular sequence in the local ring $\calO_{\ol \bfM^r_K,x}$.

(4) By \cite[Theorem~18.17(a)]{eisenbud} we must show that $\depth(H^a_0\cdot\O_{X_h,x}, \O_{X_,x}) = \dim A_{(v)} = 1$ for every point $x \in V(H_0^a,\ldots,H_h^a)$. By (3) and \cite[Corollary~17.2]{eisenbud} also the sequence $H_1^a,\ldots,H_h^a,H^a_0$ is a regular sequence in $\O_{\ol\bfM^r_K,x}$. It follows that $H^a_0$ is a non-zero-divisor in $\O_{X_h,x}$ and $\depth(H^a_0\cdot\O_{X_h,x}, \O_{X_,x})= 1$ as desired.

(5) Let $X\subset (\ol \scrM^r_K)^{\ge h}\otimes_{\F_v} \ol \F_v$ be a geometric irreducible component with reduced subscheme structure. Then $\omega_K|_X$ is ample on $X$ by \cite[II, Proposition~4.6.13 (i bis)]{EGA}. If $V_X(H^a_h)=X\cap (\ol \scrM^r_K)^{\ge h+1}\otimes_{\F_v} \ol \F_v=\varnothing$, then $H^a_h$ induces an isomorphism $\calO_X\isoto (\omega_K|_X)^{\otimes q^{h\deg(v)}-1}$. Then $\calO_X$ is ample by \cite[II, Proposition~4.5.6(i)]{EGA} and $X$ is quasi-affine by \cite[II, Proposition~5.1.2]{EGA}. Since $X$ is a projective $\ol\F_v$-scheme it is finite over $\ol\F_v$ by \cite[Corollary~13.82]{GW}. This contradicts that its dimension is $r-h\ge1$. So $X\cap (\ol \scrM^r_K)^{\ge h+1}\otimes_{\F_v} \ol \F_v \ne\varnothing$. Since this is cut out from $X$ by one equation $H^a_h=0$, its dimension equals $\dim(X)-1=r-h-1=\dim(\ol \scrM^r_K)^{\ge h+1}$ by (1) and \cite[Corollary~13.11]{eisenbud}. 
It follows that an irreducible component $X'$ of $(\ol \scrM^r_K)^{\ge h+1}\otimes_{\F_v} \ol \F_v$ is contained in $X$. Since $(\scrM^r_K)^{\ge h+1}\otimes_{\F_v} \ol \F_v$ is dense in $(\ol \scrM^r_K)^{\ge h+1}\otimes_{\F_v} \ol \F_v$ by (1), it follows that the generic point of $X'$ lies in $(\scrM^r_K)^{\ge h+1}\otimes_{\F_v} \ol \F_v\subset X\cap\scrM^r_K\otimes_{\F_v} \ol \F_v$. This proves the statement for $(\scrM^r_K)^{\ge h}$.

(6) Let $x\in(\ol\scrM^r_K)^{\ge h}$. Since $H^a_0,\dots, H^a_{h-1}$ is a regular sequence in $\O_{\ol\bfM^r_K,x}$ by (3) the assertion follows from \cite[Proposition~18.13]{eisenbud}

(7) The maps are surjective by (5). The injectivity follows from Lemma~\ref{LemmaYConnected} below and (6), because $(\ol \scrM^r_K)^{\ge h+1}$ is the subscheme of $(\ol \scrM^r_K)^{\ge h}$ cut out by the generalized Hasse invariant $H^a_h$ and $H^a_h$ is a global section of an ample invertible sheaf. 

(8) Let $A_v$ be the completion of $A_{(v)}$ and let $F_v$ be its fraction field. Let $F'_v$ be a finite extension of the $v$-adic completion of the compositum $\ol\F_v F_v$ such that every connected component of $\ol{\bfM}^r_{K}\otimes_{A_{(v)}} \ol F_v$ is defined over $F_v'$. Since ${\bfM}^r_{K}\subset \ol{\bfM}^r_{K}$ is fiber-wise open and dense and $\ol{\bfM}^r_{K}\otimes_{A_{(v)}} \ol F$ is normal, we get the right equality in the following chain of sets of connected components:
\[
\pi_0(\ol{\bfM}^r_{K}\otimes_{A_{(v)}} \ol \F_v)\simeq \pi_0(\ol{\bfM}^r_{K}\otimes_{A_{(v)}} F'_v)\simeq \pi_0(\ol{\bfM}^r_{K}\otimes_{A_{(v)}} \ol F)=\pi_0(\bfM^r_K\otimes_{A_{(v)}} \ol F).\]
In this chain the first bijection comes from Lemma~\ref{Lemma6.9} below, and the second bijection from \cite[IV$_2$, Proposition~4.5.1]{EGA}.  By Proposition~\ref{Prop2.4}, the set on the right is isomorphic to $(\A^\infty)^\times/(F^\times \det K)$.
Thus, $\ol\scrM^r_{K}\otimes_{\F_v} \ol
\F_v=\ol{\bfM}^r_{K}\otimes_{A_{(v)}} \ol \F_v$ has
$\bigl|(\A^\infty)^\times/(F^\times \det K)\bigr|$ connected
components. 

The same argument also shows that 
the connected components of $\ol\scrM^r_{K}$ are in bijection with
those of the generic fiber $\ol{\bfM}^r_{K}\otimes_{A_{(v)}} F_v$,
which are also the $\Gal(\ol F_v/F_v)$-orbits of $\pi_0(M^r_K
\otimes_F \ol F)$, where we fix an embedding $\Gal(\ol F_v/F_v)\embed
\Gal(\ol F/F)$. Choose a nonzero ideal $\grn\subset A$ such that $K(\grn)\subset K$. The Weil pairing map $M^r_{K(\grn)}\to M^1(\grn)$
induces a $\Gal(\ol F/F)$-equivariant bijection $\pi_0(M^r_{K(\grn)}\otimes_F \ol F)\isoto M^1(\grn)(\ol F)$ (cf.~\eqref{eq:pi0}) and this induces an equivariant bijection $\pi_0(M^r_{K}\otimes_F \ol F)\isoto M^1_{\det K}(\ol F)$. The explicit reciprocity law (\cite[Section~8]{hayes79}, \cite[Section 8]{drinfeld:1}) 
shows that the action of $\Gal(\ol F/F)$ on $M^1(\grn)(\ol F)$ factors through $\Gal(F_\grn/F)$, where $F_\grn$ is the ray class field with modulus $\grn,$  which makes $M^1(\grn)(\ol F)=M^1(\grn)(F_\grn)\simeq (\A^\infty)^\times/(F^\times (1+\grn \wh A)^\times)$ a principal homogeneous space under $\Gal(F_\grn/F)$.
Moreover, the action of the Frobenius element $(\grp,F_\grn/F)\in \Gal(F_\grn/F)$ on $M^1(\grn)(F_\grn)$ is the same as that of the ideal class $[\grp]$ in the ray class group $\Pic_\grn(A):=\calI(\grn)/\calP_\grn\simeq (\A^\infty)^\times/(F^\times (1+\grn \wh A)^\times)$, where $\calI(\grn)$ is the prime-to-$\grn$ ideal group of $F$ and $\calP_\grn$ is the $\grn$-principal ideal subgroup. In particular, as an $F$-scheme, $M^1(\grn)\simeq \Spec F_\grn$. Then $M^1(\grn)\otimes_F F_v=\Spec (F_\grn\otimes_F F_v)$ has $[F_\grn:F]/f_v$ closed points. By Proposition~\ref{Prop2.4}, we have $[F_\grn:F]=\bigl |(\A^\infty)^\times/(F^\times (1+\grn \wh A)^\times) \bigr |=h(A)\cdot|(A/\grn)^\times|/(q-1)$. The number $f_v$ is equal to the order of the class $[\grp]\in \Pic_\grn(A)$. It can be computed by the exact sequences~\eqref{Eq1Pic_n} and \eqref{Eq2Pic_n} which yield
\[ 1 \longrightarrow (A/\grn)^\times/\Fq^{\times} \longrightarrow \Pic_\grn(A) \longrightarrow \Pic(A) \longrightarrow 1. \]  
Namely, $f_v=f'_1 f_2'$, where $f_1'$ is the smallest positive integer such that $[\grp]^{f'_1}$ lies in the subgroup $(A/\grn)^\times/\Fq^{\times}$ and $f_2'$ is the order of $[\grp]^{f_1'}$ in the subgroup $(A/\grn)^\times/\Fq^{\times}$.  Clearly $f_1'=f_1$ and $f_2'=f_2$, and hence $f_v=f_1 f_2$. This proves the last statement.
Using the $\Gal(\ol F/F)$-equivariant surjective map $M^1(\grn)(\ol F) \to M^1_{\det K}(\ol F)\simeq  (\A^\infty)^\times/(F^\times \det K)$, we obtain $M^1_{\det K}=\Spec F_{\det K}$. It then follows that
$\pi_0(M^1_{\det K}\otimes_F F_v)=\Spec (F_{\det K} \otimes_F F_v)$ has
$[F_{\det K}:F]/f_v$ elements, where $f_v$ is the residue class degree of $v$ in $F_{\det K}$, which is the same as the order of the Frobenius element $(\grp,F_{\det K}/F)$. This proves the second statement and (8).

(9) This follows immediately from (7) and (8).
  \qed
\end{proof}

\begin{lemma}\label{LemmaYConnected}
 If $X$ is a connected projective Cohen-Macaulay scheme over a field $k$ of pure dimension $\ge 2$, and $Y$ is a closed subset which
    is the support of an effective ample divisor, then $Y$ is
    connected. 
\end{lemma}

\begin{proof}
Since the support does not change when we replace an effective divisor by a power of it, we may assume that $Y$ is the support of a very ample divisor $D$. Let $\calO(1)$ be the corresponding very ample invertible sheaf. For each $q>0$, let $Y_q$ be the closed subscheme supported on $Y$ corresponding to the divisor $qD$. Then we have an exact sequence
\[ 0\to \calO_X(-q) \to \calO_X \to \calO_{Y_q} \to 0. \]
Taking cohomology we have an exact sequence
\[ H^0(X,\calO_X) \xrightarrow{\alpha} H^0(Y, \calO_{Y_d}) \to H^1(X, \calO_X(-q)). \]
As $X$ is Cohen-Macaulay and equi-dimensional
$H^i(X,\calO_X(-q))=0$ for $i<\dim X$ and $q \gg 0$, by \cite[Chap. III, Theorem~7.6(b)]{hartshorne}. Note that the assumption of loc.\ cit.\ that $k$ is algebraically closed is not needed, because cohomology commutes with the flat base change from $k$ to an algebraic closure by \cite[I$_{\rm new}$, Proposition~9.3.2]{EGA}. Thus, for $q\gg 0$, we have $H^1(X, \calO_X(-q))=0$ and the map $\alpha$ is surjective. 
But $H^0(X,\calO_X)$ is a finite local $k$-algebra as $X$ is connected, and $H^0(Y, \calO_{Y_q})$ contains $k$, so we conclude that  $H^0(Y, \calO_{Y_d})$ is also a finite local $k$-algebra. Therefore, $Y$ is connected.
\qed
\end{proof}

\begin{lemma}\label{Lemma6.9}
Let $X$ be a scheme over a ring $R$.
\begin{enumerate}
\item \label{Lemma6.9_A}
If $R$ is an integral domain with fraction field $K$ and $X$ is normal and flat over $R$, then there is a canonical bijection between the connected components of $X$ and of $X\times_R K$, which is given by sending $Y$ to $Y\times_R K$.
\item \label{Lemma6.9_B}
If $R$ is noetherian, henselian and local with residue field $k$, and $X$ is proper over $R$, then there is a canonical bijection between the connected components of $X$ and of $X\times_R k$.
\end{enumerate}
\end{lemma}

\begin{proof}
(\ref{Lemma6.9_A})
Let $Y_K$ be a connected component of $X_K:=X\times_R K$, and let $e\in\Gamma(X_K,\O_{X_K})$ be the element which is identically $1$ on $Y_K$ and identically $0$ outside $Y_K$. For every point $x\in X$ the elements of $R\setminus\{0\}$ are non-zero-divisors in $\O_{X,x}$ by the flatness of $X$ over $R$. So $e$, which lies in $\O_{X,x}\otimes_R K=(R\setminus\{0\})^{-1}\O_{X,x}$, lies in the total ring of fractions of $\O_{X,x}$. Since $e$ is an idempotent, that is $e^2-e=0$, and $\O_{X,x}$ is normal, we obtain $e\in\O_{X,x}$, and hence $e\in \Gamma(X,\O_X)$. We claim that the open and closed subset $Y$ of $X$ on which $e$ is invertible is a connected component. Indeed, if $Y$ was the disjoint union of two non-empty open sets $U_1$ and $U_2$, their intersection with $Y_K$ would cover $Y_K$ which is connected. So one of them, say $U_1$ has empty intersection with $Y_K$, that is $U_1\otimes_R K=\varnothing$. If $x\in U_1$ then $(0)=\O_{X,x}\otimes_RK=(R\setminus\{0\})^{-1}\O_{X,x}$, and hence there is an element of $R\setminus\{0\}$ which is a zero-divisor on $\O_{X,x}$ in contradiction to the flatness of $X$ over $R$. The argument also shows that every connected component of $X$ meets $X_K$. This establishes the bijection (\ref{Lemma6.9_A}).

\medskip\noindent
(\ref{Lemma6.9_B}) follows from the lifting of idempotents in form of \cite[IV$_4$, Proposition~18.5.19]{EGA}. \qed
\end{proof}

\begin{prop}\label{PropIgusa}
Recall that $(\scrM^r_{K})^{(h)}$ equals the $v$-rank stratum $(\scrM^{r}_{K})_{(r-h)}$ from Section~\ref{sec:SS.1}. Let $h<r$ and let $(E,\varphi, \bar \eta)$ be the universal family over $(\scrM^r_{K})^{(h)}$. For every $m\ge 0$, let
\[
\Ig^{(h)}_{m}\;:=\;\Isom_{(\scrM^r_{K})^{(h)}}((\grp^{-m}/A)^{r-h}, \varphi[\grp^m]^{\rm et})
\]
be the Igusa cover of level $m$ of $(\scrM^r_{K})^{(h)}$, where $\varphi[\grp^m]^{\rm et}$ is the \'etale part of $\varphi[\grp^m]$. Then the natural map $\pi:\Ig^{(h)}_{m}\to (\scrM^r_{K})^{(h)}$ induces bijections $\pi_0(\Ig^{(h)}_{m})\simeq \pi_0((\scrM^r_{K})^{(h)})$ and $\pi_0(\Ig^{(h)}_{m}\otimes_{\F_v} \ol \F_v)\simeq \pi_0((\scrM^r_{K})^{(h)}\otimes_{\F_v} \ol \F_v)$.
\end{prop}

\begin{proof}
Let $\F=\F_v$ or $\F=\ol \F_v$. It suffices to show that for every connected component $S\in \pi_0((\scrM^r_{K})^{(h)}\otimes_{\F_v} \F)$, the cover $\pi^{-1}(S)$ over $S$ is connected. Let $\bar s$ be a geometric point of $S$. The action of the fundamental group $\pi_1(S,\bar s)$ on the fiber $\pi^{-1}(\bar s)$ gives a global monodromy $\rho_S: \pi_1(S,\bar s)\to \GL_{r-h}(A/\grp^m)$. By Theorem~\ref{HI.5}(5), every connected component $S$ contains in its closure $\olS\subset\scrM^r_K\otimes_{\F_v} \F$ a supersingular point $x\in \olS$. By the analog of the Serre-Tate theorem for Drinfeld $A$-modules, the completed local ring $\wh \calO_{\olS,x}$ is the universal deformation ring of the one-dimensional formal $A_v$-module attached to the supersingular Drinfeld module $(E_x, \varphi_x)$ over the point $x$. Let $\Spec \wh \calO_{\olS,x}^{(h)}$ be the stratum in $\Spec \wh \calO_{\olS,x}$ where the height equals $h$, and hence the $v$-rank equals $r-h$. Then $\Spec \wh \calO_{\olS,x}^{(h)}$ is irreducible; see for example \cite{strauch}. Let $s'$ and $\bar s'$ be the generic point and geometric generic point of $\Spec \wh \calO_{\olS,x}^{(h)}$, and denote their residue fields by $k(s')$ and $k(\bar s')$, respectively. Then $s'$ and $\bar s'$ map into $S$. We may change the initial geometric base point $\bar s$ of $S$ and assume that $\bar s=\bar s'$. Then the action of the Galois group $\Gal(k(\bar s')/k(s'))$ on the fiber $\pi^{-1}(\bar s')$ gives a local monodromy $\rho_x:\Gal(k(\bar s')/k(s'))\to \GL_{r-h}(A/\grp^m)$ and it factors through the global monodromy $\rho_S$:
\[ \rho_x: \Gal(k(\bar s')/k(s'))\longrightarrow \pi_1(S,\bar s') \longrightarrow \GL_{r-h}(A/\grp^m). \]
By \cite[Theorem 2.1]{strauch}, the local monodromy $\rho_x$ is surjective. It follows that the global monodromy $\rho_S$ is surjective and that $\pi^{-1}(S)$ is connected. \qed
\end{proof}

\begin{remark}
Very recently Fukaya, Kato and Sharifi \cite{fukaya-kato-sharifi} have
constructed toroidal (smooth) compactifications of the Drinfeld moduli
scheme $\bfM_K^r$ over $A$ for $A=\Fq[t]$ and $K=K(\grn)$ with a
nonzero ideal $\grn\subset A$. This leads to the following results.  
\end{remark}

\begin{prop}\label{CompSpFiber}
Let $A=\Fq[t]$, $K=K(\grn)$ with a nonzero ideal $\grn\subset A$, $\grp\nmid \grn$ a prime ideal of $A$ with corresponding place $v$. 
  
{\rm (1)} The moduli space $\scrM^r_{K}=\bfM^r_{K}\otimes_{A_{(v)}}\otimes \F_v$ has $\bigl|(A/\grn)^\times/\F_q^\times\bigr|$ geometric connected components.

{\rm (2)} The ordinary locus $\scrM^{r, \rm ord}_{K}:=(\scrM^{r}_{K})^{(1)}$ of $\scrM^{r}_{K}$ and its Igusa cover $\Ig^{(1)}_{m}$ have $\bigl|(A/\grn)^\times/\F_q^\times\bigr|$ geometric connected components.
\end{prop}

\begin{proof}
(1) Let $A_v$ be the completion of $A_{(v)}$ and let $F_v$ be its fraction field. By \cite{fukaya-kato-sharifi}, there is a proper smooth compactification $\ol{\bfM}^r_{K, \Sigma}$ of ${\bfM}^r_{K}$ over $A_{(v)}$. Let $F'_v$ be a finite extension of the $v$-adic completion of the compositum $\ol\F_v F_v$ such that every connected component of $\ol{\bfM}^r_{K, \Sigma}\otimes_{A_{(v)}} \ol F_v$ is defined over $F_v'$. By Lemma~\ref{Lemma6.9} there is the left bijection in $\pi_0(\ol{\bfM}^r_{K, \Sigma}\otimes_{A_{(v)}} \ol \F_v)\simeq \pi_0(\ol{\bfM}^r_{K, \Sigma}\otimes_{A_{(v)}} F'_v)\simeq \pi_0(\ol{\bfM}^r_{K, \Sigma}\otimes_{A_{(v)}} \ol F)$, and the right bijection is \cite[IV$_2$, Proposition~4.5.1]{EGA}. Since ${\bfM}^r_{K}\subset \ol{\bfM}^r_{K, \Sigma}$ is fiber-wise open and dense and $\ol{\bfM}^r_{K, \Sigma}$ is smooth over $A_{(v)}$, we get bijections:
\[
\pi_0(\scrM^r_{K})=\pi_0(\ol{\bfM}^r_{K, \Sigma}\otimes_{A_{(v)}} \ol \F_v)\simeq \pi_0(\ol{\bfM}^r_{K, \Sigma}\otimes_{A_{(v)}} \ol F)=\pi_0(\bfM^r_K\otimes_{A_{(v)}} \ol F).
\]
By Proposition~\ref{Prop2.4}, the latter set is isomorphic to $(\A^\infty)^\times/(F^\times (1+\grn \wh A))$, which is isomorphic to $(A/\grn)^\times/\Fq^\times$ because $A=\Fq[t]$ is a principal ideal domain. 
Thus, $\scrM^r_{K}$ has $|(A/\grn)^\times/\F_q^\times|$ connected components.

(2) follows from (1) and Proposition~\ref{PropIgusa}. \qed
\end{proof}

\subsection{Hecke eigensystems of Drinfeld modular forms modulo $v$}
\label{sec:HI.3}

In this final section we determine the Hecke eigensystems arising from $M_k(r, K_v, \ol \F_v)^{K^v}$, see Definition~\ref{MF.7}. Recall the group $G'$ over $F$ from Section~\ref{sec:SS.3} and the isomorphism $G(\A^{\infty v})\simeq G'(\A^{\infty v})$ from \eqref{eq:SS.id}. Let 
\[
\calA(G',\ol\F_v)\;:=\;\bigl\{\,f:G'(F)\backslash G'(\A)/G'(F_\infty)\longrightarrow \ol\F_v\quad\text{locally constant functions}\,\bigr\}
\]
and recall the group $U(v)$ from \eqref{eq:SS.8}.
For brevity we write 
  $\calH^{\infty v\grn}_{\ol \F_v}=\calH_{\ol \F_v}(G(\A^{\infty
  v\grn}),K^{v\grn})\simeq\calH_{\ol \F_v}(G'(\A^{\infty
  v\grn}),K^{v\grn})$ 
  for the prime-to-$v\grn$ spherical Hecke algebra over $\ol \F_v$. 
  We consider the smooth admissible $G(\A^{\infty v})$-module $M_k(r,K_v,\ol\F_v)$ from Theorem~\ref{ThmSmoothHeckeAction}.

\begin{thm}\label{HI.6}
Let $K_v=K_v(1)=G(A_v)$ and let $K^v\subset G(\A^{\infty v})$ be an open 
compact subgroup. Let $\grn\subset A$ be a 
non-zero ideal, prime to $v$ such that $K(\grn)$ is contained in 
a conjugate of $K_vK^v$. 
Consider the sets of prime-to-$v\grn$ Hecke eigensystems
$\calH^{\infty v\grn}_{\ol\F_v}\to\ol\F_v$ arising from
  \begin{enumerate}
  \item algebraic Drinfeld
  modular forms in $M_k(r, K_v, \ol \F_v)^{K^v}$ for all $k\ge 0$, and 
  \item  elements of $\calA(G',\ol \F_v)^{U(v)K^v}$, respectively. 
  \end{enumerate}
If Conjecture~\ref{ConjCM} holds for $K=K_v\wt K^v$  for a cofinal system of compact open subgroups $\wt K^v\subset G(\A^{\infty v})$, then the two sets of Hecke eigensystems are equal. In
  particular, there are then only finitely many Hecke eigensystems of
  algebraic Drinfeld modular forms over $\ol \F_v$ of a fixed level
  and all weights. 
\end{thm}
\begin{proof}
Let $K=K_v\wt K^v$ belong to the cofinal system for which $\ol \bfM^r_K$ is Cohen-Macaulay. Consider the inclusion maps $i_h:(\ol
  \scrM^{r}_{K,\ol\F_v})^{\ge h}:=V(H_1^a,\ldots,H_{h-1}^a)\longrightarrow \ol \scrM^{r}_{K,\ol\F_v}:=\ol \scrM^{r}_{K}\otimes_{\F_v}\ol\F_v$.
  Since the Hasse invariant $H^a_h$ is a non-zero divisor on $(\ol\scrM^{r}_{K,\ol\F_v})^{\ge h}$ by Theorem~\ref{HI.5}(3), the multiplication by $H^a_{h}$ gives a short
    exact sequence of coherent sheaves on $\ol \scrM^{r}_{K,\ol\F_v}$.
  \begin{equation*}
    \label{eq:HI.4}
    0\longrightarrow i_{h,*}\,i_h^* \omega_K^{\otimes (k-q^{h\deg(v)}+1)}\otimes\ol\F_v\xrightarrow{\,H^a_h\,} i_{h,*}\,i_h^*
    \omega_K^{\otimes k}\otimes\ol\F_v\longrightarrow i_{h+1,*}\,i_{h+1}^*
    \omega_K^{\otimes k}\otimes\ol\F_v\longrightarrow 0. 
  \end{equation*}
This gives an exact sequence of global sections
\begin{eqnarray}
\nonumber
    0 \;\longrightarrow\; H^0\bigl((\ol \scrM^{r}_{K,\ol\F_v})^{\ge h},
          i_h^* \omega_K^{\otimes (k-q^{h\deg(v)}+1)}\otimes\ol\F_v\bigr)
\;\xrightarrow{\enspace H^a_{h}\;}\;
  H^0\bigl((\ol\scrM^{r}_{K,\ol\F_v})^{\ge h},i_h^* \omega_K^{\otimes k}\otimes\ol\F_v\bigr) \\
  \label{eq:HI.5-0}
\;  \xrightarrow{\quad\bfr\enspace} \;
   H^0\bigl((\ol \scrM^{r}_{K,\ol\F_v})^{\ge h+1},i_{h+1}^* \omega_K^{\otimes k}\otimes\ol\F_v \bigr),
\end{eqnarray}
where $\bfr$ is the restriction map onto $(\ol \scrM^{r}_{K,\ol\F_v})^{\ge h+1}$.
For $1\le h\le r$ and all $k\in\bbZ$ we define 
\begin{eqnarray*}
V(h,k) & := & \Bigl(\dirlim[\wt K^v]H^0\bigl((\ol\scrM^{r}_{K_v\wt K^v,\ol\F_v})^{\ge h},i_h^* \omega_{K_v\wt K^v}^{\otimes k}\otimes\ol\F_v\bigr) \Bigr)^{K^v} ,
\end{eqnarray*}
where $\wt K^v$ runs through the cofinal system for which $\ol \bfM^r_{K_v\wt K^v}$ is Cohen-Macaulay, and where $K^v$ is the subgroup which was fixed in the theorem. Note that $V(1,k)=(0)$ when $k<0$ by Remark~\ref{Remk<0}, but we do not know whether $V(h,k)=(0)$ for $k<0$ and $1<h<r$, because $(\ol\scrM^{r}_{K_v\wt K^v,\ol\F_v})^{\ge h}$ might not be reduced and then Lemma~\ref{Lemma_DualOfAmple} cannot be applied. For $h=r$ we have $V(r,k)\simeq V(r, k+q^{r\deg v}-1)$ by the periodicity property \eqref{eq:SS.13} for all $k\in\bbZ$.

Since taking the inductive limit is an exact functor and taking $K^v$-invariants is left exact with $H^a_h$ fixed under $K^v$ by Lemma~\ref{HI.1}(1), sequence \eqref{eq:HI.5-0} yields an exact sequence of $\calH^{\infty v\grn}_{\ol\F_v}$-modules
\begin{eqnarray}
  \label{eq:HI.5}
  \begin{CD}
    0 @>>> V(h,k-q^{h\deg(v)}+1) @>H^a_{h}>> 
   V(h,k) @>{\bfr}>> V(h+1,k)\,.
  \end{CD}
\end{eqnarray}
For all $1\le h \le r$ and $k\in\bbZ$ let $H(h,k)\subset \Hom_{\ol\F_v}(\calH^{\infty v\grn}_{\ol \F_v}  , \ol \F_v)$ be  
  the subsets of all prime-to-$v\grn$ Hecke eigensystems arising from 
  the Hecke module $V(h,k)$.

  When $h=1$, $(\ol \scrM^{r}_{K,\ol\F_v})^{\ge h}=\ol \scrM^{r}_{K,\ol\F_v}$ and the union 
  $\bigcup_{k\ge 0} H(1,k)$ is the set of all prime-to-$v\grn$ Hecke
  eigensystems arising from the Hecke modules $V(1,k)=M_k(r,K_v, \ol
  \F_v)^{K^v}$ for all $k\ge 0$. Moreover, when $k<0$ then $V(1,k)=(0)$  by Remark~\ref{Remk<0}, and 
  hence $H(1,k)=\varnothing$.

  On the other hand, when $h=r$, $(\ol \scrM^{r}_{K,\ol\F_v})^{\ge r}$ equals the supersingular set $\scrS_{K,\ol\F_v}:=\scrS_K\otimes_{\F_v}\ol\F_v$ from Section~\ref{sec:SS.3}, which is contained in $\scrM^r_K$.
  Therefore, $V(r,k)= S_k(r,K_v K^v, \ol\F_v)$ by \eqref{eq:SkFixPts}, and $H(r,k)$ equals the set 
  of Hecke eigensystems of the supersingular 
  Hecke modules $S_k(r,K_v K^v, \ol
  \F_v)$ for all $k\in\Z$, which we studied in Proposition~\ref{SS.7}. 
  To prove the theorem we next show

\def\labelenumi{\theenumi} 
\def\theenumi{(\alph{enumi})} 
  \begin{enumerate}
  \item\label{HI.6_A} for any integer $j$, one has $\bigcup_{k\le j}
   H(h,k)\subset \bigcup_{k\le j}  
   H(h+1,k)$ for all $1 \le h \le r-1$,
  \item\label{HI.6_B} there is a positive integer $k_0$ such that 
  $H(r,k)\subset H(1,k)$ for all $k\ge k_0$,
  \end{enumerate}

  \medskip\noindent
\ref{HI.6_A} Let $0\ne f\in V(h,k)$ be a prime-to-$v\grn$ Hecke
  eigenform defined on $(\ol\scrM^r_{K_v\wt K^v,\ol\F_v})^{\ge h}$ for some $\wt K^v$. If $\bfr(f)\neq 0$ in sequence \eqref{eq:HI.5}, that is in sequence~\eqref{eq:HI.5-0}, then $\bfr(f)$ gives rise to the same Hecke eigensystem as $f$. Otherwise, $f$ is divisible by $H^a_h$.

We show that $f$ cannot be arbitrarily often divisible by $H^a_h$. Namely, let $x$ be a point in $V(H^a_h)\subset(\ol\scrM^r_{K_v\wt K^v,\ol\F_v})^{\ge h}$. The Krull intersection theorem \cite[Corollary~5.4]{eisenbud} for an affine open neighborhood of $x$ on which $\omega_K$ is trivial shows that $f$ can only be arbitrarily often divisible by $H^a_h$, if $f$ is zero in an open neighborhood $U$ of $V(H^a_h)$. This neighborhood $U$ intersects each irreducible component of $(\ol\scrM^r_{K_v\wt K^v,\ol\F_v})^{\ge h}$ non trivially by Theorem~\ref{HI.5}(5). So the complement of $U$ has codimension at least $1$. Now consider an arbitrary point $x\in (\ol\scrM^r_{K_v\wt K^v,\ol\F_v})^{\ge h}$ and write $R$ for the local ring of $(\ol\scrM^r_{K_v\wt K^v,\ol\F_v})^{\ge h}$ at $x$. Consider the ideal $I=\{a\in R\colon af=0\}\subset R$. The vanishing locus of $I$ is contained in the complement of $U$ in $\Spec R$, and hence has codimension at least $1$. Since $R$ is Cohen-Macaulay, the depth of $I$ is at least $1$ by \cite[Theorem~18.7]{eisenbud}. This means that $I$ contains a nonzerodivisor $a$ of $R$. Then $af=0$ implies $f=0$ in $R$. Since this holds at every point $x$ of $(\ol\scrM^r_{K_v\wt K^v,\ol\F_v})^{\ge h}$ we conclude that $f=0$ on all of $(\ol\scrM^r_{K_v\wt K^v,\ol\F_v})^{\ge h}$, which is a contradiction.

Therefore, $f$ cannot be arbitrarily often divisible by $H^a_h$ and there is an integer $s$ and
  an element $0\ne f'\in V(h,k-s(q^{h\deg(v)}-1))$ such that
  $(H^a_{h})^s \cdot f'=f$ and $\bfr(f')\neq 0$. Since the
  multiplication by $H^a_h$ is Hecke equivariant 
  by Lemma~\ref{HI.1}(2), the form $f'$ and hence
  $\bfr(f')$ give rise to the same Hecke eigensystem as $f$. This
  proves \ref{HI.6_A}, as the eigenform $\bfr(f')$ lies in $V(h+1,k')$ for $k':=k-s(q^{h\deg(v)}-1)\le k$.

\medskip\noindent
\ref{HI.6_B} Let $K=K_vK^v$. Since $\omega_K$ is ample, \cite[Proposition~III.5.3]{hartshorne} yields a positive integer $k_0$ such
that for any integer $k\ge k_0$ and any $1\le h<r$ we have $H^1\bigl((\ol \scrM^{r}_{K,\ol\F_v})^{\ge h},
          i_h^* \omega_K^{\otimes (k-q^{h\deg(v)}+1)}\otimes\ol\F_v\bigr)=0$. Therefore, the restriction maps $\bfr$ in sequence~\eqref{eq:HI.5-0} are surjective for every $1\le h<r$, and hence their composition
\begin{equation}
  \bfr: H^0(\ol \scrM^{r}_{K,\ol\F_v},\omega_K^{\otimes k}\otimes\ol\F_v)\to H^0(\scrS_{K,\ol\F_v},
  i_r^*\omega_K^{\otimes k}\otimes\ol\F_v)  =V(r,k)
\end{equation}
is likewise surjective. Since $H^0(\ol \scrM^{r}_{K,\ol\F_v},\omega_K^{\otimes k}\otimes\ol\F_v)\subset V(1,k)$ by Lemma~\ref{MF.35}(2), it follows that the map of $\calH^{\infty v\grn}_{\ol \F_v}$-modules $\bfr\colon V(1,k)\to V(r,k)$ is surjective for every $k\ge k_0$.
Since $\calH^{\infty v\grn}_{\ol \F_v}$ is commutative, both $\calH^{\infty v\grn}_{\ol \F_v}$-modules decompose as the direct sums of their
  common generalized $\calH^{\infty v\grn}_{\ol \F_v}$-eigenspaces 
  $V(1,k)=\bigoplus_\chi V(1,k)_\chi$ and $V(r,k)=\bigoplus_\chi V(r,k)_\chi$,
  respectively. Moreover, $\bfr(V(1,k)_\chi)=V(r,k)_{\chi}$.
  In particular, if $\chi=(a_{v'})_{v'}$ is the Hecke eigensystem of an eigenform
  $f\in V(r,k)$, then $V(1,k)_\chi\neq 0$ and there is an
  eigenform $f\in V(1,k)_\chi$ with Hecke eigensystem
  $\chi$. This proves \ref{HI.6_B}.

It follows from the periodicity property \eqref{eq:SS.13}: 
$H(r,k)=H(r, k+q^{r\deg v}-1)$ that 
\begin{equation}
  \label{eq:HI.7}
 \bigcup_{k\ge k_0} H(r,k)=\bigcup_{k\in\Z} H(r,k)=\bigcup_{1\le k\le q^{r\deg v}-1} H(r,k). 
\end{equation}
Combining \ref{HI.6_A} and \ref{HI.6_B}, we prove 
that $\bigcup_{k\ge 0} H(1,k)$ is the same
set of prime-to-$v\grn$ Hecke eigensystems as that 
arising from $S_k(r, K_vK^v,\ol \F_v)$ 
for $k=1,\dots, q^{r\deg v}-1$. The theorem then follows from
\eqref{eq:SS.12}. Note that the vector space $\calA(G',\ol
\F_v)^{U(v)K^v}$ has finite dimension given by \eqref{eq:SS.16} in Lemma~\ref{SS.9}, so we have 
the finiteness of the Hecke eigensystems. \qed
\end{proof}

\begin{cor}
Assume that Conjecture~\ref{ConjCM} holds for $K=K_v\wt K^v$  for a cofinal system of compact open subgroups $\wt K^v\subset G(\A^{\infty v})$.  Let $\grn$ be a non-zero ideal of $A$ with $v\nmid \grn$, and
  $N(r,\grn,v)$ the number of prime-to-$\grn v$ Hecke
  eigensystems arising from $M_k(r, K_v,\ol \F_v)^{K(\grn)^v}$ for all $k\ge
  1$. Then with $q_v=q^{\deg v}$ we have
  \begin{equation}
    \label{eq:HI.12}
    \begin{split}
    N(r,\grn, v)& \le \dim \calA(G',\ol \F_v)^{U(v)K(\grn)^v} \\
    &=\#\GL_r(A/\grn)\cdot h(A) \frac{q_v^r-1}{q-1} 
    \prod_{i=1}^{r-1} \zeta_F^{\infty v}(-i).       
    \end{split}
  \end{equation}
\end{cor}
\begin{proof}
  This follows from Theorem~\ref{HI.6} and the dimension formula
  \eqref{eq:SS.16}. \qed
\end{proof}

Put $\zeta_A(s):=\zeta_F^\infty(s)=\zeta_F(s) (1-q_\infty^{-s})$ and 
\[ c(r,A,\grn):=\#\GL_r(A/\grn)\cdot \frac{h(A)}{q-1}
    \prod_{i=1}^{r-1} |\zeta_{A}(-i)|. \]
Then $|\zeta_{F}^{\infty v}(-i)|=|\zeta_A(-i)|(q_v^i-1)$ and 
$N(r,\grn,v)\le c(r,A,\grn) \prod_{i=1}^r (q_v^i-1)$. Thus, we
obtain the asymptotic behavior for $N(r,\grn, v)$ when $v$ varies:
\begin{equation}
  \label{eq:HI.8}
  N(r,\grn,v)=O(q_v^{r(r+1)/2}) \quad \text{as $q_v\to +\infty$.}
\end{equation}

\begin{remark}\label{HI.8}
  Theorem~\ref{HI.6} is the function field analogue of a theorem of 
  Serre \cite{serre:quat} which describes elliptic modular forms 
  modulo $p$ by quaternion algebras. In \cite{ghitza:thesis} Ghitza
  generalized Serre's theorem to Siegel modular forms (mod $p$).  
  Ghitza followed Serre's idea by restricting 
  modular forms (mod $p$) to the superspecial locus, but he also 
  gives an argument which applies the Kodaira-Spencer map, 
  due to the lack of generalized Hasse invariants.
  Instead of using the Kodaira-Spencer map argument, we use
  generalized Hasse invariants; this is more direct and 
  also close to Serre's original
  proof. Our proof also shows that the prime-to-$v\grn$ 
  Hecke eigensystems arising from
  every intermediate stratum are the same. 
We remark that the construction of generalized Hasse invariants for
Shimura varieties of PEL-type is
known recently due to the work of Boxer \cite{boxer} 
and of Goldring and Koskivirta \cite{goldring-koskivirta:19}. 
\end{remark}

\section*{Acknowledgements}
The article was prepared during the second named
author's visits at M\"unster University. He wishes to thank the
institution for warm hospitality, generous support and 
excellent working conditions. 
Our intelligent debt to Richard Pink is obvious.
It is a great pleasure to thank Richard Pink and Simon Schieder 
for their inspiring articles \cite{pink, pink-schieder} 
on which the present article is partially based. 
We also thank Jo\~{a}o Louren\c{c}o, Richard Pink, David Rydh and
Zhiwei Yun for helpful discussions. We are grateful to the referee for
careful reading and helpful comments which eliminate a few
mistakes and inaccuracies of this paper.      

Both authors acknowledge the support of the DAAD-MoST-PPP ``G-Shtukas:
Geometry and Arithmetic''.  
Furthermore, both authors are grateful for the support of the DFG
(German Research Foundation) in form of Project-ID 427320536 -- SFB
1442 and Germany's Excellence Strategy EXC 2044--390685587
``Mathematics M\"unster: Dynamics--Geometry--Structure''. The second
named author was partially support by the MoST grants
107-2115-M-001-001-MY2 and 109-2115-M-001-002-MY3.

\end{document}